\documentclass[a4paper,11pt]{amsart}
\usepackage[english]{babel}
\usepackage[T1]{fontenc}
\usepackage[all]{xy}
\usepackage{etex}
\usepackage{pifont}
\usepackage{rotating}
\usepackage{pdfpages}
\usepackage{pdfsync}
\usepackage{hyperref}
\hypersetup{
    colorlinks,%
    citecolor=black,%
    filecolor=black ,%
    linkcolor=black,%
    urlcolor=black
}

\usepackage{appendix}
\usepackage{graphicx}
\usepackage{amssymb,amsmath,amsthm,amscd,mathrsfs,float,mathdots}
\usepackage[right=4cm, left=4cm, top=3.5cm, bottom=4.5cm]{geometry}

\addtocounter{section}{0}             
\numberwithin{equation}{section}       

\addtocounter{section}{0}             
\numberwithin{equation}{section}
\theoremstyle{plain} 

\numberwithin{equation}{section}
\newtheorem{theorem}[equation]{Theorem}

\newtheorem{lemma}[equation]{Lemma}

\newtheorem{proposition}[equation]{Proposition}
\newtheorem{conjecture}[equation]{Conjecture}
\newtheorem*{thma}{Theorem A}
\newtheorem*{thmb}{Theorem B}
\newtheorem*{thmc}{Theorem C}
\newtheorem*{conjed}{Conjecture D}
\newtheorem*{conjee}{Conjecture E}
\newtheorem*{conjef}{Conjecture F}
\newtheorem*{thmg}{Theorem G}

\newtheorem{exemple}[equation]{Example}

\newcommand{\Z}{\mathbb{Z}}
\newcommand{\Q}{\mathbb{Q}}

\newcommand{\K}{\mathbf{k}}
\newcommand{\F}{\mathbb{F}}
\newcommand{\GL}{\mathbf{GL}}
\newcommand{\G}{\mathbb{G}}

\newcommand{\Vect}{\mathrm{Vect}}
\newcommand{\IC}{\mathrm{IC}}
\newcommand{\R}{\mathrm{R}}
\newcommand{\g}{\mathfrak{g}}
\newcommand{\B}{\mathcal{B}}

\newcommand{\borel}{\mathfrak{b}}

\newcommand{\Obb}{\mathbb{O}}
\newcommand{\Hche}{\check{H}}
\newcommand{\Gche}{\check{G}}
\newcommand{\m}{\mathfrak{m}}

\newcommand{\locring}{\mathcal{O}}

\newcommand{\qelbar}{\overline{\Q}_{\ell}}
\newcommand{\affinehecke}[1]{\ensuremath{\mathbb{H}_{#1}}}
\newcommand{\N}[1]{\ensuremath{\mathcal{N}_{#1}}}
\newcommand{\Ntilde}[1]{\ensuremath{\widetilde{\mathcal{N}}_{#1}}}

\newcommand{\weyl}[1]{\ensuremath{\widetilde{W}_{#1}}}
\newcommand{\pervsph}[2]{\ensuremath{P_{#1(\locring)}(Gr_{#2})}}

\newcommand{\iwahorihecke}[1]{\ensuremath{\mathcal{H}_{I_{#1}}}}
\newcommand{\flagvar}[1]{\ensuremath{\mathcal{F}l_{#1}}}

\newcommand{\heckefunc}[1]{\ensuremath{\overset{\leftarrow}{H}_{#1}}}

\newcommand{\I}[1]{\ensuremath{I_{#1}}}

\newcommand{\Gstrat}{{}_{{}_{s_{1},s_{2}}}G(F)}
\newcommand{\Flagstrat}{{}_{{}_{s_{1},s_{2}}}\mathcal{F}l_{G}}

\newcommand{\newtimes}{\tilde{\times}}
\newcommand{\bt}{\tilde{\boxtimes}}
\newcommand{\iso}{\widetilde{\longrightarrow}}

\newcommand{\isom}{{\widetilde\to}}
\newcommand{\on}{\operatorname}

\newcommand{\Gm}{\mathbb{G}_m}

\newcommand{\Fl}{{\mathcal{F}l}}

\newcommand{\Qlb}{\mathbb{\bar Q}_\ell}
\newcommand{\rh}{{\stackrel{\rightarrow}{H}}}

\newcommand{\SL}{\on{SL}}
\newcommand{\gb}{\mathfrak{b}}
\newcommand{\diag}{\on{diag}}
\newcommand{\id}{\on{id}}

\newcommand{\toup}[1]{\stackrel{#1}{\to}}

\newcommand{\cT}{\mathcal{T}}
\newcommand{\cA}{\mathcal{A}}

\newcommand{\cO}{\mathcal{O}}

\newcommand{\cH}{\mathcal{H}}

\newcommand{\cZ}{\mathcal{Z}}

\newcommand{\cB}{\mathcal{B}}
\newcommand{\cN}{\mathcal{N}}
\newcommand{\cX}{\mathcal{X}}

\newcommand{\ZZ}{\mathbb{Z}}
\newcommand{\HH}{\mathbb{H}}

\newcommand{\gJ}{\mathfrak{J}}

\address{Universit\'e Paris-Sud \endgraf
Institut Math\'ematiques d'Orsay, Bat 425, 91405 Orsay Cedex
France 
}
\email{bfhariri@gmail.com}

\begin{document}
\selectlanguage{english}
\title[ Geometric Howe correspondence and Langlands functoriality]{Geometric tamely ramified local theta correspondence in the framework of the geometric Langlands program}
\author{Banafsheh Farang-Hariri}

\begin{abstract}
This paper deals with the geometric local theta correspondence at the Iwahori level for dual reductive pairs of type II over a non Archimedean field $F$ of characteristic $p\neq 2$ in the framework of the geometric Langlands program.  First we construct and study the geometric version of the invariants of the Weil representation of the Iwahori-Hecke algebras. In the particular case of $(\GL_1, \GL_m)$ we give a complete geometric description of the corresponding category.  The second part of the paper deals with geometric local Langlands functoriality at the Iwahori level in a general setting. Given two reductive connected groups $G$, $H$ over $F$ and a morphism $\check{G}\times\SL_2\to\check{H}$ of Langlands dual groups, we construct a bimodule over the affine extended Hecke algebras of $H$ and $G$ that should realize the  geometric local Arthur-Langlands functoriality at the Iwahori level. Then, we propose a conjecture describing the geometric local theta correspondence at the Iwahori level constructed in the first part in terms of this bimodule and we prove our conjecture for pairs $(\GL_1, \GL_m)$. 

\end{abstract}

\maketitle
{\small\tableofcontents}
\section{Introduction} 
In this paper, our aim  is to study the geometric local theta correspondence  (also known as the geometric Howe correspondence) at the Iwahori level for dual reductive pairs of type II in the framework of the geometric Langlands program. We develop this work in two directions. The first path consists in geometrizing the classical Howe correspondence at the Iwahori level by means of perverse sheaves and understanding the underlying geometry.  The second path consists in  constructing a bimodule  that should realize the geometric local Arthur-Langlands funtoriality  at the Iwahori level and studying the relation between the bimodule realizing the geometric Howe correspondence and the one realizing  the geometric local Arthur-Langlands funtoriality  at the Iwahori level. Some of the constructions are done in all  generality while some others are only established for  dual reductive pairs of type II. 
\par\medskip
The basic notions of the Howe correspondence  from the classical point of view have been presented in \cite{MVW}. Let $\K=\F_{q}$ be a finite field of characteristic $p$  different from $2$ and let $F=\K((t))$ and $\mathcal{O}=\K[[t]]$. All representations are assumed smooth and will be defined over $\qelbar$, where $\ell$ is a prime number different from $p$. Let $(G,H)$ be a split dual reductive pair in some symplectic group  $Sp(W)$ over $\K$.  Denote by $(\mathcal{S},\omega)$ the Weil (metaplectic) representation of the meteplectic group associated to $Sp(W)$, see \cite{MVW} and \cite{Kudla}. We assume that the metaplectic cover admits a section over $G(F)$ and $H(F)$. Then, the Howe correspondence is a correspondence between some classes of representations of $G(F)$ and $H(F)$ by means of the restriction of the Weil representation to $G(F)\times H(F)$. This correspondence has been proved in odd characteristic for dual pairs of type I in \cite{Wal} and for dual pairs of type II by Howe and Minguez \cite{Minguez1}. 
\par\medskip
It is interesting to understand the geometry underlying the Howe correspondence and establish its analog  in the geometric Langlands program. This has been initiated by  V.Lafforgue and  Lysenko in  \cite{Lysenko-Lafforgue}, where the authors construct a geometric version of the Weil representation. The second author then studied the unramified case in \cite{Lysenko1} from  global and local point of view for dual reductive pairs $(\mathrm{Sp}_{2n},\mathrm{S}\mathbb{O}_{2m})$ and $(\GL_{m},\GL_{n})$. One of our motivations is to extend the results in \cite{Lysenko1} to the geometric setting of tamely ramified case (the Iwahori level). 
\par\medskip
It is known that the Howe correspondence realizes the Langlands functoriality in some special cases. In the classical setting the reader may refer to \cite{Howe}, \cite{Kudla},  \cite{Minguez1}, \cite{Rallis}, and in the geometric setting one can refer to \cite{Lysenko1}. Adams in \cite{Adams} suggested conjectural relations between Howe correspondence and Langlands functoriality.  Let $\check{G}$ (resp. $\check{H}$) denotes the Langlands dual group of $G$ (resp. of $H$) over $\qelbar$. Under some assumptions, it is expected that there is a morphism $\check{G}\times \mathrm{SL}_{2}\to \check{H}$ such that if  $\pi$ is a smooth irreducible representation of $G(F)$ appearing as a quotient of the Weil representation $\mathcal{S}$ and $\pi^{'}$ is the smooth irreducible representation of $H(F)$ which is the image of $\pi$ under the Howe correspondence, then the Arthur packet of $\pi^{'}$ is the image of the Arthur packet of $\pi$ under the above morphism. For more details  we refer the reader to  \cite{Arthur},  \cite{Kudla}, \cite{Moeglin2}, \cite{Rallis}.   
\par\medskip
Let us describe the basic setting of this paper. Let $I_G$ (resp. $I_{H}$) be a Iwahori subgroup of $G(F)$ (resp. $H(F)$). At the Iwahori level we are interested  in the class of tamely ramified representations. A 
irreducible smooth representation  of $(\pi, V)$ of $G(F)$ is called tamely ramified if the space of invariants under the Iwahori 
subgroup $I_{G}$ is non-zero. The category of tamely ramified representations is the full subcategory of smooth representations 
of finite length consisting of those  representations whose all irreducible subquotients are tamely ramified. Denote by 
$\iwahorihecke{G}$ the Iwahori Hecke algebra of $G$.  According to \cite[Theorem 4.10]{Borel1} there exists an equivalence of categories between the category of tamely ramified smooth representations of $G(F)$ and the category of finite-dimensional 
$\iwahorihecke{G}$-modules. By using this result, we are going to work with finite-dimensional $\iwahorihecke{G}$-modules instead of smooth tamely ramified representations of $G(F)$.  Our strategy is to study the bimodule structure of the space of invariants $\mathcal{S}^{I_{G}\times I_{H}}$ as a 
module over the tensor product $\iwahorihecke{G}\otimes \iwahorihecke{H}$ of Iwahori-Hecke algebras in a geometrical setting. In the sequel (except for \S $\ref{conjecturegtrlf}$, where we will consider any reductive connected group), we restrict ourselves to the case of dual reductive pairs of type II. More precisely, let $L_{0}=\K^{n}$ and $U_{0}=\K^{m}$ with $n\leq m$ and let $G=\GL(L_{0})$ and $H=\GL(U_{0})$. Denote by $\Pi(F)$ the space $U_{0}\otimes L_{0}(F)$ and $\mathcal{S}(\Pi(F))$ the Schwartz space of locally constant functions with compact support on $\Pi(F)$. This Schwartz space realizes the restriction of the Weil representation to $G(F)\times H(F)$, see \cite{MVW}. 
\par\medskip
Let us explain in details the constructions as well as the main results of the geometrization of the Howe correspondence.
Our first step is to define the geometric counterpoint of the space of invariants $\mathcal{S}(\Pi(F))^{I_{G}\times I_{H}}$  and the Hecke 
actions of  $\iwahorihecke{G}$ and $\iwahorihecke{H}$ on this space.  This is done in \S  $\ref{section3}$, where we define the category of 
$I_{G}\times I_{H}$-equivariant perverse sheaves $P_{I_{G}\times I_{H}}(\Pi(F))$ on the  pro-ind scheme  $\Pi(F)$ as well as the 
derived category $D_{I_{G}\times I_{H}}(\Pi(F)).$ The  construction of these categories uses some limit procedure (this issue has been taken care of in \cite[Appendix B]{Lysenko1}). Moreover, we 
define two  Hecke functors geometrizing the bimodule structure of $\mathcal{S}^{I_{G}\times I_{H}}$ in \S $\ref{section4}$; which define the 
action of the category $D_{I_{G}}(\flagvar{G})$ of $I_{G}$-equivariant  $\ell$-adic sheaves on the affine flag variety $\flagvar{G}$ (the same for  $H$)  on $D_{I_{G}\times I_{H}}(\Pi(F))$. These are the generalization to the Iwahori case of the Schwartz space and Hecke functors defined in \cite{Lysenko1} at the unramified level. 
\par\medskip
Next we study the action of the Hecke functors on $D_{H(\locring)\times I_{G}}(\Pi(F))$ of $H(\locring)\times I_{G}$-equivariant perverse sheaves on $\Pi(F)$. This category is 
acted on by  the category $P_{H(\locring)}(Gr_{H})$ of $H(\locring)$-equivariant perverse sheaves on the affine Grassmannian 
$Gr_{H}$ and the category $P_{I_{G}}(\flagvar{G})$.  In \S $\ref{5}$, combining our computations and a result of \cite{Lysenko1} in the unramified case, we prove the following result:
\begin{thma}[Theorem $\ref{H(O)}$]
The Hecke functor 
$$D_{I_{G}}(Gr_G)\to D_{H(\locring)\times I_{G}}(\Pi(F))$$
yields an isomorphism at the level Grothendieck groups between $K(P_{I_{G}}(Gr_G))$ and  $K(P_{H(\locring)\times I_{G}}(\Pi(F)))$ commuting with the above actions of $K(P_{H(\locring)}(Gr_H))$ and $K(P_{I_{G}}(\flagvar{G}))$.
\end{thma}

At the unramified level, this isomorphism is actually verified at level of categories themselves (see \cite{Lysenko1}) and one would hope that the same would be true at the Iwahori level. 
\par\medskip
  
In \S $\ref{6}$, we present one of our key results on the simple objects of the category $P_{I_{H}\times I_{G}}(\Pi(F))$. In Proposition $\ref{r}$, we describe these simple objects in the case $n=m$ and then we establish the general case: 
  
 \begin{thmb}[Theorem $\ref{irrec}$]
 Assume $n\leq m$. 
 
Any irreducible object of $P_{I_{H}\times I_{G}}(\Pi(F))$ is of the form $\IC(\Pi_{N,r}^{w})$ for some $w$ in  $X_{G}\times S_{n,m},$ where $\Pi_{N,r}^{w}$ is the $I_{H}\times I_{G}$-orbit indexed by $w$ on $\Pi(F)$.
 \end{thmb} 
 
In \S $\ref{7}$ we restrict ourselves to the case of the dual pairs $\GL_{1}$ 
 and $\GL_{m}$  for all $m\geq 1$. In this setting, in a series of propositions, we are able to give a complete geometric description of  the module 
 structure of $K(P_{I_{G}\times I_{H}}(\Pi(F)))$ under the action of the Hecke functors. More precisely we work with the category 
 $DP_{I_{G}\times I_{H}}(\Pi(F))$ which takes in consideration the action of the multiplicative group $\G_{m}$ by cohomological shift $-1$. All our computations are at the level of perverse sheaves and the symmetry 
 in this case comes from the action of the perverse sheaves in $P_{I_{H}}(\flagvar{H})$ associated with the elements of length 
 zero in the affine extended Weyl group of $H$.
 
\begin{thmc}[Theorem $\ref{theon1m}$] Let $n=1$ and $m\geq 1$.

The bimodule $K(DP_{I_{G}\times I_{H}}(\Pi(F)))$ is free of rank $m$ over the representation ring of $\check{G}\times \G_{\m}$ with basis $\{\IC^{0},\dots,\IC^{m-1}\}$ and the explicit action of $\affinehecke{H}$ is given by the following formulas: 

\[
\left\{
\begin{split}
&\mathrm{For}\,  1\leq i\leq m:\,  \heckefunc{H}(L_{s_{i}},\IC^{i})\iso \IC^{i+1}\oplus \IC^{i-1}.\\
&\mathrm{For}\, 1\leq i\leq m:\, \heckefunc{H}(L_{s_{i!}},\IC^{i})\iso \IC^{i+1,!}\oplus \IC^{i-1}.\\
&\mathrm{If}\, j\neq i\, \mathrm{mod}\, m : \heckefunc{H}(L_{s_{i}},\IC^{j})\iso \IC^{j}(\qelbar[1](1/2)+\qelbar[-1](-1/2)).\\
&\mathrm{If}\,  j\neq i\, \mathrm{mod}\, m : \heckefunc{H}(L_{s_{i!}},\IC^{j})\iso \IC^{j}[-1](-1/2).\\
& \mathrm{For\,\,  any}\,  i\,\, \mathrm{and}\,\,  k\,\,  in\,\,  \Z:\, \,   \heckefunc{H}(L_{w_{i}},\IC^{k})\iso  \IC^{k+i}
\end{split}
\right.
\]
\end{thmc}

 \par\medskip
   Section $\ref{conjecturegtrlf}$ is devoted to a purely general construction on the geometric local Arthur-Lagnlands functoriality at the Iwahori level. Consider $G$ and $H$  two split reductive connected groups over $\K$ and 
   a map $\check{G}\times \mathrm{SL}_{2}\to \check{H}$ of dual Langlands groups over $\qelbar$. To this data we attach a bimodule 
   $K(\mathcal{X})$ over the affine extended Hecke algebras $\affinehecke{G}$ and $\affinehecke{H}$. We propose the following conjecture: 
   
\begin{conjed}[Conjecture $\ref{firstconjecture}$, \cite{BFH}]
The bimodule over the affine extended Hecke algebras $\affinehecke{G}$ and $\affinehecke{H}$ realizing the local geometric Langlands functoriality at the Iwahori level for the map $\sigma:\check{G}\times \G_{\m}\longrightarrow \check{H}$ identifies with $K(\mathcal{X}).$
\end{conjed}   
   
 We also describe some additional properties of the bimodule $K(\mathcal{X})$ in \S $\ref{jj}$. 
Let us explain some motivation for this conjecture, as well as the forthcoming Conjecture F.
    In \cite{Frenkel-Gaitsgory}, the authors conjecture the existence of some category  
   $\mathcal{C}_{G}$ over the stack of $\check{G}$-local systems  over $D^{*}=\mathrm{Spec}(\K((t)))$ endowed with a "fiberwise" action of $G(F)$. Some conjectures about this category have been formulated in \cite{Frenkel-Gaitsgory}. The  
    construction of this category is more tractable at the Iwahori level. Denote by $\N{\check{G}}$ the nilpotent cone of 
   $\check{G}$ and $\Ntilde{\check{G}}$ its Springer resolution. 
The stack quotient $\N{\check{G}}/\check{G}$ classifies  $\check{G}$-local systems  with regular singularities at the origin and unipotnent monodromy. These  $\check{G}$-local systems  are called tamely ramified. Denote by $\mathcal{C}_{G,nilp}$ the category obtained from $\mathcal{C}_{G}$ via the base change $\N{\check{G}}/\check{G}\to LS_{\check{G}}(D^{*})$, where $LS_{\check{G}}(D^*)$ stands for the $\check{G}$-local 
   systems on $D^*$. The authors conjecture \cite[Formula 0.20]{Frenkel-Gaitsgory} the following isomorphism
\begin{equation}
\label{Arkhi}
K(\mathcal{C}_{G,nilp}^{I_{G}})\iso K(\Ntilde{\check{G}}/\check{G}),
\end{equation}
where the left hand side  is the Grothendieck group of the category of $I_{G}$-invariants in the category 
$\mathcal{C}_{G,nilp}$ and the right hand side is the Grothendieck group of the category of coherent sheaves on the 
stack $\Ntilde{\check{G}}/\check{G}$. Moreover, this 
isomorphism should be compatible with the action of the 
affine extended Hecke algebra.  The stack $\mathcal{X}$ appearing in Conjecture E is a refinement of the stack $\Ntilde{\check{G}}/\check{G}$ in our setting. 
\par\medskip
Let us now explain the link between $K(\mathcal{X})$ and geometric Howe correspondence at the Iwahori level. Consider a dual (split) reductive pair $(G,H)$ over $\K$ with a given map $\check{G}\times \mathrm{SL}_{2}\to \check{H}.$ In 
\cite{Lysenko-Lafforgue} the authors construct a category $\mathcal{W}$ called the Weil category  equipped with an action of $(G\times  
H)(F)$. This is a geometrization of the Weil representation.
Inspired from the series of conjectures presented in \cite{Frenkel-Gaitsgory},  Vincent Lafforgue conjectured that there should exist an equivalence of  categories 
\begin{equation}
\label{Lafforgue}
\mathcal{W}\iso \mathcal{C}_{G}\times_{LS_{\check{H}}(D^{*})}\mathcal{C}_{H}
\end{equation}
as categories equipped with an action of $(G\times H)(F).$ 
\par\medskip
Building on Conjectures $\eqref{Arkhi}$ and $\eqref{Lafforgue}$, we present a new  conjecture at the level of Grothendieck groups linking  the geometric Howe correspondence and local Arthur-Langlands functoriality at the Iwahori level. Let us first give the context of this conjecture. Denote by $(D\mathcal{W})^{I_{G}\times I_{H}}$ the invariants of the category $D\mathcal{W}$,    the latter being a graded version of $\mathcal{W}$.  Denote by $DP_{I_{G}}(\flagvar{G})$  the category whose objects are direct sums of shifted $I_{G}$-equivariant perverse sheaves on $\flagvar{G}$. This monoidal category takes in consideration the action of $\G_m$ by cohomological shift. The category $(D\mathcal{W})^{I_{G}\times I_{H}}$ is acted on by  $DP_{I_{G}}(\flagvar{G})$ and $DP_{I_{H}}(\flagvar{H})$.
The group $K(DP_{I_{G}}(\flagvar{G}))\otimes \qelbar$ is isomorphic 
to the Iwahori-Hecke algebra $\iwahorihecke{G}$. Hence, $K((D\mathcal{W})^{I_{G}\times I_{H}})$ is a bimodule under the action of  $\iwahorihecke{G}$ and  $\iwahorihecke{H}$. According to Iwahori-Matsumoto \cite{Iwahori-mat} the Iwahori-Hecke algebra $\iwahorihecke{G}$ is isomorphic  
to the  affine extended Hecke algebra $\affinehecke{G}$ after specialization. Hence, the algebra  $K((D\mathcal{W})^{I_{G}\times I_{H}})$ is a bimodule over $\affinehecke{G}$ and $\affinehecke{H}$.
Moreover, by the Kazhdan-Lusztig-Ginzburg isomorphism \cite{KH} and \cite{CG},  the affine extended Hecke algebra $\affinehecke{G}$  is isomorphic to  the $\check{G}\times \G_{m}$-equivariant K-theory  $K^{ \check{G}\times \G_{m}}
(Z_{\check{G}})$ of the Steinberg variety $Z_{\check{G}}$ of $\check{G}$.
 This isomorphism is used to define  $\affinehecke{G}$ and $\affinehecke{H}$ module structure on $K(\mathcal{X})$. 
Remark that the usual Kazhdan-Lusztig-Ginzburg isomorphism can be upgraded to the following isomorphism:
$$K(DP_{I_{G}}(\flagvar{G}))\iso K^{ \check{G}\times \G_{m}}
(Z_{\check{G}}).$$
\noindent Recently, Bezrukavnikov proved in \cite{Bezrukavnikov} an equivalence of categories  between $D_{I_{G}}(\flagvar{G})$ and  the category of coherent sheaves on the Steinberg variety of $\check{G}$  lifting this isomorphism to the categorical level.
\par\medskip
We may now present our second conjecture: 

\begin{conjee}[\cite{BFH}]
\label{pff}
The bimodule $K(\mathcal{X})$ is isomorphic to the Grothendieck group of the category $(D\mathcal{W})^{I_{G}\times I_{H}}$ under the action of the affine extended Hecke algebras  $\affinehecke{G}$  and $\affinehecke{H}$. 
\end{conjee}

\par\medskip 
If  $G=\GL_{n}$ and 
$H=\GL_m$,  we can give a more concrete conjecture building on our construction of the geometric version of the Howe correspondence $K(DP_{I_{G}\times I_{H}}(\Pi(F)))$:
\begin{conjef}[Conjecture $\ref{secondeconjecture}$, \cite{BFH}]  Let $G=\GL_{n}$ and 
$H=\GL_m$.

The bimodule $K(\mathcal{X})$ is isomorphic to the Grothendieck group of the category $K(DP_{I_{G}\times I_{H}}(\Pi(F)))$ under the action of the affine extended Hecke algebras  $\affinehecke{G}$ and $\affinehecke{H}$
\end{conjef}

The last section \S $\ref{9}$ is devoted to the proof of the Conjecture F in a particular case. The result we obtain is  
 
\begin{thmg}[Theorem $\ref{theoremprincipal}$, \cite{BFH}]
For any $m\geq 1$, Conjecture F is true for the pair $(\GL_{1},\GL_{m})$.
\end{thmg}

This theorem expresses the Howe 
correspondence in terms of $K(\mathcal{X})$ for pairs $(\GL_1,\GL_m)$. The idea underlying this Theorem is that the explicit description of the Howe correspondence in the classical setting obtained 
by Minguez in \cite{Minguez1} should be upgraded to a finer description of the bimodule  in terms of the stack $\mathcal{X}$ 
attached to the map $\check{G}\times \mathrm{SL}_{2}\to \check{H}$. This opens an important perspective as the same 
description should also hold for other dual pairs. Especially it should be interesting to obtain a similar result for the dual 
pairs $(\mathrm{Sp_{2n}}$, $\mathrm{S}\mathbb{O}_{2m})$ and provide in this way a conceptually new approach to the computations done in
\cite{Aubert} and \cite{Aub2}. Another important perspective is a hope that the whole derived category $D_{I_{G}\times I_{H}}(\Pi(F))$ could 
possibly be described in terms of the derived category of coherent sheaves over the stack $\mathcal{X}$ in the same spirit of the recent work in \cite{Bezrukavnikov} and \cite{Arkhipov}. 
\par\bigskip
\textbf{Acknowledgments}: The author would like to thank Vincent Lafforgue who agreed that an unpublished conjecture of his be included in the manuscript. The results presented in this paper are  part of the Ph.D. thesis of the author under the supervision of Sergey Lysenko.   
\par\medskip
\section{Notation}
In this paper, $\K$ is an algebraically closed field of characteristic $p>2$ except for \S $\ref{conjecturegtrlf}$ and \S $\ref{9}$ where $\K$ is assumed to be finite. Let $F=\K((t))$ be the field of Laurent series with coefficients in $\K$ and $\mathcal{O}=\K[[t]]$ its ring of integers. Let $\ell$ be a prime number different from $p$. We will denote by $G$ a connected reductive group over $\K$ and by $G(F)$ the set of its $F$-points. Fix a maximal torus $T$ and a Borel subgroup $B$ of $G$ containing $T$.  Throughout the paper we denote by $\check{X}$ the lattice of  characters of $T$ and  by $X$ the cocharacters lattice of $T$, see for background and details \cite{Bourbaki}. We denote by $\check{R}$ the set of roots and by $R$ the set of coroots.  Denote by $(\check{X},\check{R},X,R, \Delta)$ the root datum associated with $(G,T,B)$, where $\Delta$ denotes the basis of simple roots. Denote by $X^{+}$ the set of dominant cocharacters of $G$.  Denote by $I_{G}$ the Iwahori subgroup of $G(F)$ associated with $B$. Denote by $\check{G}$ the Langlands dual group of $G$ over $\qelbar$. All representations are assumed to be smooth and are considered over $\qelbar$. We denote by  $\mathrm{Rep}(\check{G})$ (resp. $\mathrm{R}(\check{G})$) the category (resp. ring) of smooth representations of $\check{G}$ over $\qelbar$.
\par\medskip
Denote by $W_{G}$ the finite Weyl group of the root datum  $(\check{X},\check{R},X,R, \Delta)$ and by $s_{\check{\alpha}}$ the simple reflection corresponding to the root $\check{\alpha}$. We denote by $w_{0}$ the longest element of the Weyl group $W_{G}$. In all our notation, if there is no ambiguity we will omit the subscript $G$. Denote by $\weyl{G}$ the affine extended Weyl group,  which is the semi-direct product $W_{G}\ltimes X$ where $W_{G}$ acts on $X$ in a natural way. We will assume additionally that the root datum is irreducible and the unique highest root will be denoted by $\check{\alpha_{0}}$. Let $S_{aff}=\{s_{\alpha}\vert \alpha\in\Delta\}\cup\{s_{0}\}$, where $s_{0}=t^{-\alpha_{0}}s_{\check{\alpha}_{0}}$. The subgroup $W_{aff}$ of $\weyl{G}$ generated by $S_{aff}$ is the affine Weyl group associated with the root datum.  Denote by $\ell$ the length function defined on the coxeter group $W_{aff}$ which extends to a length function on $\weyl{G}$. Let $Q$ denote a subgroup of $X$ generated by coroots. One has $W_{aff}\iso W_{G}\ltimes Q$ and the subgroup $W_{aff}$ is normal in $\weyl{G}$ and admits a complementary subgroup $\Omega=\{w\in{\weyl{G}}\vert \ell(w)=0\},$
the elements of length zero. Moreover, we have $\weyl{G}\iso W_{aff}\rtimes \Omega$, which we will use as a description of $\weyl{G}$. 
\par\medskip
For any scheme or stack $S$ locally of finite type over $\K$, we denote by $D(S)$ the bounded derived category of constructible $\qelbar$-sheaves over $S$. Write $\mathbb{D}:D(S)\to D(S)$ for the Verdier duality functor. We denote by $P(S)$ the full subcategory of perverse sheaves in $D(S)$. We will also use a subcategory $DP(S)$ of $D(S)$  defined over any scheme  or stack $S.$ The objects of $DP(S)$ are the objects of  $\bigoplus_{i\in\Z} P(S)[i],$ and for $K,K^{'}\in P(S)$ and  $i, j \in \Z$ the morphisms are: 
\[ Hom_{DP(S)}(K[i],K^{'}[j]) = \left\{ \begin{array}{ll}
         Hom_{P(S)}(K,K^{'}) & \mbox{if $i=j$};\\
        0 & \mbox{if $i\neq j$}.\end{array} \right. \]
Let $X$ be a scheme of finite type over $\K$ and let $G$ be a connected algebraic group acting  on $X$. We denote by $P_{G}(X)$ the full subcategory of $P(X)$ consisting of $G$-equivariant perverse sheaves. The derived category  of $G$-equivariant $\qelbar$-sheaves on $X$ is denoted by $D_{G}(X).$ For  any smooth $d$-dimensional irreducible locally closed subscheme $Z$ of $X$, if $i:Z\to X$ is the corresponding immersion, we define the intersection cohomology sheaf (IC-sheaf for short) $\IC(Z)$ as the perverse sheaf $i_{Z!*}(\qelbar)[d].$ 
\par\medskip 
Let us recall the affine Grassmannian and affine flag variety and some of their properties, see \cite{BD} and \cite{Mirkovic}.
We denote by $Gr_{G}$ the affine Grassmannian defined as the $\K$-space quotient 
 $G(F)/G(\locring)$.  If $G$ is the linear algebraic 
group $\GL_{n}$ over $\K$, the $\K$-points of $Gr_{G}$ is 
naturally identified with the set of lattices in 
$\K((t))^{n}$, see \cite{Laszlo}. The affine Grassmannian is 
an ind-scheme of ind-finite type. 
Given $\lambda$ in $X,$ the $G(\locring)$-orbit associated with $W_{G}.{\lambda}$ is $G(\locring).t^{\lambda}$, we denote by $Gr_{G}^{\lambda}.$ We have the Cartan decomposition of $G(F):$
$$G(F)=\bigcup_{\lambda\in{X^{+}}}G(\locring)t^{\lambda}G(\locring).$$
For any $\lambda$ and $\mu$ in $X^{+}$, $Gr_{G}^{\mu}\subset \overline{Gr_{G}^{\lambda}}$, if and only if $\lambda-\mu$ is a sum of positive coroots and 
 \[\overline{Gr^{\lambda}}=\bigsqcup_{\mu\leq{\lambda}} Gr^{\mu}.\] 
For any $\lambda$ in $X^{+}$, the dimension of $Gr_{G}^{\lambda}$ is $\langle 2\check\rho,\lambda\rangle,$ where $\check\rho=\dfrac{1}{2}\sum _{\check{\alpha}\in \check{R}^{+}}\check\alpha$ is the half sum of positive roots.
\par\medskip
Denote by $\flagvar{G}$ the affine flag variety for $G$ defined as the quotient $\K$-space $G(F)/I_{G}$, which is an ind-scheme of ind-finite type as well.
The affine flag variety decomposes as a disjoint union
\[\flagvar{G}=\bigcup_{w\in{\weyl{G}}}I_{G}wI_{G}/I_{G}.\]
The closure of each Schubert cell $I_{G}wI_{G}/I_{G}$ is a union of Schubert cells and the closure relations are given by the Bruhat order:
\[\overline{I_{G}wI_{G}/I_{G}}=\bigcup_{w^{'}\leq w}I_{G}w^{'}I_{G}/I_{G}.\]
For any $w\in{\weyl{G}}$  we will denote the Schubert cell $I_{G}wI_{G}/I_{G}$ by $\flagvar{G}^{w}$. It is isomorphic to $\mathbb{A}^{\ell(w)}.$
\par\medskip
Let $R$ be a $\K$-algebra. A complete periodic flag of lattices inside $R((t))^{n}$ is a flag
$$L_{-1} \subset L_{0} \subset L_{1} \subset \dots$$
 such that each $L_{i}$ is a lattice in $R((t))^{n},$   each quotient $L_{i+1}/L_{i}$ is a locally free $R$-module of rank one and $L_{n+k}=t^{-1}L_{k}$ for any $k$ in $\Z.$
\par\medskip
 For $1\leq i\leq n$, denote by $\{e_{1}, \dots, e_{n}\}$ a basis of $L_{0}$ and set 
 $$\Lambda_{i,R}=(\oplus_{j=1}^{i} t^{-1}R[[t]]e_{j})\oplus(\oplus_{j=i+1}^{n}R[[t]]e_{j}).$$ 
For all $i$ in $\Z$, we set $\Lambda_{i+n,R}=t^{-1}\Lambda_{i,R}$. 
This defines the standard complete lattice flag 
$$\Lambda_{-1,R}\subset\Lambda_{0,R}\subset\Lambda_{1,R}\subset \dots $$
denoted by $\Lambda_{\bullet,R}$ in $R((t))^{n}.$  
 Each point of $\GL_{n}(R((t)))$ gives rise to a flag of lattices inside $R((t))^{n}$ by applying it to the standard lattice flag.
The Iwahori subgroup $I_{G}\subset \GL_{n}(\K[[t]])$ is precisely the stabilizer of the standard lattice flag $\Lambda_{\bullet,\K}.$  For any $\K$-algebra $R$, $\flagvar{\GL_{n}}(R)$ is naturally in bijection with the set of complete periodic lattice flags in $R((t))^{n}.$   
\par\medskip
Denote by $P_{G(\locring)}(Gr_{G})$ (resp. $P_{I_{G}}(Gr_{G})$)  the category of $G(\locring)$-equivariant (resp. $I_{G}$-equivariant) perverse sheaves on the affine Grassmannian $Gr_{G}$ and denote by $P_{I_{G}}(\flagvar{G})$ the category of $I_{G}$-equivariant perverse sheaves on the affine flag variety $\flagvar{G}$. 
The category $P_{G(\locring)}(Gr_{G})$ is equipped with a geometric convolution functor denoted by $\star$ which preserves perversity and makes $P_{G(\locring)}(Gr_{G})$ into a symmetric monoidal category, see \cite{Mirkovic}. 
We define the extended geometric Satake equivalence in the following way: 
$$DP_{G(\locring)}(Gr_{G})\iso \mathrm{Rep}(\check{G}\times \G_{m}),$$ 
for any  perverse sheaf  $K$ in $P_{G(\locring)}(Gr_{G})$ and any integer $i$, this functor sends $K[i]$ to $\mathrm{loc}(K)\otimes I^{\otimes -i}$, where $I$ is the standard representation of $\G_{m}$ and $\mathrm{loc}:P_{G(\locring)}(Gr_{G})\to \mathrm{Rep}(\check{G})$ is the Satake equivalence. 
\par\medskip
One may define a geometric convolution functor on $D_{I_{G}}(\flagvar{G})$ as well but this convolution functor does not preserve perversity, see \cite{Gaitsgory} and \cite{Arkhipov}.
\par\medskip
Assume temporarily that the ground field $\K$ is the finite field $\F_{q}$.
We define the Iwahori-Hecke algebra  $\iwahorihecke{G}$ to be the space $C_{c}(I_{G}\backslash G(F) / I_{G})$ of locally constant, $I_{G}$-bi-invariant compactly supported $\qelbar$-valued functions on $G(F).$ We fix a Haar measure $dx$ on $G(F)$ such that $I_{G}$ be of measure $1$  and endow  $\iwahorihecke{G}$ with the convolution functor. There are two well-known  presentations of this algebra by generators and relations. The first is due to Iwahori-Matsumoto \cite{Iwahori-mat} and the second is by Bernstein in \cite{Lu1} and \cite{Lu2}. We will use the second presentation. Moreover, we have  the isomorphism $K(DP_{I_{G}}(\flagvar{G}))\otimes \qelbar\iso \iwahorihecke{G}.$

\section{Geometric model of the Schwartz space at the Iwahori level}
\label{section3}
Let $M_{0}$ be a finite-dimensional representation of $G$ over $\K$ and let $M=M_{0}\otimes _{\K}\mathcal{O}$. The definitions of the derived category  $D(M(F))$ of $\ell$-adic sheaves on $M(F)$ and the category $P(M(F))$ of $\ell$-adic perverse sheaves on $M(F)$ are given in \cite{Lysenko1}. The category $D(M(F))$ is a geometric analogue of the Schwartz space of locally constant functions with compact support on $M(F).$  We recall their definition briefly and the use them to define the $I_G$-equivariant version of these categories.  One can find general details on ind-pro systems in \cite[Appendix B]{Lysenko1}. These are the generalizations of the construction in \cite{Lysenko1} in the tamely ramified case. 
\par\medskip
For any two integers $N,r\geq 0$ with $N+r>0,$  set $M_{N,r}=t^ {-N}M/t^{r}M.$ Given positive integers $N_{1}\geq {N_{2}}$, $r_{1}\geq r_{2},$ we have the following Cartesian diagram 
 
\begin{equation}
\label{diag1}
\xymatrix{
M_{{}_{N_{2},r_{1}}}  \ar@{^(->}[r]^{i} \ar@{->>}[d]^{p} & M_{{}_{N_{1},r_{1}}} \ar@{->>}[d]^{p} \\
M_{{}_{N_{2},r_{2}}} \ar@{^(->} [r]^{i} &M_ {N_{1},r_{2}} 
}
\end{equation}
where $i$ is the natural closed immersion and $p$ is the projection. 
Consider the following  functor 
\begin{align}
\label{p*}
D(M_{N,r_{2}})& \longrightarrow {D(M_{N,r_{1}})} \nonumber\\
              K & \longrightarrow p^*K[\dim rel(p)]
\end{align}
According to \cite[ Proposition 4.2.5]{BBD} the functor $\eqref{p*}$ is fully faithful and exact for the perverse t-structure.  The functor $i_*$ is fully faithful and exact for the perverse $t$-structure as well. This yields a commutative diagram of triangulated categories:
\begin{equation}
\label{transitionfunctors}
 \xymatrix{
D(M_{{}_{N_{2},r_{1}}})  \ar@{^(->}[r]^{i_{*}} & D(M_{{}_{N_{1},r_{1}}})  \\
D(M_{{}_{N_{2},r_{2}}})  \ar[u]^{p^{*}[\dim rel(p)]} \ar@{^(->} [r]^{i_{*}} & D(M_{{}_{N_{1},r_{2}}}) \ar[u]_{p^{*}[\dim rel(p)]}
}
\end{equation}
The derived category $D(M(F))$ is defined as the inductive $2$-limit of derived categories $D(M_{N,r})$ as $N,r$ go to infinity. Similarly, $P(M(F))$ is defined as the inductive $2$-limit of the categories $P(M_{N,r})$. 
\par\medskip
Assume $N+r>0$. The subgroup $G(\locring)$  acts on $M_{N,r}$ via its finite dimensional quotient $G(\locring/t^{N+r}\mathcal{O}).$ Denote by $I_{s}$ the kernel of the map $G(\locring)\longrightarrow G(\locring/t^ {s}\locring).$ The Iwahori subgroup $I_G$ acts on $M_{N,r}$ via its finite-dimensional quotient $I_G/I_{N+r}$. For $s>0$ denote by $K_s$ the quotient $I_G/I_{s}.$
\par\medskip
Let $r_1\geq N+r>0,$ we have the projection $K_{r_{1}}\twoheadrightarrow K_{N+r}.$ This leads to the following projection between stack quotients
 $$q: K_{r_1}\backslash M_{N,r}\twoheadrightarrow K_{N+r}\backslash M_{N,r},$$
and gives rise to an equivalence of equivariant derived categories 
$$ D_{K_{N+r}}(M_{N,r})\iso D_{K_{r_{1}}}(M_{N,r}).$$
This equivalence is also exact for perverse $t$-structure. Denote by $D_{I_G}(M_{N,r})$ the derived category of $K_{r_1}$-equivariant $\ell$-adic sheaves $D_{K_{r_1}}(M_{N,r})$ for any $r_{1}\geq N+r.$
\par\medskip
By taking the stack quotient of Diagram $\eqref{diag1}$ by $K_{N_{1}+r_{1}}$, we obtain
\begin{equation} 
\label {diag2}
\xymatrix{
D_{I_G}(M_{{}_{N_{2},r_{1}}})  \ar@{^(->} [r]^{i_{*}} & D_{I_{G}}(M_{{}_{N_{1},r_{1}}})  \\
D_{I_{G}}(M_{{}_{N_{2},r_{2}}})  \ar[u]^{p^{*}[\dim rel(p)]} \ar@{^(->} [r]^{i_{*}} & D_{I_{G}}(M_{N_{1},r_{2}}) \ar[u]_{p^{*}[\dim rel(p)]}
}
\end{equation}
 where each arrow is fully faithful and exact for the perverse $t$-structure.   Define $D_{I_G}(M(F))$ as the inductive $2$-limit of $D_{I_G}(M_{{}_{N,r}})$ 
as $N,r$ go to infinity. Similarly we define the category $P_{I_{G}}(M(F)).$ Since the Verdier duality $\mathbb{D}$ is compatible with the transition functors in both diagrams $\eqref{transitionfunctors}$ and $\eqref{diag2}$ we have the Verdier duality self-functors $\mathbb{D}$ on $D_{I_{G}}(M(F))$ and $D(M(F))$.
\par\medskip
In order to define an action of the Hecke functors on $D_{I_{G}}(M(F))$, let us first
define the equivariant derived category $D_{I_{G}}(M(F)\times \flagvar{G}).$ Let $s_1,s_2\geq 0$ and set 
\begin{equation}
\label{gstrat}
\Gstrat=\{ g\in G(F) \vert \hspace{2mm} t^{s_{1}}M\subset{gM}\subset{t^{-s_{2}}M} \}.
\end{equation}
Then $\Gstrat\subset G(F)$ is closed and stable under the left and right multiplication by $G(\locring).$ 
Further, $\Flagstrat=\Gstrat/ I_{G}$ is closed in $\flagvar{G}.$ For $s^{'}_{1}\geq s_{1}$ and $s^{'}_{2}\geq s_{2},$
 we have the closed embeddings ${}_{s_{1},s_{2}}\mathcal{F}\ell_{G}\hookrightarrow{{}_{s^{'}_{1},s_{2}^{'}}\mathcal{F}\ell_{G}}$
and the union of ${}_{s_{1},s_{2}}\mathcal{F}\ell_{G}$ is the affine flag variety $\flagvar{G}.$  The map sending $g$ to $g^{-1}$ yields an isomorphism between ${{}_{{}_{s_{1},s_{2}}}G(F)}$ and ${{}_{{}_{s_{2},s_{1}}}G(F)}.$
\par\medskip
From now on let us assume $M_0$ is a faithful representation of $G$. Then ${}_{s_{1},s_{2}}\mathcal{F}\ell_{G}\subset \mathcal{F}\ell_{G}$ is a closed subscheme of finite type.
\begin{lemma}
For any $s_1,s_2\geq 0,$ the action of $G(\locring)$ on $\Flagstrat$ factors through the quotient $G(\locring/ t^{s_{1}+s_{2}+1}\locring).$
\end{lemma} 
\begin{proof}
 Choose a Borel $B^{'}$ in $GL(M_{0})$ such that $B=G\cap B^{'}$. Denote by 
$$M\subset M_{1}\subset \dots \subset M_{n}=t^{-1}M$$
 the full flag preserved by $B^{'}.$ The Iwahori subgroup $I_{G}$ consists of the elements $g$ of $G(F)$ preserving $M$ together with the flag $(M)_{i}$ above.  Hence the map  from $\mathcal{F}\ell_{G}$ to $\mathcal{F}\ell_{GL(M_{0})}$ sending a point $gI_{G}$ to the flag $(gM\subset gM_{1}\subset \dots \subset gM_{n})$ is a closed immersion. Thus $\Flagstrat$ is realized as the closed subscheme in the scheme classifying a lattice $M^{'}$ such that $t^{s_{1}}M\subset M^{'}\subset t^{-s_{2}}M$ together with the full flag 
 $$M^{'}\subset M_{1}^{'}\subset \dots \subset M^{'}_{n}=t^{-1}M^{'}.$$
The action of $G(\locring)$ on the latter scheme factors through $G(\mathcal{O}/t^{s_{1}+{s_{2}}+1}\locring)$. 
 \end{proof}
\par\medskip 
 The action of $I_{G}$ on $\Flagstrat$ factors through $K_{s}=I_{G}/I_{s}$ for $s\geq s_{1}+s_{2}+1$. If $s\geq \max\{N+r, s_{1}+s_{2}+1\},$ the group $K_{s}$ acts on $M_{{}_{N,r}}\times{\Flagstrat}$ diagonally and we can consider the equivariant derived category $D_{K_{s}}(M_{{}_{N,r}}\times{\Flagstrat})$. For $s^{'}\geq{s}$ one has a canonical equivalence
$$D_{K_s}(M_{{}_{N,r}}\times{\Flagstrat})\iso {D_{K_{s^{'}}}(M_{{}_{N,r}}\times{\Flagstrat})}.$$
\par\medskip
Define $D_{I_{G}}(M_{N,r}\times{\Flagstrat})$ as the category $D_{K_{s}}(M_{N,r}\times{\Flagstrat})$ for any $s\geq \max\{N+r,s_{1}+s_{2}+1\}.$
\par\medskip
We define the category $D_{I_G}(M_{{}_{}}(F)\times{\flagvar{G}})$ as the inductive $2$-limit of the category $D_{I_G}(M_{{}_{N,r}}\times{\Flagstrat})$ as $N,r,s_1,s_2$ go to infinity. The subcategory $P_{I_{G}}(M(F)\times \flagvar{G})\subset D_{I_{G}}(M(F)\times \flagvar{G})$ of perverse sheaves is defined along the same lines.

\section{Hecke functors at the Iwahori level}
\label{section4}
 \label{heckefunctors}
We use the same notation as in the previous section. Denote by $\check\mu$ in $\check{X}^{+}$ the character by which $G$ acts on $\det(M_{0}).$  The connected components of the affine Grassmannian $Gr_{G}$ are indexed by the algebraic fundamental group $\pi_{1}(G)$ of $G$, see \cite{BD}. For $\theta$ in $\pi_{1}(G)$, choose $\lambda$ in $X^{+}$ whose image in $\pi_{1}(G)$ equals $\theta.$ Denote by $Gr_{G}^{\theta}$ the connected component of $Gr_G$ containing $Gr_{G}^{\lambda}$. The affine flag manifold $\flagvar{G}$ is a fibration over $Gr_{G}$ with typical fiber $G/B.$ Hence the connected components of the affine flag variety $\flagvar{G}$ are also indexed  by $\pi_{1}(G).$ For $\theta$ in $\pi_{1}(G),$ denote by $\flagvar{G}^{\theta}$ the preimage of $Gr_G^{\theta}$ in $\flagvar{G}$. Set $\Flagstrat^{\theta}=\flagvar{G}^{\theta}\cap{\Flagstrat}.$ 
\par\medskip
Let us now define the Hecke functors (geometrization of the action of the Iwahori-Hecke algebra $\iwahorihecke{G}$ on the invariants of the Weil representation under the action of $I_{G}$) of $P_{I_{G}}(\flagvar{G})$ on $D_{I_{G}}(M(F)),$ denoted by
\begin{equation}
\label{heckefuncG}
\heckefunc{G}:D_{I_{G}}(\flagvar{G})\times D_{I_{G}}(M(F))\longrightarrow D_{I_{G}}(M(F)).
\end{equation}
Consider the following isomorphism 
 \[
\xymatrix @R=.3em{
 \alpha:M(F)\times G(F) \ar[r] & M(F)\times G(F)  \\
(v,g)   \ar[r] & (g^{-1}v,g). 
}
\]
Any element $(a,b)\in{I_G}\times{I_G}$ act on the source by $(a,b).(v,g)=(av,agb)$ and act on the target $(v^{'},g^{'})$ by $(a,b).(v^{'},g^{'})=(b^{-1}v^{'},ag^{'}b).$ The map $\alpha$ is $I_{G}\times I_{G}$-equivariant with respect to these two actions. Hence this yields a morphism of stacks 
$$M(F)\times \flagvar{G}\longrightarrow (M(F)/I_{G})\times \flagvar{G},$$
and enables us to define the following morphism of stack quotients 
$$act_{q}:I_G\backslash(M(F)\times \flagvar{G})\longrightarrow (M(F)/I_{G})\times({I_G}\backslash\flagvar{G}), $$
 where the action of $I_{G}$ on $M(F)\times \flagvar{G}$ is the diagonal one.  
 The following lemma generalizes a construction done in \cite[\S 4]{Lysenko1} to the Iwahori case.
\begin{lemma}
\label{Lm_229}
There exists  an inverse image functor 
$$act_{q}^{*}:D_{I_G}(M(F))\times D_{I_G}(\flagvar{G}) \longrightarrow D_{I_G}(M(F)\times \flagvar{G})$$
which preserves perversity and is compatible with the Verdier duality in the following way:
for any $\mathcal{K}$ in $D_{I_{G}}(M(F))$ and $\mathcal{F}$ in $D_{I_{G}}(\flagvar{G})$ we have  $$\mathbb{D}(act_{q}^{*}(\mathcal{K},\mathcal{T}))\iso act_{q}^{*}(\mathbb{D}(\mathcal{K}),\mathbb{D}(\mathcal{T})).$$
\end{lemma}

\begin{proof}
Given $N,r,s_{1},s_{2}\geq 0$ with $r\geq{s_{1}}$ and $s\geq{\max\{N+r,s_{1}+s_{2}+1\}},$
one can define the following commutative diagram 
\begin{center}
\[
\xymatrix @R=2cm{
&& {M}_{{}_{N,r}}\times \Gstrat \ar[rr]^{act} \ar[d]_{q_G} && M_{{}_{N+s_{1},r-s_{1}}}\ar[d]^{q_{M}} \\
  {M}_{{}_{N,r}} \ar[d] &&  {M}_{{}_{N,r}}\times{\Flagstrat} \ar[d]  \ar[ll]_-{pr_{1}} \ar[rr]^{act_{q}}   && K_{s}\backslash{M_{N+s_{1},r-s_{1}}} \\
K_{s}\backslash {M}_{{}_{N,r}}  && K_s\backslash({{M}_{{}_{N,r}}\times\Flagstrat)}  \ar[ll]_-{pr} \ar[rru]^{act_{q,s}} \ar[rr] ^{pr_{2}} && K_s\backslash\hspace{1mm}(\Flagstrat)}
\]
\end{center}
The action map $act$ sends the couple $(v,g)$ to $g^{-1}v$. The maps  $pr_{1}$, $pr_2$ and $pr$ are projections. The map $q_{G}$ sends the couple $(v,g)$ to $(v,g{I_{G}}).$ All the vertical arrows are the projections of the stack quotients for the action of the corresponding group. The group $K_{s}$ acts diagonally on $M_{N,r}\times \Flagstrat$ and the map $act_{q}$ is equivariant with respect to this action. This enables us to define  the following functor:
$$D_{I_G}(M_{N+s_{1},r-s_{1}})\times D_{I_G}(\Flagstrat)\overset{temp}{\longrightarrow}{D_{I_G}(M_{N,r}\times{\Flagstrat})}$$
sending $(\mathcal{K},\mathcal{T})$ to 
$$(act_{q,s}^{*}\mathcal{K})\otimes{pr_{2}^{*}}\mathcal{T}[\dim(K_s)-c+s_1\dim M_0]
$$
where $c$ equals $\langle\theta,\check\mu\rangle$ over $\Flagstrat^{\theta}.$
\par\medskip
Set $r_{1}\geq{r_{2}}$ and $s\geq{\max\{s_{1}+s_{2},N+r_{1}\}}.$ Then we have the diagram 
\begin{equation}
 \label{diag3}
  \xymatrix{
K_s\backslash(M_{N,r_{1}}\times{\Flagstrat}) \ar[d] \ar[r]^{act_{q,s}} & K_s\backslash(M_{N+s_{1},r_{1}-s_{1}}) \ar[d]\\
K_s\backslash(M_{N,r_{2}}\times{\Flagstrat})  \ar [r]^{act_{q,s}} & K_s\backslash(M_{N+s_{1},r_{2}-s_{1}})
}
\end{equation}

The functors $temp$ and the transition functors in $\eqref{diag3}$ are compatible. This gives rise to a functor
$$temp_{N,s_{1},s_{2}}:D_{I_{{}_{G}}}(M_{N+s_{1}})\times{D_{I_{{}_{G}}}(\Flagstrat)\longrightarrow{D_{I_{{}_{G}}}(M_{N}\times{\Flagstrat})}},$$
where $M_{N}=t^{-N}M.$  
\par\medskip
Let $N_{1}\geq{N+s_{2}}$. Then $N\leq N_{1}-s_{2} \leq N_{1}+s_{1}$ and we have the closed immersion $M_{N}\hookrightarrow M_{N_{1}+s_{1}}$. Thus we have
 \begin{equation}
 \label{diag4}
\xymatrix{
D_{I_G}(M_{N})\times{D_{I_G}(\Flagstrat)}  \ar[r] & D_{I_G}(M_{N_{1}+s_{1}})\times{D_{I_G}(\Flagstrat) \ar[d]^{temp_{N_{1},s_{1},s_{2}}}}\\
   & D_{I_{{}_{G}}}(M_{N_{1}}
  \times{\Flagstrat}) \ar[d]\\
  & D_{I_{G}}(M(F)\times{\Flagstrat})
}
\end{equation}
where the first arrow is the extension by zero under the closed immersion $M_{N}\hookrightarrow M_{N_{1}+s_{1}}$. 
For any $\mathcal{K}$ in $D_{I_{G}}(M_{N})$ and any $\mathcal{T}$ in $D_{I_{G}}(\Flagstrat)$, the image of $(\mathcal{K},\mathcal{T})$ under the composition $\eqref{diag4}$ does not depend on $N_{1}.$ So we get a functor 
$$temp_{s_1,s_2}:D_{I_{{}_{G}}}(M_{N})\times{D_{I_{G}}(\Flagstrat)}\longrightarrow{D_{I_{{}_{G}}}}(M(F)\times{\Flagstrat}).$$

For any $s^{'}_{1}\geq{s_{1}},$ and $s^{'}_{2}\geq{s_{2}},$ we have the extension by zero functors
$$D_{I_{{}_{G}}}(\Flagstrat)\hookrightarrow{D_{I_{{}_{G}}}({}_{{}_{s^{'}_{1},s^{'}_{2}}}\mathcal{F}\ell_{G})},$$ 
which are compatible with our functor $temp_{s_{1}s_{2}}$, so 
this yields the desired functor 
 $$act_q^*: D_{I_{{}_{G}}}(M(F))\times{D_{{I}_{{}_{G}}}}(\flagvar{G})\overset{temp}{\longrightarrow}{D_{I_{{}_{G}}}(M(F)\times{\flagvar{G}})}.$$
\par\medskip

One checks that $\mathbb{D}(act_{q}^{*}(\mathcal{K},\mathcal T))\iso act_{q}^{*}(\mathbb{D}(\mathcal{ K}),\mathbb{D}(\mathcal{T}))$,
and $act_{q}^{*}$ preserves perversity. 
\end{proof}

 For any $N,r,s_{1},s_{2}$ greater than  zero satisfying the following condition $s\geq{\max\{N+r,s_{1}+s_{2}+1\}},$ consider the projection 
$$pr:K_s\backslash{(M_{N,r}\times{\Flagstrat})}\longrightarrow K_s\backslash M_{N,r},$$
which gives us 
$$pr_{!}:D_{K_s}({M}_{{}_{N,r}}\times{\Flagstrat})\longrightarrow{D_{K_s}({M}_{{}_{N,r}})}.$$
These functors are compatible with the transition functors in $\eqref{diag3}$ and yield a functor
$$pr_{!}:D_{I_G}(M(F)\times{\flagvar{G})\longrightarrow{D_{I_G}(M(F))}}.$$

For any $\mathcal{K}$ in $D_{I_{{}_{G}}}(M(F))$ and $\mathcal{T}$ in $D_{I_{{}_{G}}}(\flagvar{G})$, the Hecke operator  $\overset{\leftarrow}{H_{G}}(\, ,\, )$ $\eqref{heckefuncG},$ is defined by 
$$
\heckefunc{G}(\mathcal{T}, \mathcal{K})=pr_{!}(act_q^*(\mathcal{K}, \mathcal{T})).
$$
Moreover, this functor is compatible with the convolution functor on $D_{I_{G}}(\flagvar{G})$. Namely, given $\mathcal{T}_1, \mathcal{T}_2$ in $D_{I_G}(\flagvar{G})$ and $\mathcal{K}$ in $D_{I_G}(M(F))$, one has naturally
$$
\heckefunc{G}(\mathcal{T}_1, \heckefunc{G}(\mathcal{T}_2, \mathcal{K}))\iso \heckefunc{G}(\mathcal{T}_1\star  \mathcal{T}_2, \mathcal{K}).
$$ 
One may also consider the category $DP_{I_G}(\flagvar{G})$ and consider the Hecke functors in the form
\begin{equation}
\heckefunc{G}:  DP_{I_{G}}(\flagvar{G})\times D_{I_{G}}(M(F))\longrightarrow D_{I_{G}}(M(F))
\end{equation}
defined by $\heckefunc{G}(\mathcal{T}[i], K)=\heckefunc{G}(\mathcal{T}, K)[i]$ for $i\in\mathbb{Z}$ and $\cT\in P_{I_{G}}(\flagvar{G})$. 
\par\medskip
Let $\ast: P_{I_G}(\Fl_G)\isom P_{I_G}(\Fl_G)$ be the covariant equivalence of categories induced by the map $G(F)\to G(F)$, $g\mapsto g^{-1}$. We may define the right action $\rh_G: D_{I_G}(\Fl_G)\times D_{I_G}(M(F))\to D_{I_G}(M(F))$ by $\rh_G(\cT, K)=\heckefunc{G}(\ast\cT, K)$. 
\par\medskip
\begin{exemple}
\label{Example}
Let $R,r\geq 0$ and $t^rM\subset V\subset t^{-R}M$ be an intermediate lattice stable under $I_G.$
Let $K\in{P_{I_{G}}}(M_{R,r})$ be a shifted local system on $V/t^{r}M\subset t^{-R}M/t^{r}M.$ We are going to explain the above construction explicitly in this case. Let $\mathcal{T}$ be in $D_{I_{G}}(\Flagstrat).$  Choose $r_{1}\geq r+s_{1}.$ If $g$ is a point in $\Flagstrat$ then $t^{r_{1}}M\subset gV.$ So we can define the scheme 
$$(V/t^ {r}M)\tilde{\times}\Flagstrat$$
as  the scheme classifying pairs $(g\I{G},m)$ such that $gI_{G}$ is an element of $\Flagstrat$ and $m$ is in $(gV)/(t^{r_1}M).$ For a point $(m,g)$ of this scheme, the element $g^ {-1}m$  lies in $V/t^{r}M.$ Assuming $s\geq R+r$ we get the digram 
$$M_{R+s_{2},r_{1}}\overset{p}{\longleftarrow}(V/t^{r}M)\tilde{\times}{\Flagstrat}\overset{act_{q,s}}{\longrightarrow} K_{s}\backslash (V/t^{r}M),$$
where $p$ is the map sending $(g\I{G},m)$ to $m$. For $gG(\locring)$ in ${Gr^{\theta}_G},$  the virtual dimension  \footnote{Recall that for an $\mathcal{O}$-sublattice $W\subset L(F)$,	its virtual dimension is $\dim(W):= \dim(W/W\cap L)-\dim(L/W\cap L)$} of $V/gV$ is $\langle\theta,\check{\mu}\rangle.$ The space $(V/t^ {r}M)\tilde{\times}\Flagstrat^{\theta}$ is locally trivial fibration over $\Flagstrat^{\theta}$ with fiber isomorphic to an affine space of dimension $\dim(V/t^{r_1}M)-\langle\theta,\check{\mu}\rangle.$ Since $K$ is a shifted local system, the tensor product 
$act_{q,s}^ {*}K \otimes pr_{2}^{*}\mathcal{T}$ is a shifted perverse sheaf. Let $K\tilde{\boxtimes}\mathcal{T}$ be the perverse sheaf $act_{q,s}^{*}K\otimes pr_{2}^{*}\mathcal{T}[\dim].$ The shift $[\mathrm{dim}]$ is the unique integer for which this complex is perverse and this shift depends on $\check{\mu}$. Then we have
$$
\heckefunc{G} (\mathcal{T},K)=p_{!}(K\tilde{\boxtimes}\mathcal{T}).
$$
\end{exemple}
\par\medskip
\paragraph{\textbf{Compatibility}}
Till now we have been working over an algebraic closed field and we have ignored all the Tate twists. Let us explain how our construction is compatible with the classical local theta correspondence when the ground field $\K$ is $\F_{q}$. 
\par\medskip
Assume temporarily that $\K=\F_{q}.$ Let us explain the relation between this geometrical convolution and classical convolution action in \cite{Minguez1} and \cite{MVW} at the level of functions. Given $K$ in 
$D_{I_{G}}(M(F))$ we can associate with it the following function $a_{K}$ in the Schwartz space $\mathcal{S}^{I_{G}}(M(F))$.  If $K$ is represented by the ind-pro-system $K_{N,r}$ in $D_{I_{G}}(M_{N,r})$ then for $m$ in $t^{-N}M_{0}(\locring)$ one has
$$
a_{K}(m)=Tr(Fr_{\overline{m}}, K_{N,r,\overline{m}})q^{\frac{rd}{2}},
$$
where $d=\dim M_{0},$  the point $\overline{m}$ is the image of $m$ in $M_{N,r}$, and $Fr_{\overline{m}}$ is the geometric Frobenius at $\overline{m}.$ For large enough $r$, this is independent of $r$. The Hecke functors on $D_{I_G}(M(F))$ defined above geometrize the action of the Hecke algebras on $\mathcal{S}^{I_G}(M(F))$ corresponding to the following
left action of $G(F)$ on $\mathcal{S}(M(F))$: for a point $g$ in $G(F)$ and a function $f$ in $\mathcal{S}(M(F))$ then
$$g.f(m)= \vert\det g\vert^{\frac{-1}{2}} f(g^{-1}m),$$
for any $m$ in $M(F).$ To any $\mathcal{T}$ in ${P_{I_{G}}(\flagvar{G})}$ one can associate a function on $G(F)/I_{G}$ given by 
$a_{\mathcal{T}}(x)=Tr(Fr_{x},\mathcal{T}_{x})$ for $x$ in $G(F)/I_{G}$. For $\mathcal{T}_{i}\in{P_{I_{G}}(\flagvar{G})}$, denoting by $f_{i}$ the corresponding function,  we have 
\[Tr(Fr_{g},(\mathcal{T}_{1}\star \mathcal{T}_{2})_{g})=\int_{x\in G(F)}f_{1}(x)f_{2}(x^{-1}g)dx,\]
 where  $dx$  is the Haar measure on $G(F)$ such that $I_{G}$ is of volume $1$. Now if $\mathcal{F}$ is in $D_{I_{G}}(M(F))$, let $K=\overset{\leftarrow }{H}_{G}(\mathcal{T},\mathcal{F})$ and denote by $f$ the function associated to $\mathcal{F}.$ Then the function $a_{K}$ associated to $K$ is
$$a_{K}(m)=\int_{x\in G(F)}\vert\det x\vert^{-\frac{1}{2}}f(x^{-1}m)a_{\mathcal T}(x)dx,$$
for any $m$ in $M(F).$
\par\bigskip
In the following (except  for \S $8$) we will restrict ourselves to the case of dual reductive pairs of type II. Let $L_{0}=\K^n$  and $U_{0}=\K^m$ with $n\leq m$ and let $G=\GL(L_{0})$ and $H=\GL(U_{0}).$ We put $\Pi_{0}=U_{0}\otimes L_{0}$, $L=L_{0}(\locring)$, $U=U_{0}(\locring),$ and $\Pi=\Pi_{0}(\locring).$ For any $\locring$-module of finite rank $M$ and any pair $N,r$ of  integers such that $N+r> 0,$ we set $M_{N,r}=t^{-N}M/t^{r}M.$ Let $T_{G}$ (resp. $T_{H}$) be the maximal torus of diagonal matrices in $G$ (resp. in $H$). Let $B_G$ (resp. $B_H$) be the Borel subgroup of upper-triangular matrices in $G$ (resp. $H$). Let $I_{G}$ and $I_{H}$ be the corresponding Iwahori subgroups.  Let $I_0$ denote the constant perverse sheaf on $\Pi$.
Using the previous construction we have 
the well-defined category of  $I_{G}\times I_{H}$-equivariant  perverse sheaves on $\Pi(F)$ inside the derived category $D_{I_{G}\times I_{H}}(\Pi(F))$ which is the geometrization of the invariants of the Schwartz space $\mathcal{S}(\Pi(F))^{I_{G}\times I_{H}}$,
and two Hecke functors corresponding to the actions of $DP_{I_{G}}(\flagvar{G})$ and  $DP_{I_{H}}(\flagvar{H})$ on $D_{I_{G}\times I_{H}}(\Pi(F)):$
$$\heckefunc{G}:DP_{I_{G}}(\flagvar{G})\times D_{I_{G}\times I_{H}}(\Pi(F))\longrightarrow D_{I_{G}\times I_{H}}(\Pi(F))$$
and 
$$\heckefunc{H}:DP_{I_{H}}(\flagvar{H})\times D_{I_{G}\times I_{H}}(\Pi(F))\longrightarrow D_{I_{G}\times I_{H}}(\Pi(F)).$$
%

\section{Structure of the category \texorpdfstring{$P_{H(\locring)\times I_{G}}(\Pi(F))$}{P_{H(\locring)\times I_{G}}(\Pi(F))}} 
\label{5}
The purpose of this section is to understand the module structure of the category$P_{H(\locring)\times I_{G}}(\Pi(F))$ under the action of  $P_{I_{G}}(\flagvar{G})$ and  $P_{H(\locring)}(Gr_H)$.  Let $U^*$ denote the dual of $U$.  A point $v$ in  $\Pi(F)$ may be seen as a $\locring$-linear map $v:U^{*}\to L(F).$ For $v$ in $\Pi_{N,r},$ let $U_{v,r}=v(U^*)+t^{r}L.$ Then $U_{v,r}$  is a $\locring$-module in $L(F).$ By identifying $Gr_G$ with the ind-scheme of lattices in $L(F)$, we may view $U_{v,r}$ as a point of the affine Grassmannian $Gr_G$. 
The Iwahori subgroup $I_G$  acts on the affine Grassmannian $Gr_{G}$ as well. The $I_G$-orbits are parametrized by cocharacters $\lambda$  in $X$. Each orbit is an affine space. We have the decomposition
\begin{equation}
\label{Iwahoriorbit}
G(F)=\bigsqcup_{\lambda\in{X}} I_{G}t^{\lambda}G(\locring).
\end{equation}
For any $\lambda$  in $X^{+},$  each  $G(\locring)$-orbit $Gr_{G}^{\lambda}$ decomposes into $I_G$-orbits which  are parametrized by $W.\lambda$ and the orbit $I_{G}t^{\lambda}G(\locring)$ is open in $Gr_{G}^{\lambda}$. For any $\lambda$ in $X$  denote by $O^{\lambda}$  the $I_G$-orbit through $t^{\lambda}G(\mathcal{O})$ in $Gr_{G}.$ Denote by $\overline{O^{\lambda}}$ its closure. The scheme $\overline{O^{\lambda}}$ is stratified by locally closed subschemes $O^{\mu},$ where $\mu$ is in $X.$  Remark that  $O^{\mu} \subset \overline{O^{\lambda}}$ does not necessarily imply $\mu\leq\lambda.$ Denote by $\mathcal{A}^{\lambda}$ the $\mathrm{IC}$-sheaf of $O^{\lambda}$ which is an object of $P_{I_{G}}(Gr_G).$

\begin{lemma}
\label{orbit}
The  set of $H(\locring)$-orbits on $\Pi_{N,r}$ identifies with the set of lattices $R$ such that $t^rL\subset{R}\subset{t^{-N}L}$ via the map sending $v$ to $U_{v,r}$. 
\end{lemma}
\begin{proof}
Let $M$ and $M^{'}$ be two free $\mathcal{O}$-modules of finite type. If $f_{1}$ and $f_{2}$ are two surjections from  $M$ to $M^{'},$ then there is $h$ in $Aut(M)$ such that $f_{1}\circ h=f_{2}$.
 Let us now consider two elements $v_{1}$ and $v_{2}$ of $\Pi_{N,r}$ such that $U_{v_{1},r}=U_{v_{2},r}.$
Adding to $v_{i}$ a suitable element $t^{r}\Pi$, we may assume that $v_{i}:U^{*}\to U_{v,r}$ is surjective for $i=1,2$.  Then the previous argument implies that there exists $h$ in $H(\mathcal{O})$ such that $v_{1}\circ h=v_{2}.$  Thus, for $v_{1}$ and $v_{2}$ in $\Pi_{N,r}$, the $H(\locring)$-orbits through $v_{1}$ and $v_{2}$ coincide if and only if  $U_{v_{1},r}=U_{v_{2},r}.$ Since $n\leq m$,  each lattice $R$  such that  $t^rL\subset{R}\subset{t^{-N}L}$ is exactly  of the form $U_{v,r}$ for some $v$ in $\Pi_{N,r}$.
\end{proof}

Let $\check{\omega}_{1}=(1,0\dots,0)$ be  the highest weight of the standard representation of $G$ and recall that $w_0$ is the longest element of the finite Weyl group $W_{G}$. 
\begin{lemma}
\label{Horbit}
There is a bijection $\lambda\to \Pi_{\lambda,r}$ between   $H(\locring)\times{I_G}$-orbits on $\Pi_{N,r}$ and elements  $\lambda$ in  $X_{G}$ such that for any $\nu$ in $W_{G}.\lambda$
\begin{equation}
\label{condition1}
\langle\nu,\check{\omega}_{1}\rangle\leq{r}\hspace{1cm} and\hspace{1cm}\langle w_0(\nu),\check{\omega}_{1}\rangle\leq N.
\end{equation}
Each orbit $\Pi_{\lambda,r}$  consists of points $v$ such that $U_{v,r}$ lies in $I_G t^ {\lambda} G(\locring)$.
\end{lemma}  

\begin{proof}
Any lattice $R$ satisfying $t^{r}L\subset R \subset t^{-N}L$ is of the form $U_{v,r}$ for some $v$ in $\Pi_{N,r}.$ Consider the lattice $U_{v,r}$ as a point in $Gr_G.$ Then by Lemma $\ref{orbit}$ the  $H(\locring)\times I_G$-orbits on $\Pi_{N,r}$ are exactly  the locally closed subschemes  $(\Pi_{\lambda,r})_{\lambda\in{X_{G}}}$  in $\Pi_{N,r}$ such that  $\lambda$ satisfies $\eqref{condition1}.$  
 \end{proof}
 
\par\medskip
For any  $\lambda$ in $X_G$, the perverse sheaves $\IC(\Pi_{\lambda,r})$ in $P_{H(\locring)\times I_{G}}(\Pi(F))$ are independent of the choice of $r$ if $\langle\nu,\check{\omega}_1\rangle <r$ for any $\nu$ in $W_G\lambda$. The resulting object of $P_{H(\locring)\times I_{G}}(\Pi(F))$ will be denoted by $\IC(\Pi_{\lambda})$. Hence we obtain the following:

\begin{proposition}
\label{irreducible}
 The irreducible objects of $P_{H(\locring)\times I_{G}}(\Pi(F))$ are in  bijection with $X_{G}$: the irreducible object corresponding to a cocharacter $\lambda$ in $X_{G}$ is the intersection cohomology sheaf $\IC(\Pi_{\lambda})$.
\end{proposition}

%

\begin{proposition}
\label{H(O)-equiv} 
For any $\lambda$ in $X_{G}$, the complex $\heckefunc{G}
(\mathcal{A}^{\lambda},I_{0})$ is canonically isomorphic to 
$\mathrm{IC}(\Pi_{\lambda}).$
\end{proposition}
This proposition implies that any irreducible object of 
$P_{H(\locring)\times I_{G}}(\Pi(F))$ is obtained 
by the action of $\mathcal{A}^{\lambda}$ on $I_{0}$ for some 
$\lambda$ in $X_{G}.$
We will give a proof of this proposition after some preparation. First remark that if $\lambda$ is dominant then $\mathcal{A}^{\lambda}$ is $G(\locring)$-equivariant and in this case Proposition $\ref{H(O)-equiv}$ results from \cite[Proposition 5]{Lysenko1}. 
\par\smallskip
Let us give  a  description of the complex $\heckefunc{G}(\mathcal{A}^{\lambda},I_{0})$.
Choose two integers $N,r$  satisfying $N+r >0$ such that  for any $\nu\in{W_{G}.\lambda},$ the condition $\eqref{condition1}$ be satisfied. Let $\Pi_{0,r}\tilde{\times }\overline{O^{\vphantom{'}}}^{\lambda}$ be the scheme classifying pairs $(v,gG(\locring))$, where  $gG(\locring)$ belongs to $\overline{O^{\vphantom{'}}}^{\lambda}$ and $v$ is an $\locring$-linear map from $U^*$ to $gL/t^{r}L.$ Let 
\begin{equation}
\label{equpi}
\pi:\Pi_{0,r}\tilde{\times}\overline{O^{\vphantom{'}}}^{\lambda}\longrightarrow \Pi_{N,r}
\end{equation}
be the map sending a pair $(v,gG(\locring))$ to the composition $$U^{*}\overset{v}{\longrightarrow}gL/t^{r}L{\hookrightarrow {t^{-N}L/t^{r}L}}.$$
This map is proper. 
The projection $p:\Pi_{0,r}\tilde{\times}\overline{O^{\vphantom{'}}}^{\lambda}\to\overline{O^{\vphantom{'}}}^{\lambda}$ is a vector bundle of rank $rnm-m\langle\lambda,\check{\omega}_{n}\rangle,$ where $\check{\omega}_{n}=(1,\dots,1).$  We obtain in this particular case an isomorphism
\begin{equation}
\label{directimage}
\heckefunc{G}(\mathcal{A}^{\lambda},I_0)\iso {\pi_{!}(\qelbar\tilde{\boxtimes}\mathcal{A}^{\lambda})},
\end{equation}
where the complex $\qelbar\tilde{\boxtimes}\mathcal{A^{\lambda}}$ is normalized to be perverse, i.e.
$$\qelbar\tilde{\boxtimes}\mathcal{A}^{\lambda}\iso {p^*{\mathcal{A}^{\lambda}}[\dim \mathrm{rel}(p)]}.$$

As mentioned before the category $P_{G(\locring)}(Gr_{G})$ is equipped with a convolution functor. Consider the following convolution diagram
\begin{equation}
\label{conv}
 Gr_{G}\times Gr_{G}\overset{p}{\leftarrow} G(F)\times Gr_{G}\overset{q}{\to} G(F)\times_{G(\locring)}Gr_{G}\overset{m}{\to}Gr_{G},
\end{equation}
where $m$ is the multiplication.
Let $\mathcal{F}_{1}$ and $\mathcal{F}_{2}$ be two $G(\locring)$-equivariant perverse sheaves over $Gr_{G}$, the convolution functor of these two perverse sheaves is by definition
 $\mathcal{F}_{1}\star \mathcal{F}_{2}=m_{!}(\mathcal{F}_{1}\tilde{\boxtimes}\mathcal{F}_{2}),$ where
the sheaf $\mathcal{F}_{1}\tilde{\boxtimes}\mathcal{F}_{2}$ is perverse equipped with an isomorphism 
 \begin{equation}
 \label{twisted}
 p^{*}(\mathcal{F}_{1}\boxtimes\mathcal{F}_{2})\iso q^{*}(\mathcal{F}_{1}\bt\mathcal{F}_{2}).
 \end{equation}
According to \cite[Proposition 6]{Gaitsgory} the category $P_{I_{G}}(Gr_{G})$ acts on $P_{G(\locring)}(Gr_{G})$ by convolution and  this convolution functor $\star$ preserves perversity.  We want to use this result in order to give a dimension estimate for the objects of $P_{I_G}(Gr_G).$ 
 \par\medskip
For  $\mu$ in $X_G^{+}$,  let $\mathcal{B}^{\mu}$  be the $\mathrm{IC}$-sheaf associated with the $G(\locring)$-orbit  $t^{\mu}G(\locring)$ in $Gr_G.$ Then, for any cocharacter $\lambda$ in $X_{G}$ the convolution functor $\mathcal{A}^{\lambda}\star \mathcal{B}^{\mu}$ is perverse. For any $\nu$ in $X_{G},$ and any  point $gG(\locring)$ in $O^{\nu}$, let $Y$ be the fiber of the map $m$ over this point. The fiber $Y$  identifies with the affine Grassmannian $Gr_{G}$. For $\eta$ in $X_{G}$ and $\delta$ in $X_{G}^{+},$ let $Y^{\eta,\delta}$ be the intersection of $Y$ with $I_{G}t^{\eta}G(\locring)\times_{G(\locring)}Gr_{G}^{\delta}$.
\par\medskip
The restriction of $\mathcal{A}^{\lambda}\star \mathcal{B}^{\mu}$ to $O^{\nu}$ is placed  in usual degrees smaller than or equal to $-\dim O^{\nu},$ and the restriction of $\mathcal{A}^{\lambda}\tilde{\boxtimes}\mathcal{B}^{\mu}\vert_{Y^{\eta,\delta}}$ is the constant complex  sitting in usual degrees smaller than or equal to $ -\dim O^{\eta}-\dim Gr_{G}^{\delta}.$

\begin{lemma}
\label{ineq2}
For any $\eta$, $\nu$ in $X_{G}$ and  any $\delta$ in $X_{G}^{+}$ the following inequality holds:
$$2\dim Y^{\eta,\delta}-\dim O^{\eta}-\dim Gr_{G}^{\delta}\leq -\dim O^{\nu}.$$
\end{lemma} 
 
\begin{proof}
Let $\mathcal{B}^{\delta,!}$ (resp. $\mathcal{A}^{\eta,!}$) be  the constant perverse sheaf on $Gr_{G}^{\delta}$ (resp.$O^{\eta}$) extended by zero with adequate perverse shift on $Gr_G.$  The extension by zero functor is right exact  for the perverse $t$-structure. Hence $\mathcal{B}^{\delta,!}$ (resp. $\mathcal{A}^{\eta,!}$) lies in non positive perverse degrees and so  does the convolution functor $\mathcal{A}^{\eta,!}\star \mathcal{B}^{\delta,!}.$  The $*$-restriction of $\mathcal{A}^{\eta,!}\tilde{\boxtimes}{\mathcal{B}^{\delta,!}}$ to $Y$ is the extension by zero from $Y^{\eta,\delta}$ to $Y$ of the constant complex. Hence this complex lies in degrees
$2\dim Y^{\eta,\delta}-\dim O^{\eta}-\dim Gr_{G}^{\delta}+\dim O^{\nu}$ and so we have the desired inequality. 
\end{proof}
\par\smallskip 
{\it Proof of Proposition $\ref{H(O)-equiv}:$} \  \  
Let $\lambda$ be in $X_{G}$ and consider the complex $\pi_{!}(\qelbar\tilde{\boxtimes}\mathcal{A}^{\lambda})$ appearing in $\eqref{directimage}.$  For $\nu$ in $X_{G},$ take a $H(\locring)\times I_G$-orbit $\Pi_{\nu,r}$ in $\Pi_{N,r}.$ If $v$ is in $\Pi_{\nu,r}$, let $Y_{v}$ be the fiber of  the map $\pi$  over $v$ defined in $\eqref{equpi}.$
 The fiber $Y_{v}$ is the scheme classifying elements $gG(\locring)$ in $\overline{O^{\vphantom{'}}}^{\lambda}$ such that $U_{v,r}$ is a sublattice of $gL.$ If $v$ is in $\Pi_{\lambda,r}$ then  $Y_{v}$ is just a point and so the map $\pi$  is an isomorphism over the open subscheme $\Pi_{\lambda,r}$. On one hand  this implies directly  that $\mathrm{IC}(\Pi_{\lambda,r})$ appears with multiplicity one in the complex of sheaves $\heckefunc{G}(\mathcal{A}^{\lambda},I_0).$ On the other hand this gives
$$\dim(\Pi_{\lambda,r})=rnm-m\langle\lambda,\check{\omega}_{n}\rangle+\dim O^{\lambda}.$$

Let  $U$ be  the open subscheme of $\Pi_{0,r}\tilde{\times}\overline{O^{\vphantom{'}}}^{\lambda}$ consisting of pairs $(v,gG(\locring))$ such that $gG(\locring)$ lies in $O^{\lambda}$ and $v:U^{*}\longrightarrow gL/t^ {r}L$ is surjective. The image of $U$ by $\pi$ is contained in $\Pi_{\lambda,r}.$ So, $\pi$ induces a surjective proper map  
$$
\pi_{\lambda}:\Pi_{0,r}\tilde{\times}\overline{ O^{\vphantom{'}}}^{\lambda}\longrightarrow \overline{\Pi^{\vphantom{'}}}_{\lambda,r}.
$$
For $v$ in $\Pi_{\nu,r},$ we stratify $Y_v$ by locally closed subschemes $Y_{v}^{\eta}$ indexed by cocharacters $\eta$ in $X_{G}.$ For any $\eta,$ the stratum $Y_{v}^{\eta}$  parametrizes elements  $gG(\locring)$ in $O^{\eta}.$
The $*$-restriction of $\qelbar\tilde{\boxtimes}\mathcal{A}^{\lambda}$ to $Y_{v}^{\eta}$ lives in usual degrees smaller than or equal to $-\dim O^{\eta}-rnm+m\langle\eta,\check{\omega}_{n}\rangle$  and the inequality is strict unless $\eta=\lambda.$  We will show that 
\begin{equation}
\label{desiredineq}
2\dim Y_{v}^{\eta}-\dim O^{\eta}-rnm+m\langle\eta,\check{\omega}_{n}\rangle\leq -\dim \Pi_{\nu,r}
\end{equation}
and that the inequality is strict unless $\nu=\lambda$, this would imply our claim.
Since we have $\dim(\Pi_{\nu,r})=rnm-m\langle\nu,\check{\omega}_{n}\rangle+\dim O^{\nu},$ the inequality $\eqref{desiredineq}$ becomes 
\begin{equation}
\label{ineq}
2\dim Y_{v}^{\eta}\leq m\langle\nu-\eta,\check{\omega}_{n}\rangle+\dim O^{\eta}-\dim O^{\nu}.
\end{equation}

Considering the map $\pi_{\eta}:\Pi_{0,r}\tilde{\times}{\overline{O^{\vphantom{'}}}^{\eta}}\longrightarrow \overline{\Pi^{\vphantom{'}}}_{\eta,r},$ we see that $\Pi_{\nu,r}\subset{\overline{\Pi^{\vphantom{'}}}_{\eta,r}}$. A dominant cocharacter $\delta$ in $X_{G}^{+}$ is called {\it very positive} if
$$
\delta=(b_1\geq\dots\geq b_{n}\geq 0).
$$
It is natural to stratify $Y_{v}^{\eta}$ by locally closed subschemes $Y_{v}^{\eta,\delta},$ where $\delta$ runs through very positive cocharacters. For any such $\delta,$ the stratum $Y_{v}^{\eta,\delta}$ consists of elements $(v,gG(\locring))$ such that the lattice $U_{v,r}$ is in $G(\locring)$-position $\delta$ with respect to the lattice $gL.$ 
For a point $(v,gG(\locring))$ of $Y_{v}^{\eta,\delta}$ the formula of virtual dimensions $\dim(L/gL)+\dim(gL/U_{v,r})=\dim(L/U_{v,r})$ gives
$$
\langle\delta+\eta-\nu,\check{\omega}_{n}\rangle=0.
$$
Finally the equation $\eqref{ineq}$ is equivalent to
$$2\dim Y_{v}^{\eta,\delta}\leq n\langle\delta,\check{\omega}_{n}\rangle+\dim O^{\eta}-\dim O^{\nu}.$$
 By using Lemma $\ref{ineq2}$ we are reduced to show that for any very positive $\delta,$ $\langle\delta,n\check{\omega}_{n}-2\check{\rho}_{G}\rangle\geq 0$. To prove this inequality notice that  
$$n\check{\omega}_{n}-2\check{\rho}_{G}=(1,3,5,\dots,2n-1).$$
Thus $n\check{\omega}_{n}-2\check{\rho}_{G}$ is very positive and so for any very positive cocharacter $\delta$ we have  $\langle\delta,n\check{\omega}_{n}-2\check{\rho}_{G}\rangle\geq 0$. This proves the inequality $\eqref{ineq}$. Moreover
for any very positive $\delta$ this inequality is strict unless $\delta=0$ which is the case if and only if $\nu=\eta.$ 
This finishes the proof. $\square$
\par\smallskip
Recall that according to the Satake isomorphism,  $P_{H(\locring)}(Gr_H)$ is equivalent to  the category $\mathrm{Rep}(\check{H})$ of representations of the Langlands dual group $\check{H}$ over $\qelbar$. The module structure of  $P_{H(\locring)\times G(\locring)}(\Pi(F))$ under the action of the category $P_{H(\locring)}(Gr_H)$ has been described in \cite[\S\ $5$]{Lysenko1}.  Namely, let  $U_{1}$ (resp. $U_{2}$) be the  vector subspace of $U_{0},$  generated by the first $n$ basis vectors (resp. by the last $m-n$ basis vectors) of $U_{0}$. Thus, $U_{0}=U_{1}\oplus U_{2}.$  Let $P\subset H$ be the parabolic subgroup preserving $U_{1}$. Let $M\,\isom\,\GL(U_{1})\times \GL(U_{2})$  the standard Levi factor in $P$ and  the map $\kappa: \check{G}\times \G_{\m}\to \check{H}$ be the composition
\begin{equation}
\label{kappa}
\check{G}\times \G_{\m}\overset{id\times 2\check{\rho}_{\GL(U_{2})}}\longrightarrow \check{G}\times \check{\GL}(U_{2})=\check{M}\hookrightarrow \check{H}.
\end{equation}

By using the extended Satake equivalence, 

$$\mathrm{gRes}^{\kappa}:P_{H(\locring)}(Gr_{H})\longrightarrow DP_{G(\locring)}(Gr_G)$$ 
for the functor corresponding to the restriction $\mathrm{Rep}(\check{H})\to \mathrm{Rep}(\check{G}\times \G_{\m})$ with respect to $\kappa$.  
 
\begin{proposition}\cite[Proposition 4]{Lysenko1}$ $
\label{lysenko}
The two functors 
$$P_{H(\locring)}(Gr_{H})\to D_{H(\locring)\times G(\locring)}(\Pi(F))$$ 
given by 
$$\mathcal{T}\to \heckefunc{H}(\mathcal{T},I_{0})\hspace{1cm} and \hspace{1cm}\mathcal{T}\to \heckefunc{G}(\mathrm{gRes}^{\kappa}(\mathcal{T}),I_{0})$$
are isomorphic. 
\end{proposition}  
\begin{proposition} 
\label{Pp_action_H(O)_I_G}
For any $\lambda$ in $X_G$ and $\cT$ in $P_{H(\locring)}(Gr_H)$ we have the following isomorphism
$$
\heckefunc{H}(\cT, \IC(\Pi_{\lambda}))\,\isom\, \heckefunc{G}(\cA^{\lambda}\star \mathrm{gRes}^{\kappa}(\mathcal{T}), I_0).
$$
\end{proposition}
\begin{proof} 
Since the actions of $P_{I_G}(\flagvar{G})$ and $P_{H(\locring)}(Gr_H)$ on $D_{H(\locring)\times I_{G}}(\Pi(F))$ commute, we get from Proposition~\ref{H(O)-equiv} and  \cite[Proposition 4]{Lysenko1}
\begin{eqnarray*}
 \heckefunc{H}(\mathcal{T},\mathrm{IC}(\Pi_{\lambda}))& \iso &\heckefunc{H}(\mathcal{T},\heckefunc{G}(\mathcal{A}^{\lambda},I_{0})) \\
&\iso & \heckefunc{G}(\mathcal{A}^{\lambda},\heckefunc{H}(\mathcal{T},I_{0}))\\
&\iso & \heckefunc{G}(\mathcal{A^{\lambda}}, \heckefunc{G}(\mathrm{gRes}^{\kappa}(\mathcal{T}),I_{0}))\\
&\iso & \heckefunc{G}(\cA^{\lambda}\star \mathrm{gRes}^{\kappa}(\mathcal{T}), I_0).
\end{eqnarray*}
\end{proof}

From Proposition~\ref{H(O)-equiv} it also follows that the functor 
\begin{equation}
\label{functor_maybe_eq_H(O)_I_G}
D_{I_{G}}(Gr_G)\to D_{H(\locring)\times I_{G}}(\Pi(F))
\end{equation}
given by $\cA\mapsto \heckefunc{G}(\cA, I_0)$ is exact for the perverse t-structures. It suffices to verify this for simple objects and this follows from Proposition $\ref{H(O)-equiv}$.   It is easy to see that neither of the categories $P_{I_{G}}(Gr_G)$ or $P_{H(\locring)\times I_{G}}(\Pi(F))$ is semi-simple. The functor (\ref{functor_maybe_eq_H(O)_I_G}) commutes with the actions of $P_{I_{G}}(\flagvar{G})$ by convolutions on the left. Let $P_{H(\locring)}(Gr_H)$ act on $D_{I_{G}}(Gr_G)$ via $\mathrm{gRes}^{\kappa}$ composed with the natural action of $D_{G(\locring)}(Gr_G)$ by convolutions on the right. According to Proposition~\ref{Pp_action_H(O)_I_G}, it is natural to expect that (\ref{functor_maybe_eq_H(O)_I_G}) commutes with the action of $P_{H(\locring)}(Gr_H)$. From Proposition~\ref{H(O)-equiv} and  Proposition \ref{lysenko} one derives the following:

\begin{theorem} 
\label{H(O)}
The functor $\eqref{functor_maybe_eq_H(O)_I_G}$ yields an isomorphism at the level of Grothendieck groups between $K(P_{I_{G}}(Gr_G))$ and  $K(P_{H(\locring)\times I_{G}}(\Pi(F)))$ commuting with the actions of $K(P_{H(\locring)}(Gr_H))$ and $K(P_{I_{G}}(\flagvar{G}))$.
\end{theorem}
\section{Simple objects of \texorpdfstring{$P_{I_{G}\times I_{H}}(\Pi(F))$}{P_{I_{G}\times I_{H}}(\Pi(F))}}
\label{6}
We use the same notation as in the previous section. 
Our goal is to describe the simple objects of  $P_{I_{H}\times I_{G}}
(\Pi(F))$. To 
do so we study the $I_{H}\times I_{G}$-orbits on $\Pi_{\lambda,r}$ defined in [\S $\ref{5}$, Lemma $\ref{Horbit}$]. It turns out that it 
is not necessary to do the study for all cocharacters $\lambda$. Indeed if 
$\lambda=(a_{1},\dots ,a_{n})$ we will restrict ourselves to the case where all  
$a_{i}$'s  are strictly smaller than $r$.    This will be sufficient for our purpose. 
Let $$\mathrm{Stab}_{\lambda}=\{ g\in{I_G}\hspace{1mm}\vert\hspace{1mm} g(t^{\lambda}L)=t^{\lambda}L \}$$
and
$$X_{N,r}^{\lambda}=\{v\in{\Pi_{N,r}}\hspace{1mm}\vert\hspace{1mm} U_{v,r}=t^{\lambda}L+t^ {r}L\}.$$
Describing $I_{H}\times I_{G}$-orbits on $\Pi_{\lambda,r}$ is equivalent to describe $I_{H}\times \mathrm{Stab}_{\lambda}$-orbits on $X_{N,r}^{\lambda}.$ 
\par\medskip

\textbf{Assume} $\mathbf{n=m}$.
\begin{lemma}
\label{or}
The $I_{H}\times \mathrm{Stab}_{\lambda}$- orbits on  $X_{N,r}^ {\lambda}$ are in bijection with  the finite Weyl group $W_{G}$.
\end{lemma} 
\begin{proof}
Let $\{e_1,e_2,\dots,e_n\}$ be the standard basis of the vector space $L_0$, so the Borel subgroup $B_{G}$ preserves  the standard flag associated with the basis $(e_{i})_{1\leq i\leq n}$.  Let $(u_1^ *,u_2^*\dots,u_{m}^*)$ be the standard basis of the dual space $U_0^ *$. 
Let $v$ be a point in $X_{N,r}^{\lambda}$ and consider the induced map 
\begin{equation}
\label{iso}
\overline{v}:U^*/tU^{*}\longrightarrow{U_{v,r}/(tU_{v,r}+t^{r}L})=t^{\lambda}L/t^ {\omega_{n}+\lambda}L,
\end{equation}
where $\omega_{n}=(1,\dots,1).$ The map $\overline v$ is an isomorphism and it may be considered as an element of $\mathrm{Aut}(t^{\lambda}L/t^{\omega_{n}+\lambda}L).$ Denote by
$$\dots \subset L_{-1}\subset L_{0}\subset L_{1}\subset \dots$$
the standard complete flag of lattices inside $L(F)$ preserved by the Iwahori group $I_{G}.$ For any $i$ in $\Z$ the images of $L_{i}\cap t^{\lambda}L$ in $t^{\lambda}L/t^{\omega_{n}+\lambda}L$ define a complete flag  which is preserved by $\mathrm{Stab}_{\lambda}.$ Thus the image of $\mathrm{Stab}_{\lambda}$ in 
$\mathrm{Aut}(t^{\lambda}L/t^{\omega_{n}+\lambda}L)$ is a Borel subgroup of $G$ but not necessary the standard one. Hence the  $I_H\times \mathrm{Stab}_{\lambda}$-orbits on the set of isomorphisms $\eqref{iso}$ are parametrized by the finite Weyl group $W_{G}.$ By Lemma $\ref{lemmatec}$ below each $I_{H}\times \mathrm{Stab}_{\lambda}$-orbit on $X_{N,r}^{\lambda}$ is the preimage of a $I_{H}\times \mathrm{Stab}_{\lambda}$-orbit on the scheme of isomorphisms $\eqref{iso}.$
Finally one gets that $I_{H}\times \mathrm{Stab}_{\lambda}$-orbits on $X_{N,r}^{\lambda}$ are exactly indexed by $W_{G}$.
\end{proof}
By this Lemma, the set of $I_{H}\times I_{G}$-orbits on $\Pi_{\lambda,r}$ is in bijection by $W_{G}$.

\begin{lemma}
\label{lemmatec}
Let $p,q$ be two integers such that $p\le q.$ Let $B$ be a free $\locring$-module of rank $p$ and $A$ be a free $\locring$-module of rank $q$. Let $v_{1},v_{2}:A\rightarrow B$ be surjective $\locring$-linear maps such that  for $i=1,2$ the induced maps $\overline{v}_{i}:A/tA\rightarrow B/tB$ coincide. Then there is $h\in{\GL(A)(\locring)}$ with $h=1$ $\mathrm{mod}$ $t$ such that $v_{2}\circ h=v_{1}.$
\end{lemma}

\begin{proof}
Let $A_i$ be the kernel of $v_i$ for $i=1,2.$ These are  free $\locring$-modules of rank $q-p.$  Choose a direct sum decomposition $A=A_i\oplus W_i,$ where $W_i$ is a free $\locring$-module of rank $p.$ Then there is a unique isomorphism $a:W_2\longrightarrow W_1$ such that $W_ 2\overset{a}{\longrightarrow} W_1\overset{v_1} \longrightarrow A$ coincides with $W_{2}\overset{v_2}{\longrightarrow}A.$ The images of $A_{i}\otimes_{\locring}k$ in $A\otimes_{\locring}k$ coincide, therefore there exists an isomorphism of $\locring$-modules $b:A_{2}\longrightarrow A_1$ such that $\overline b:A_{2}\otimes_{\locring}k\longrightarrow A_{1}\otimes_{\locring}k$ is identity. Then $a\oplus b $ is the desired map $h$.
\end{proof}


Let $\tau$ be an element of $W_{G}$ and let $w=t^{\lambda}\tau$ be the corresponding element in $\weyl{G}$, where $\lambda=(a_{1},\dots,a_{n})$. Denote by $\Pi_{N,r}^{w}$ the $I_{G}\times I_{H}$-orbit on $\Pi_{N,r}$ passing through $v$ given by 
$$v(u_{i}^{*})=t^{a_{\tau(i)}}e_{\tau(i)} \, \, \, \mathrm{for}\, \, i=1,\dots,n$$
The $I_{G}\times I_{H}$-orbit on $\Pi_{N,r}$ are exactly $\Pi_{N,r}^{w}$ for $w$ in $\weyl{G}$.  

For any $w$ in $\weyl{G}$ denote by $\mathcal{I}^{w}$  the IC-sheaf of  the $I_{H}\times I_{G}$-orbit $\Pi_{N,r}^{w}$ indexed by $w,$ and by $\mathcal{I}^{w!}$   the constant perverse sheaf on $\Pi_{N,r}^{w}$ extended by zero to $\Pi_{N,r}$. 
 As an object of $P_{I_H\times I_G}(\Pi(F))$, it is independent of $r$, so that our notation is unambiguous. We underline that this notation is only introduced under the assumption  $a_{i}<r$ for all $i$.
 As the category $P_{I_{H}\times I_{G}}(\Pi(F))$ is obtained by filtering inductive 2-limit, simple objects of this category are the image of simple objects of the pieces of the limit.  
\begin{proposition}
\label{r}
Recall that $n=m$. Any irreducible object of  $P_{I_{H}\times I_{G}}(\Pi(F))$ is of the form  $\mathcal{I}^{w}$ for some $w$ in $\weyl{G}.$
\end{proposition}

\begin{proof}
Let $\lambda=(a_1,\dots,a_{n})$.  An irreducible object of $P_{I_{H}\times I_{G}}(\Pi(F))$ is the $\IC$-sheaf of  an $I_{H}\times I_{G}$-orbit $\mathcal{Y}$ on $\Pi_{\lambda,r}$ for some integer $r$ and for some cocharacter $\lambda$ satisfying $\eqref{condition1}$.  In particular all $a_{i}$'s are smaller than or equal to $r$.  First we will show that we can restrict ourselves to the case where all $a_i$'s are strictly less than $r$. 
Assume that $a_{i}=r$ for some $i$. For $s>r$ consider the projection $q:\Pi_{N,s}\to \Pi_{N,r}.$ Then the $H(\locring)\times I_{G}$-orbit  $\Pi_{\lambda,s}$ is open in $q^{-1}(\Pi_{\lambda,r})$. The map $q:\Pi_{\lambda,s}\to \Pi_{\lambda,r}$ is not surjective but the sheaf $\IC(\mathcal{Y})$ is non-zero over the locus in $q^{-1}(\Pi_{\lambda,r})$ of maps $v:U^{*}\to t^{-N}L/t^{r}L$ whose geometric fiber of the image is of maximal dimension $n.$
Hence the IC-sheaf of $\mathcal{Y}$ is a also an IC-sheaf of some $I_{H}\times I_{G}$-orbit on $\Pi_{\lambda,s}$. We are reduced to the case where all $a_i$ are strictly less than $r$. Recall that the geometric fiber of an $\locring$-module $\mathcal{L}$ is $\mathcal{L}\otimes_{\locring}\K$.
\par\medskip
Next we are going to prove that each $I_{H}\times I_{G}$-equivariant local system on $\Pi_{N,r}^{w}$ is constant.  The map $  X_{N,r}^{\lambda}\to \mathrm{Isom}(U^{*}/tU^{*},t^{\lambda}L/t^{\lambda+\omega_{n}}L)$ given by $v\to \bar{v}$ is an affine fibration. The group $\mathrm{Hom}(U^{*},t^{\lambda+\omega_{n}}L/t^{r}L)$ acts freely and transitively on the fibers of this map. So we are reduced to show that any $B_{G}\times B_{H}$-equivariant local system on any $B_{H}\times B_{G}$-orbit on $U_{0}\otimes L_{0}$ is constant.  This is Indeed  true because  the stabilizer in $B_{G}$ of a point in the double coset $B_{G}wB_{G}/B_{G}$ for any $w$ in $ W_{G} $ is connected. 
\end{proof}
If $\lambda$ is dominant then the image of $\mathrm{Stab}_{\lambda}$ in $\mathrm{Aut}(t^{\lambda}L/t^{\lambda+\omega_{n}}L)$ is the standard Borel subgroup of $G.$ Thus when $w=t^{\lambda}$ with $\lambda$ being dominant we have that $\Pi_{N,r}^{w}$ is an open subscheme of  $\Pi_{\lambda,r}$  and $\mathcal{I}^{w}=\IC(\Pi_{\lambda,r}).$
\par\medskip
\textbf{Assume that} $\mathbf{n\leq m}$.

 In this case the map $\eqref{iso}$ is not an isomorphism but only  a surjection. We may consider the $I_H\times \mathrm{Stab}_{\lambda}$-orbits on the set of surjections $\eqref{iso}.$ Let $S_{n,m}$ be the set of pairs $(s,I_s),$ where $I_s$ is a subset of $n$ elements of $\{1,\dots,m\}$ and $s:I_{s}\longrightarrow \{1,\dots,n\}$ is a bijection.
Let $W_{1}\subset W_{2}\subset\dots \subset W_{m}=U_{0}^{*}$ be a complete flag preserved by $B_{H}$. We denote by  $\overline{W}_{i}$ the image of $W_{i}$ under the map $\eqref{iso}.$ Then 
$I_{s}=\{1\leq i\leq m \vert \dim \overline{W}_{i} > \dim \overline{W}_{i-1}\}.$

Recall that for $\lambda=(a_{1},\dots,a_{n})$ we assume that $a_{i}<r$ for all $i$. From Lemma $\ref{lemmatec}$ one deduce that each $I_{H}\times \mathrm{Stab}_{\lambda}$-orbit on $X_{N,r}^{\lambda}$ is the preimage of a $I_{H}\times\mathrm{Stab}_{\lambda}$-orbit on the set of surjections $\eqref{iso}$. Let
$w=(\lambda,s)$ be in $X_{G}\times S_{n,m}$ then the $I_H \times I_G$-orbit passing through $v$ a point of $\Pi_{N,r}$ is given by 
\begin{equation}
 \left\{ \begin{array}{ll}
\label{orbitn<m}
         v(u^{*}_i)=t^{a_{si}}e_{si} & \mbox{ for $i\in{I_s}$};\\
        v(u^{*}_i)=0 & \mbox{for $i\notin{I_s}$}.\end{array} \right.
        \end{equation}
 We denote this orbit by $\Pi_{N,r}^{w}$ and its closure by $\overline{\Pi}_{N,r}^{w}$. For any $w=(\lambda,s)$ in $ X_{G}\times S_{n,m}$ denote by $\mathcal{I}^{w}$ the IC-sheaf of $\Pi_{N,r}^{w}.$ The corresponding object of $D_{I_H\times I_G}(\Pi(F))$ is well-defined and independent of $N,r$.  Denote by $\mathcal{I}^{w!}$ the extension by zero of the constant perverse sheaf from $\Pi_{N,r}^{w}$ to $\Pi_{N,r}$. The corresponding object of $D_{I_{H}\times I_{G}}(\Pi(F))$ is well-defined and independent of $N$ and $r$. 

\begin{theorem}
 \label{irrec}
Any irreducible object of $P_{I_{H}\times I_{G}}(\Pi(F))$ is of the form $\mathcal{I}^{w}$ for some $w$ in  $X_{G}\times S_{n,m}.$
 \end{theorem} 
   
\begin{proof}
An irreducible object of $P_{I_{H}\times I_{G}}(\Pi(F))$ is the $\IC$-sheaf of  an $I_{H}\times I_{G}$-orbit $\mathcal{Y}$ on $\Pi_{\lambda,r}$ for some integer $r$ and for some cocharacter $\lambda=(a_{1},\dots,a_{n})$ satisfying $\eqref{condition1}$.  As in the proof  of  Proposition $\ref{r}$ we may assume that all  $a_{i}$'s are strictly less than $r$.
Consider a $I_{H}\times I_G$-orbit $\Pi_{N,r}^ {w}$ on $\Pi_{N,r}$ passing through $v$ as defined in  $\eqref{orbitn<m}.$ Let $St(v)$ be the stabilizer of $v$ in $I_{H}\times I_G.$ We are going to show that $St(v)$ is connected. This will imply that any $I_{H}\times I_G$-equivariant  local system on $I_{H}\times I_{G}$-orbit  $\Pi_{N,r}^{w}$ is constant.
\par\medskip
 The stabilizer $St(v)$ of $v$ is a subgroup of $I_{H}\times \mathrm{Stab}_{\lambda}.$ Let $B_{\lambda}$ be the image of $\mathrm{Stab}_{\lambda}$ in $\mathrm{Aut}(t^{\lambda}L/t^ {\lambda+\omega_{n}}L)$ then $B_{\lambda}$ is a Borel subgroup of $\mathrm{Aut}(t^{\lambda}L/t^ {\lambda+\omega_{n}}L).$ We define two groups  $I_{0,\lambda}$ and $I_{0,H}$ by the exact sequences 
 
 $$1\longrightarrow I_{0,\lambda}\longrightarrow \mathrm{Stab}_{\lambda} \longrightarrow B_{\lambda} \longrightarrow 1,$$
 and
$$1\longrightarrow I_{0,H} \longrightarrow I_{H} \longrightarrow B_{H} \longrightarrow 1.$$
Note that $I_{H}$  is semi-direct product of $I_{0,H}$ and $B_{H}.$
 Let $St_{0}(v)$ be the stabilizer of $v$ in $I_{0,H}\times I_{0,\lambda}.$  By Lemma  $\ref{lemmatec},$ the $I_{0,H}\times I_{0,\lambda}$-orbit through $v$ on $X_{N,r}^{\lambda}$ is the affine space of surjections 
 $f:U^*\longrightarrow t^{\lambda}L/t^ {r}L$ such that $f=v$ mod $t.$ Thus $St_{0}(v)$ is connected.
 Let $\overline{v}:U^{*}/tU^{*}\longrightarrow t^{\lambda}L/t^{\lambda+\omega_{n}}L$ be the reduction of $v$ mod $t.$ The   stabilizer  $St(\overline{v})$ of $\overline{v}$ in $B_{H}\times B_{\lambda}$ is connected. By Lemma $\ref{lemmatec}$ the reduction map from $St(v)$  to $St(\overline{v})$ is surjective. Using the following exact sequence
$$1 \longrightarrow St_0(v) \longrightarrow St(v) \longrightarrow St(\overline{v}) \longrightarrow 1$$
 we obtain that $St(v)$ is connected.

\end{proof}

\section{Study of Hecke functors, \texorpdfstring{$n=1$}{n=1} and \texorpdfstring{$m\geq 1$}{m\geq 1}}
\label{7}
We will assume that $n=1$ and $m\geq 1$ in the entire section and we will give a complete description of $DP_{I_H\times I_G}(\Pi(F))$  under the actions of  $P_{I_H}(\flagvar{H})$ et $P_{I_G}(\flagvar{G})$. We use the same notation as in previous section.  We will work most of the time over an algebraically closed field (and ignore the Tate twists). 
\par\medskip
 For $1\le i\le m$ we denote by $\omega_{i}=(1,\dots,1,0,\dots, 0)$ the cocharacter of $T_{H}$  where $1$ appears $i$ times. The  Iwahori group $I_{H}$ preserves $t^{-\omega_{i}}U$ and  $t^{\omega_{i}}U^{*}.$
Let $\Omega_{H}$ be the  normal subgroup in the affine extended Weyl group $\weyl{H}$  of elements of length zero. Note that  $\omega_m=(1,\dots,1)$ is in $\Omega_{H}.$  
 \par\medskip
 For $1\le i\le m,$ let $U^i=t^{-\omega_{i}}U$. Define $U^i$ for all $i\in\ZZ$ by the property that $U^{i+m}=t^{-\omega_m}U^i$ for all $i$. Thus,
$$
\dots \subset U^{-1}\subset U^{0}\subset U^{1}\subset\dots
$$
is the standard flag preserved by $I_H$. For any integer $k$ in $\Z$, we denote by $\IC^{k}$ the IC-sheaf of $U^k\otimes L$. 
\begin{proposition}
\label{ire1m}
The irreducible objects of $P_{\I{H}\times I_{G}}(\Pi(F))$ are exactly  the perverse sheaves $\IC^{k},$ $k\in\Z$. 
\end{proposition}
\begin{proof}
The assertion follows from Theorem $\ref{irrec}$.
\end{proof}
We will denote by $\IC^{k,!}$ the constant perverse sheaf on $U^{k}\otimes L - U^{k-1}\otimes L$ extended by zero. This is a (non irreducible) perverse sheaf.
 Denote by $I_{0}=\IC^0$  the constant perverse sheaf on $\Pi$.
 \par\medskip
Assume temporarily that $\K$ is finite. For $w\in \weyl{G},$ denote by $j_{w}$ the inclusion of $\flagvar{G}^{w}$ in $\flagvar{G},$ and let  $L_{w}=j_{w!*}\qelbar[\ell(w)](\ell(w)/2),$ the IC-sheaf of $\flagvar{G}^{w}.$  We write
 $L_{w!}=j_{w!}\qelbar[\ell(w)](\ell(w)/2)$ and $L_{w*}=j_{w*}\qelbar[\ell(w)](\ell(w)/2)$ for the standard and costandard objects. As $j_w$ is an affine map, both $L_{w!}$ and $L_{w*}$ are perverse sheaves. They satisfy $\mathbb{D}(L_{w*})=L_{w!},$ where $\mathbb{D}$ denotes the Verdier duality.  To each $\mathcal{G}$ in $P_{I_G}(\flagvar{G})$ we attach a function  $[\mathcal{G}]:G(F)/I_{G}\longrightarrow \qelbar$  given by for $x$  in $G(F)/I_{G}$ $[\mathcal{G}](x)=Tr(Fr_{x},\mathcal{G}_x),$  where $Fr_x$ is the geometric Frobenius at $x$.  The function $[\mathcal{G}]$ is an element of $\iwahorihecke{G}.$ In particular $[L_{w!}]=(-1)^{\ell(w)}q_{w}^{-1/2}T_w$ and $[L_{w*}]=(-1)^{\ell(w)}q_{w}^{1/2}T_{w^{-1}}^{-1},$ where $q_{w}=q^{\ell(w)}.$ Here $T_{w}$ denotes the characteristic function of the double coset $I_{G}wI_{G}.$
\par\bigskip
Let us describe $\heckefunc{H}(\mathcal{A}^{\lambda},I_{0}),$ for any  cocharacter $\lambda$ of $H.$ Recall that $\mathcal{A}^{\lambda}$ is the IC-sheaf of the $I_{H}$-orbit  $O^{\lambda}$  through $t^{\lambda}H(\locring)$ in $Gr_{H}.$ Let $\lambda=(a_{1}\dots,a_{m})$ and choose $N,r$ such that $-N\leq a_{i}<r$ for all $i.$ Let $\Pi_{0,r}\newtimes \overline{O}^{\lambda}$ be the scheme classifying pairs $(v,hH(\locring))$, where $hH(\locring)$ is a point in $\overline{O}^{\lambda}$ and   $v$ is a $\locring$-linear map $L^{*}\to hU/t^{r}U$. Let 
$$
\pi:\Pi_{0,r}\newtimes \overline{O}^{\lambda}\longrightarrow \Pi_{N,r}
$$
be the map sending $(v,hH(\locring))$ to the composition $L^{*}\overset{v}{\longrightarrow} hU/t^{r}U \longrightarrow t^{-N}U/t^{r}U.$
By definition we have 
$$
\heckefunc{H}(\mathcal{A}^\lambda, I_{0})\iso \pi_{!}(\qelbar\bt \mathcal{A}^{\lambda}),
$$
where $\qelbar\bt \mathcal{A}^{\lambda}$ is normalized to be perverse.
Denote by $p_{H}$ the projection of $\flagvar{H}\to Gr_H.$ Note that for any $\mathcal{T}$ in $P_{I_{H}}(\flagvar{H})$ we have
$$
\heckefunc{H}(\mathcal{T},I_{0})\iso \heckefunc{H}(p_{H!}(\mathcal{T}),I_{0}).
$$
For $1\leq i < m,$ let $s_{i}$ be the simple reflection (permutation) $(i,i+1)$ in $W_{H}.$

\begin{proposition}
\label{si}
For $1\le i<m$ we have
 $$\heckefunc{H}(L_{s_{i}},I_{0})\iso I_{0}\otimes \mathrm{R}\Gamma(\mathbb{P}^{1},\qelbar)[1]\iso I_{0}\otimes (\qelbar[1]\oplus \qelbar[-1]).$$
 Similarly, 
  $$\heckefunc{H}(L_{s_{i!}},I_{0})\iso I_{0}[-1].$$
\end{proposition}

\begin{proof}
One has $p_{H!}(L_{s_{i}})\iso \mathrm{R}\Gamma(\mathbb{P}^{1},\qelbar)[1]$ and the assertion follows.
\end{proof}

Assume that $m>1$ and let $s_{m}=t^{\lambda}\tau$, where $\lambda=(-1,0,\dots,0,1)$ and $\tau=(1,m)$ is the reflection corresponding to the highest root. This is the unique  affine simple reflection in $\weyl{H}$.

\begin{proposition}
\label{sm}
If $m>1$,  we have the following canonical isomorphisms 
$$
\heckefunc{H}(L_{s_{m}},I_{0})\iso \IC^{1}\oplus \IC^{-1} \hspace{2mm}\mathrm{and} \hspace{2mm} \heckefunc{H}(L_{s_{m}!},I_{0})\iso \IC^{1,!}\oplus \IC^{-1}
$$
\end{proposition}

\begin{proof}
The composition 
$$ \overline{\flagvar{H}}^{s_m}\hookrightarrow \flagvar{H}\overset{p_{H}}\longrightarrow Gr_H$$
is a  closed immersion and so  $p_{H!}(L_{s_{m}})\iso \mathcal{A}^{\lambda}.$ Thus we have 
$$\heckefunc{H}(L_{s_{m}},I_{0})\iso \heckefunc{H}(\mathcal{A}^{\lambda},I_{0}).$$
In this case the scheme $\overline{O}^{\lambda}$ classifies lattices $U^{'}$ such that
$$\dots \subset U^{-1}\subset U^{'}\subset U^{1}\subset \dots 
$$
and $\dim(U^{'}/U^{-1})=1.$ Let $N=r=1,$ then the image of the projection
$$\pi:\Pi_{0,1}\newtimes \overline{O}^{\lambda}\longrightarrow \Pi_{1,1}$$
is contained in $L\otimes (U^{1}/tU).$ Let $v$ be a map  from $L^{*}$ to $U^{1}/tU$ in the image of $\pi$. If $v$ factors through $U^{-1}/tU$ then the fiber of $\pi$ over the point $v$ is $\mathbb{P}^{1}$, otherwise it is a point. The first claim follows, the second is analogous.
\end{proof}

 In a similar way one gets the following. 
 
\begin{proposition} 
\label{actionsism}
 For for $1 \leq i \leq m,$ we have
$$\heckefunc{H}(L_{s_{i}},\IC^{i})\iso \IC^{i+1}\oplus \IC^{i-1}\hspace{2mm}\mathrm{and}\hspace{2mm}\heckefunc{H}(L_{s_{i}!},\IC^{i})\iso \IC^{i+1,!}\oplus \IC^{i-1}.$$ 
\end{proposition}

\begin{proof}
The proof follows from Lemma $\ref{si}$ and $\ref{sm}.$
\end{proof}
The symmetry in our situation is due to the fact that $\Omega_H$ acts freely and transitively on the set of irreducible objects of $P_{I_G\times I_H}(\Pi(F))$.
\par\medskip
For $1\leq i\leq m$ there is a unique permutation $\sigma_{i}$ in $W_{H}$ such that $t^{-\omega_{i}}\sigma_{i}$  is of length zero. Indeed, $\sigma_{i}$ is the permutation 
$$(1,2,\dots,m-i,m-i+1,\dots,m)\longrightarrow (i+1,i+2,\dots,m,1,\dots,i).$$ 
For $1\le i\le m$ we put $w_{i}=t^{-\omega_{i}}\sigma_{i}$. We extend this definition as follows: for any $i\in\ZZ$ let $w_i$ in $\Omega_H$ be the unique element such that $w_iU^r=U^{r+i}$ for any $r$. 
For $1\leq i\leq m-1$ we have $w_{1}s_{i}w_{1}^{-1}=s_{i+1}$ and $w_{1}s_{m}w_{1}^{-1}=s_{1}.$ Thus, the affine Weyl group of $H$ acts on the set $\{s_1,\ldots, s_m\}$ by conjugation. 
\begin{proposition}
\label{length0}
1) For any $i$ and $k$ in $\Z$ one has a canonical isomorphism
$$
\heckefunc{H}(L_{w_{i}},\IC^{k})\iso \IC^{k+i}.
$$
2) For $1\le i\le m$, $j\in\ZZ$ with $j\ne i \mod m$ one has
$$
\heckefunc{H}(L_{s_i}, \IC^j)\,\isom\, \IC^j\otimes 
(\Qlb[1]\oplus \Qlb[-1]).$$
\end{proposition}
Propositions $\ref{actionsism}$ and $\ref{length0}$ describe completely the action of $P_{I_{H}}(\flagvar{H})$ on the simple objects $\IC^{k}$, $k\in\Z$. Thus the module structure of $K(DP_{I_{H}\times I_{G}}(\Pi(F)))$ under the action of $K(P_{I_{H}}(\flagvar{H}))$. Now we are going to define the action of the center of $P_{I_{H}}(\flagvar{H})$ on $\IC^{k}.$ 
\par\medskip
Let $\sigma : \check{G}\times \mathbb{G}_{\m}\longrightarrow \check{H}$ be given by \eqref{kappa}. Denote by $\mathrm{Res}^{\sigma}:\mathrm{Rep}(\check{H})\longrightarrow \mathrm{Rep}(\check{G}\times \mathbb{G}_{m}),$ the corresponding geometric restriction functor. For any $G(\locring)$-equivariant perverse sheaf $\mathcal{T}$ on $Gr_{H}$,  $\mathcal{T}$ is naturally isomorphic to $p_{H!}(\mathcal{Z}(\mathcal{T})).$ 
\par\medskip
Denote by $s$ the standard representation of $\G_{m}$ and by $g$ the standard representation of $\check{G}$.
The category $\mathrm{Rep}(\check{G}\times \G_{\m})$ acts on $DP_{I_{G}\times I_{H}}(\Pi(F))$ as follows: 
\begin{equation}
\label{coeur}
\left\{
\begin{split}
&  \heckefunc{G}(s^{j},\IC^{k})\iso \IC^{k}[j].\\
&  \heckefunc{G}(g^{j},\IC^{k})\iso \IC^{k-mj}. 
\end{split}
\right.
\end{equation}
It follows that the representation ring $\mathrm{R}(\check{G}\times \G_{\m})$ acts on $K(DP_{I_{G}\times I_{H}}(\Pi(F))),$ which becomes in this way a free   $\mathrm{R}(\check{G}\times \G_{\m})$-module of rank $m$ with basis $\{\IC^{0},\dots,\IC^{m-1}\}.$

\begin{theorem}
\label{centre}
The respective actions of the center of $P_{I_{H}}
(\flagvar{H})$ and the center of $P_{I_{G}}(\flagvar{G})$ on 
the category $DP_{I_{G}\times I_{H}}(\Pi(F))$ are compatible. 
More precisely, the center of $P_{I_{H}}(\flagvar{H})$ acts via  
the geometric restriction functor $\mathrm{Res}^{\sigma}: 
\mathrm{Rep}(\check{H})\longrightarrow \mathrm{Rep}(\check{G}\times \mathbb{G}_{m})$ on the irreducible objects $\IC^{k}$ for any integer $k$.
\end{theorem}

\begin{proof}
Let us recall that there is a central functor  $$\mathcal{Z}:\pervsph{H}{H}\longrightarrow{P_{I_{H}}(\flagvar{H})}$$  constructed by Gaitsgory  in \cite[Theorem 1]{Gaitsgory}. For any $\mathcal{S}$ in $P_{H(\locring)}(Gr_{H})$, we have 
\begin{equation}
\label{center}
  \heckefunc{H}(\mathcal{Z}(\mathcal{S}),\IC^{0})\iso \heckefunc{H}(p_{H!}(\mathcal{Z}(\mathcal{S})),\IC^{0})\iso \heckefunc{H}(\mathcal{S},\IC^{0})\iso \heckefunc{G}(\mathrm{Res}^{\sigma}(\mathcal{S}),\IC^{0}),
 \end{equation}
where the last isomorphism is  \cite[Proposition 5]{Lysenko1}. 
Recall that for any $k$ in $\ZZ$ we have $\heckefunc{H}(L_{w_{k}},\IC^{0})\iso \IC^{k}.$ 
For $\mathcal{S}$ in $P_{H(\locring)}(Gr_{H})$, $\mathcal{Z}(\mathcal{S})$ is central so 
$$\heckefunc{H}(\mathcal{Z}(\mathcal{S}),\IC^{k})\iso \heckefunc{H}(L_{w_{k}},\heckefunc{H}(\cZ(\mathcal{S}),\IC^{0}))\iso \heckefunc{G}(\mathrm{Res}^{\sigma}(\mathcal{S}),\IC^{k}),$$
where the last isomorphism is from $\eqref{center}.$ The assertion follows. 
\end{proof}

\par\medskip

Assume that $\K$ is a finite field $\F_{q}$.  Let us  rewrite all useful  formulas obtained in Propositions $\ref{si}$ and $\ref{sm}$ and Proposition $\ref{length0},$  taking in consideration the Tate twists. These formulas will be used  in \S $\ref{9}.$   
\par\medskip
\begin{theorem}
\label{theon1m}
The bimodule $K(DP_{I_{G}\times I_{H}}(\Pi(F)))$ is free of rank $m$ over $\mathrm{R}(\check{G}\times \G_{\m})$ with basis $\{\IC^{0},\dots,\IC^{m-1}\}$ and the explicit action of $\affinehecke{H}$ is given by the following formulas: 

\begin{equation}
\label{coeurcoeur}
\left\{
\begin{split}
&\mathrm{For}\,  1\leq i\leq m:\,  \heckefunc{H}(L_{s_{i}},\IC^{i})\iso \IC^{i+1}\oplus \IC^{i-1}.\\
&\mathrm{For}\, 1\leq i\leq m:\, \heckefunc{H}(L_{s_{i!}},\IC^{i})\iso \IC^{i+1,!}\oplus \IC^{i-1}.\\
&\mathrm{If}\, j\neq i\, \mathrm{mod}\, m : \heckefunc{H}(L_{s_{i}},\IC^{j})\iso \IC^{j}(\qelbar[1](1/2)+\qelbar[-1](-1/2)).\\
&\mathrm{If}\,  j\neq i\, \mathrm{mod}\, m : \heckefunc{H}(L_{s_{i!}},\IC^{j})\iso \IC^{j}[-1](-1/2).\\
& \mathrm{For\,\,  any}\,  i\,\, \mathrm{and}\,\,  k\,\,  in\,\,  \Z:\, \,   \heckefunc{H}(L_{w_{i}},\IC^{k})\iso  \IC^{k+i}
\end{split}
\right.
\end{equation}
\end{theorem}

More generally, for $a<b$, denote by $\IC^{a,b,!}$ the sheaf $\qelbar[b-a]$ defined on $(U^{b}/U^{a})-\{0\}$ extended by zero to $U^{b}/U^{a}$. This is not perverse in general. In Grothendieck group $K(DP_{I_{H}\times I_{G}}(\Pi(F)))$ we have 
$$\IC^{a,b,!}= \IC^{b}-\IC^{a}[b-a].$$

Let $\omega_{i}=(1,\dots ,1,0,\dots, 0)$ where $1$ appears $i$ times and $0$ appears $m-i$ times. 
\begin{proposition}
We have a canonical isomorphism in $K(DP_{I_{H}\times I_{G}}(\Pi(F)))$:
$$\heckefunc{H}(L_{t^{\omega_{i}}!},\IC^{0})\iso \IC^{i-m,0,!}[i-\langle \omega_{i}, 2\check{\rho}_{H} \rangle]+\IC^{-m}[m-i-\langle\omega_{i},2\check{\rho}_{H}  \rangle].$$
\end{proposition}
\begin{proof}
First remark that 
$$\heckefunc{H}(L_{t^{\omega_{i}}},\IC^{0})\iso \heckefunc{H}(\mathcal{A}^{\omega_{i}!},\IC^{0}).$$
Let $N=r=1,$ the scheme $O^{\omega_{i}}$ classifies lattices $tU^{0}\subset U^{'}\subset U^{0}$ such that  $\dim(U^{'}/tU^{0})=m-i$ and $(U^{'}/tU^{0})\cap (U^{m-i}/tU^{0})=0$. Therefore the orbit $O^{\omega_{i}}$ is an affine space of dimension $\ell(t^{\omega_{i}})=\langle \omega_{i}, 2\check{\rho}_{H}\rangle=(m-i)i$. Let $\Pi_{0,1}\newtimes O^{\omega_{i}}$ be the scheme classifying pairs $(v,U^{'}),$ where  $U^{'}$ is  in $O^{\omega_{i}}$ and $v$ is a map from $L^{*}$ to $U^{'}/tU^{0}$. Consider the map 
$$\pi: \Pi_{0,1}\newtimes O^{\omega_{i}}\longrightarrow \Pi_{0,1}$$
sending $(v,U^{'})$ to $v$. Then we have 
$$\heckefunc{H}(\mathcal{A}^{\omega_{i}!},\IC^{0})\iso \pi_{!}\hspace{1mm}\IC(\Pi_{0,1}\newtimes O^{\omega_{i}})$$
and the assertion follows from the remark above on the elements $\IC^{a,b,!}$.
\end{proof}

\section{On the geometric local Langlands functoriality at the Iwahori level}
\label{conjecturegtrlf}
For basic notions in equivariant $K$-theory, we refer to \cite[Chapter 5]{CG}. Some of the constructions we will use are recalled in Appendix $\ref{Appendix1}$. Let us just recall the Kazhdan-Lusztig-Ginzburg isomorphism and fix some additional notation. 

Let $\K$ be  the finite field $\F_{q}$. Let $G$ be a connected reductive group over $\K$ and denote by $\check{G}$ its Langlands dual group over $\qelbar$.  Assume additionally that $[\check{G}, \check{G}]$ is simply connected.
Let $v$ be an indeterminate. Let $(W,S)$ be the Coxeter group associated with the root datum defined on $G$, where $W$ is the finite Weyl group and $S$ the set of simple reflections. The finite Hecke algebra $\mathbb{H}_{W}$ is free $\Z[v^{-1},v]$-algebra with basis $\{T_{w}, w\in{W}\}$ such that the following rules hold:
\begin{enumerate}
\item $(T_{s}+1)(T_{s}-v)=0$ if $s\in{S}$ is a simple reflection.
\item $T_{y}.T_{w}=T_{yw}$ if $\ell(yw)=\ell(y)+\ell({w})$.
\end{enumerate}
The group algebra $\Z[X]$ is isomorphic to $\mathrm{R}(\check{T}),$ the representation ring of the dual torus to $T$. We will write $e^{\lambda}$ for the element of $\mathrm{R}(\check{T})$ corresponding to the coweight $\lambda$ in $X.$
The affine extended Hecke algebra associated with $G$ was introduced by Bernstein \cite{bern} (it first appeared in \cite{Lu4}) and is isomorphic to the so-called Iwahori-Hecke algebra of a split $p$-adic group with connected center. The latter was introduced in \cite{Iwahori-mat} and reflects the  structure of the space $C_{c}(I_{G}\backslash G(F)/I_{G})$  of locally constant compactly supported $\qelbar$-valued functions on $G(F)$ which are bi-invariant under the action of $I_G.$  
The extended affine Hecke algebra $\affinehecke{G}$ is a free $\Z[v,v^{-1}]$-module with basis $\{e^{\lambda}T_{w}\vert w\in W , \lambda\in{X}\},$ such that:
\begin{enumerate}
\item
\label{HW} The $\{T_{w}\}$ span a sub-algebra of  $\affinehecke{G}$ isomorphic to $\mathbb{H}_{W}.$
\item
\label{RT}The elements $\{e^{\lambda}\}$ span a $\Z[v,v^{-1}]$-sub-algebra of  $\affinehecke{G}$ isomorphic to $\mathrm{R}(\check{T})[v^{-1},v].$
\item
\label{3}  For any $s_{\alpha}\in{S}$ with $\langle\lambda,\check{\alpha}\rangle=0$, $T_{s_{\alpha}}e^{\lambda}=e^{\lambda}T_{s_{\alpha}}.$
\item
\label{4}  For any $s_{\alpha}\in{S}$ with $\langle\lambda,\check{\alpha}\rangle=1$, $T_{s_{\alpha}}e^{s_{\alpha}(\lambda)}T_{s_{\alpha}}=ve^{\lambda}$.
\end{enumerate}

Properties $\eqref{3}$,$\eqref{4}$ together are equivalent to the following useful formula
\begin{equation}
\label{equationuseful}
T_{s_{\alpha}}e^{s_{\alpha}(\lambda)}-e^{\lambda}T_{s_{\alpha}}=(1-v)\frac{e^{\lambda}-e^{s_{\alpha}(\lambda)}}{1-e^{-\alpha}},
\end{equation}
where $\alpha$ is a simple coroot, $s_{\alpha}$ the corresponding simple reflection and $\lambda\in{X}$.
The properties $\eqref{HW}$ and $\eqref{RT}$ give us two canonical embeddings of algebras
$$\mathrm{R}(\check{T})[v^{-1},v]\hookrightarrow  \affinehecke{G}\hspace{1cm}\mathrm{and}\hspace{1cm} \mathbb{H}_{W}\hookrightarrow  \affinehecke{G}.$$
The multiplication in  $\affinehecke{G}$ gives rise to a $\Z[v^{-1},v]$-module isomorphism
$$ \affinehecke{G}\simeq \mathrm{R}(\check{T})[v^{-1},v]\otimes_{\Z[v^{-1},v]}\mathbb{H}_{W}.$$
This is a $v$-analogue of the $\Z$-module isomorphism \cite[7.1.8]{CG}, 
$$\Z[\weyl{G}]\simeq \mathrm{R}(\check{T})\otimes _{\Z} \Z[W_{G}].$$

\par\medskip

 Let $\check{\g}$ be the Lie algebra of $\check{G}$,  $\cB_{\check{G}}$ be the variety of Borel subalgebras in $\check{\g}$, and $\cN_{\check{G}}$ be the nilpotent cone in $\check{\g}$. The Springer resolution $\Ntilde{\check{G}}$ of $\cN_{\check{G}}$ is given by $$\Ntilde{\check{G}}=\{(x,\borel)\in\cN_{\check{G}}\times \cB_{\check{G}}\vert x\in\borel\}.$$ 
Let $\mu: \Ntilde{\check{G}}\to \cN_{\check{G}}$ be the Springer map. Let $s$ be  the standard coordinate on $\G_{\m}$. We let $\G_{\m}$ act on $\check{\g}$ by requiring that $s$ sends an element $x$ to $s^{-2}x$. We also define an action of  $\check{G}\times\G_{\m}$ on $\Ntilde{\check{G}}$ by the formula  
$$(g,s).(x,\gb)=(s^{-2}gxg^{-1}, g\gb g^{-1}).$$ 
The map $\mu$ is $\check{G}\times\G_{\m}$-equivariant. The Steinberg variety is defined by
$$Z_{\check{G}}=\Ntilde{\check{G}}\times_{\cN_{\check{G}}} \Ntilde{\check{G}}=\{(x,\borel, \gb')\in \cN_{\check{G}}\times \cB_{\check{G}}\times \cB_{\check{G}}
\, \vert \, x\in\borel\cap\borel^{'}\}.$$
The extended affine Hecke algebra $\mathbb{H}_G$ can be considered as a $\ZZ[s,s^{-1}]$-algebra, where $v=s^2$. Viewing $\ZZ[s,s^{-1}]$ as the representation ring of $\G_{m}$, one has the following result  due to Kazhdan-Lusztig-Ginzburg \cite[Theorem 7.2.5]{CG}: there is an isomorphism of natural $\ZZ[s,s^{-1}]$-algebras
\begin{equation}
\label{theoaffinehecke}
K^{\check{G}\times \G_{\m}}(Z_{\check{G}})\iso \affinehecke{G}.
\end{equation}
Let us explain briefly what we are going to do. Assume to be given two connected reductive groups $G,H$ and a homomorphism  $\check{G}\times\SL_2\to \check{H}$, where $\check{G}$ (resp., $\check{H}$) denotes the Langlands dual group of $G$ over $\qelbar$ (resp., of $H$). We still assume that the respective derived groups of $\check{G}$ and $\check{H}$ are simply connected. We construct  a bimodule over the affine extended  Hecke algebras $\HH_G$ and $\HH_H$ realizing the local geometric Arthur-Langlands functoriality at the Iwahori level for this homomorphism. We propose a definition of this explicit kernel at this level of generality given in Conjecture $\ref{firstconjecture}$. It is based to a large extent on the Kazhdan-Lusztig-Ginzburg isomorphism  $\eqref{theoaffinehecke}$. 
 \par\medskip
 We fix a maximal torus $T_{G}$ (resp. $T_{H}$) in $G$ (resp. $H$) and a Borel subgroup $B_{G}$ (resp. $B_{H}$) in $G$ (resp. $H$) containing $T_{G}$ (resp. $T_{H}$). Assume we are given a morphism
$$\sigma:\check{G}\times \SL_{2}\longrightarrow \check{H}$$
and let $\xi:\SL_{2}\to \check{H}$ be its second component and $\eta:\check{G}\to\check{H}$ be its first component. Let $\alpha:\mathbb{G}_{m}\to \mathrm{SL_{2}}$ be the standard maximal torus sending an element $x$ to $\mathrm{diag}(x,x^{-1})$. Let $\sigma:\check{G}\times \G_{m}\to \check{H}$ be  the restriction of the above homomorphism via $\id\times \alpha$:

\begin{equation}
\label{dualmap}
\xymatrix{
\check{G}\times\mathbb{G}_{m} \ar[r]^{\mathrm{id}\times\alpha} &\check{G}\times\mathrm{SL}_{2} \ar[r]^{\eta\times \xi}&\check{H}}.
\end{equation}

For any element  $g$ in $\check{G}$ we will often denote  its image $\eta(g)$ in $\check{H}$ by the same letter $g$ as well as for the linearised morphisms between  the corresponding Lie algebras.  
Denote by 
$$\overline{\sigma}:\check{G}\times\G_{\m}\longrightarrow\check{H}\times\G_{\m}$$
the morphism whose first component is $\sigma$ and whose second component is the second projection $pr_{2}:\check{G}\times\G_{\m}\longrightarrow\G_{\m}.$ 
 The representation ring $\R(\check{G}\times\G_{\m})$  is isomorphic to $\R(\check{G})[s,s^{-1}]$. 
Remark that (at least for pairs $(\mathrm{S}\mathbb{O}_{2n},\mathrm{Sp}_{2m})$ and $(\GL_{n},\GL_{m})$) according to \cite{Lysenko1}, the local Langlands functoriality at the unramified level sends the unramified representation with Langlands parameter $\gamma$ in $\check{G}$ to the unramified representation with Langlands parameter $\sigma(\gamma, q^{1/2})$ of $\check{H}.$ This is realized by the restriction homomorphism 
$\mathrm{Res}^{\sigma}:\mathrm{Rep}(\check{H})\longrightarrow \mathrm{Rep}(\check{G}\times\G_{\m})$
 induced by $\sigma.$  
\par\medskip 
On the one hand, it is understood that the standard representation $s$ of $\G_{\m}$ corresponds to the cohomological shift $-1$ in order to have the compatibility with \cite{Lysenko1}.  On the other hand while specializing $s$, we should think of $s$ as $q^{1/2}$ to makes things compatible with the theory of automorphic forms.  
\par\medskip
Let $e$ denote the standard nilpotent element  of  $\mathrm{Lie}(SL_{2})$
$$
e=\left( {\begin{array}{cc}
 0 & 1  \\
 0& 0 \\
\end{array} } \right).$$
  If $d\xi:\mathrm{Lie}(SL_{2})\to \mathrm{Lie}(\check{H})$ is the linearised morphism associated to $\xi$, we denote $d\xi(e)$ by $x$. As $\sigma$ is a group morphism, for any $z$ in $\g$, $[d\eta(z),x]=0$. Hence, if $z$ is nilpotent, so is $d\eta(z)+x$.

\begin{lemma}
\label{mapf}
The map  $f$ from $\N{\check{G}}$ to $\N{\check{H}}$ sending any element $z$ in  $\N{\check{G}}$  to $z+x$  is a $\overline{\sigma}$-equivariant map. It defines a morphism of stack quotients 
\begin{equation}
\label{mapfbar}
\overline{f}:\N{\check{G}}/(\check{G}\times \G_{\m})\longrightarrow \N{\check{H}}/(\check{H}\times \G_{\m}).
\end{equation}
\end{lemma}

\begin{proof}
 We have the following equality in $\mathrm{Lie(SL}_{2})$
 \begin{equation}
\label{sl2}
ses^{-1}=s^{2}e.
\end{equation}

This implies that $s^{-2}\xi(s)x\xi(s)^{-1}=x$. For $(g,s)$ in $\check{G}\times \G_{\m},$ let $(h,s)=\overline{\sigma}(g,s)=(g\xi(s),s)$. Then for any $z$ in $\N{\check{G}}$ 
$$s^{-2}gzg^{-1}+x=s^{-2}h(z+x)h^{-1},$$
which implies that $f$ is $\overline{\sigma}$-equivariant and the morphism of stack quotients $\overline{f}$ is well-defined. 
\end{proof}

The Springer map $\Ntilde{\check{H}}\longrightarrow \N{\check{H}}$ is $(\check{H}\times\G_{\m})$-equivariant. By using this and  Lemma $\ref{mapf}$ we obtain the following diagram:
 \[
\xymatrix @R=2cm{
\mathcal{X}=(\Ntilde{\check{G}}/(\check{G}\times\G_{\m}))\times_{\N{\check{H}}/(\check{H}\times\G_{\m})}(\Ntilde{\check{H}}/(\check{H}\times\G_{\m})) \ar[r] \ar[d]& \Ntilde{\check{G}}/(\check{G}\times\G_{\m}) \ar[d] \\
\Ntilde{\check{H}}/(\check{H}\times\G_{\m})  \ar[r] & \N{\check{H}}/(\check{H}\times\G_{\m}), 
}
\]
where the bottom horizontal map is induced from the Springer map for $\check{H}$ and the vertical right arrow is the composition of  the $\check{G}\times \G_{\m}$-equivariant Springer map for $\check{G}$ with the map $\overline{f}$ defined in Lemma $\ref{mapf}.$
Note that in the left top corner of the diagram we took the fiber product in the sense of stacks, see \cite[\S 2.2.2]{Laumon}, we denoted it by $\mathcal{X}.$
The $K$-theory  $K(\mathcal{X})$ of $\mathcal{X}$  is naturally a module over the associative  algebras 
$K^{\check{G}\times\G_{\m}}(\Ntilde{\check{G}}\times_{\N{\check{G}}} \Ntilde{\check{G}})$ and $K^{\check{H}\times\G_{\m}}(\Ntilde{\check{H}}\times_{\N{\check{H}}} \Ntilde{\check{H}}).$ The action is by convolution (see Section $\ref{A}$ Appendix $\ref{Appendix1}$).  Thanks to $\eqref{theoaffinehecke}$, these two algebras may be identified with the  extended affine Hecke algebras $\affinehecke{G}$ and $\affinehecke{H}$ respectively.
 We may now state the conjecture : 

\begin{conjecture}
\label{firstconjecture}
The bimodule over the affine extended Hecke algebras 
\\
$K^{\check{G}\times\G_{\m}}(Z_{\check{G}})$ and $K^{\check{H}\times\G_{\m}}(Z_{\check{H}})$ realizing the local geometric Langlands functoriality at the Iwahori level for the map $\sigma:\check{G}\times \G_{\m}\longrightarrow \check{H}$ identifies with $K(\mathcal{X}).$
\end{conjecture}
\par\medskip
Remark that if $\check{G}=\check{H}$ and the map $\xi$ is trivial, then $\mathcal{X}$ equals $Z_{\check{G}}$ and $K(\mathcal{X})$ identifies with the  extended affine Hecke algebra $\affinehecke{G}$  for $G.$ Thus $K(\mathcal{X})$  is naturally a free module of rank one over both algebras $\affinehecke{H}$  and $\affinehecke{G}$. 

\subsection{Properties of the stack $\cX$}

 Consider the induced variety
$$\N{\check{G},\check{H}}=(\check{H}\times \G_{\m})\times_{\check{G}\times \G_{\m}}\N{\check{G}}$$
 with respect to $\overline{\sigma}$ (see Appendix $\ref{Appendix1}$ for the exact definition of the induced variety).  Similarly  consider the induced variety
$$\Ntilde{\check{G},\check{H}}=(\check{H}\times \G_{\m})\times_{\check{G}\times \G_{\m}}\Ntilde{\check{G}}.$$

\begin{proposition} 
\label{jj} 
There exists a natural isomorphism of stacks 
$$\mathcal{X}\iso(\Ntilde{\check{G},\check{H}} \times_{\N{\check{H}}}\Ntilde{\check{H}})/(\check{H}\times\G_{\m}),$$
and so an isomorphism of K-groups
$$K(\mathcal{X})\iso K^{\check{H}\times \G_{m}}(\Ntilde{\check{G},\check{H}} \times_{\N{\check{H}}}\Ntilde{\check{H}}).$$
\end{proposition}

\begin{proof}
 Since the map $f$ defined in Lemma $\ref{mapf}$ is $\overline{\sigma}$-equivariant, it induces a $\check{H}\times \G_{\m}$-equivariant map
\begin{equation}
\label{mapf1}
f_{1}:(\check{H}\times \G_{\m})\times_{\check{G}\times\G_{\m}}\N{\check{G}}\longrightarrow \N{\check{H}}.
\end{equation}
The map $f_{1}$ in $\eqref{mapf1}$ induces a map from $\Ntilde{\check{G},\check{H}}$ to $\N{\check{H}}$ and we can consider the fiber product $\Ntilde{\check{G},\check{H}} \times_{\N{\check{H}}}\Ntilde{\check{H}}.$
Note that $\Ntilde{\check{G},\check{H}}/(\check{H}\times \G_{\m})$ is isomorphic to the stack quotient $\Ntilde{\check{G}}/(\check{G}\times \G_{\m}),$ see Appendix \S $\ref{C}$ in $\ref{Appendix1}$. It follows that $\mathcal{X}$ identifies with the stack quotient of $\Ntilde{\check{G},\check{H}} \times_{\N{\check{H}}}\Ntilde{\check{H}}$ by the  action of $\check{H}\times\G_{\m}$ thanks to the following general fact :
if $\phi:X\longrightarrow Z$ and $\psi:Y\longrightarrow Z$ are  equivariant morphisms of $G$-schemes, then the fiber product $X/G\times_{Z/G}Y/G$ in the category of stacks identifies with  the quotient stack $(X\times_{Z}Y)/G.$

 \end{proof}
The action of $K^{\check{H}\times \G_{m}}(\Ntilde{\check{H}}\times_{\N{\check{H}}}\Ntilde{\check{H}})$ and $K^{\check{G}\times \G_{m}}(\Ntilde{\check{G}}\times_{\N{\check{G}}}\Ntilde{\check{G}})$ by convolution on  $K^{\check{H}\times \G_{m}}(\Ntilde{\check{G},\check{H}}\times_{\N{\check{H}}}\Ntilde{\check{H}})$   is defined in sections $\ref{B}$ and $\ref{C}$ of Appendix $\ref{Appendix1}$.
\par\medskip
If the map $\sigma$ is an inclusion of $\check{G}$ in $\check{H}$, the natural map 
$$\check{H}\times_{\check{G}}\N{\check{G}}\to (\check{H}\times \G_{\m})\times_{\check{G}\times \G_{\m}}\N{\check{H},\check{G}}=\N{\check{G},\check{H}}$$
is an isomorphism.  We can identify $\N{\check{G},\check{H}}$ with the variety of pairs  
$$(h\check{G}\in \check{H}/\check{G}, v\in \N{\check{H}})$$
satisfying   $h^{-1}vh\in x+\N{\check{G}}$ via the map sending any element of  $(h,z)$ of  $\check{H}\times \N{\check{G}}$ to $(h\check{G}, v=h(z+x)h^{-1}).$ The latter map makes sense because  $\check{G}$ centralizes $x$.  Thus the map $f_{1}$ $\eqref{mapf1}$ becomes the projection sending any element $(h\check{G},v)$ of  $\N{\check{G},\check{H}}$ to $v$. In this case the left $\check{H}\times \G_{\m}$-action on $\N{\check{G},\check{H}}$ is such that  for any $(h_{1},s)$ in $\check{H}\times \G_{\m}$ and any $(h\check{G},v)$ in
$\N{\check{G},\check{H}}$, 
$$(h_{1},s).(h\check{G},v)=(h_{1}h\xi(s)^{-1}\check{G},s^{-2}h_{1}vh_{1}^{-1}).$$
\par\medskip

\begin{proposition}
There is a natural isomorphism
$$K(\mathcal{X})\iso K^{\check{G}\times \G_{\m}}(\Ntilde{\check{G}}\times_{\N{\check{H}}}\Ntilde{\check{H}}),$$
and the $\mathrm{R}(\check{H}\times\G_{\m})$-module structure on the right hand side is defined by the functor $\mathrm{Res}^{\overline{\sigma}}:\mathrm{R}(\check{H}\times\G_{\m})\to\mathrm{R}(\check{G}\times\G_{\m}).$
\end{proposition}
\begin{proof}
The scheme $\Ntilde{\check{G}}\times_{\N{\check{H}}}\Ntilde{\check{H}}$ classifies couples $((z,\borel_{1}),\borel),$ where  $(z,\borel_{1})$ lies in $\Ntilde{\check{G}}$ and $\borel$ is Borel subalgebra  in $\mathrm{Lie(H)}$ containing  $z+x$.
We define an action of $\check{G}\times \G_{\m}$ on $\Ntilde{\check{G}}\times_{\N{\check{H}}}\Ntilde{\check{H}}$ as follows: 
for any $(g,s)$ in $\check{G}\times \G_{\m}$ and any $((z,\borel_{1},\borel)$ in $\Ntilde{\check{G}}\times_{\N{\check{H}}}\Ntilde{\check{H}}$ 
$$(g,s).((z,\borel_{1}),\borel)=(s^{-2}gzg^{-1},g\borel_{1}g^{-1}, g\xi(s)\borel\xi(s)^{-1}g^{-1}).$$
By Lemma $\ref{gg}$ in Appendix $\ref{Appendix1}$  we have an  $\check{H}\times\G_{\m}$-equivariant isomorphism
$$(\check{H}\times\G_{\m})\times_{\check{G}\times\G_{\m}}(\Ntilde{\check{G}}\times_{\N{\check{H}}}\Ntilde{\check{H}})\iso \Ntilde{\check{G},\check{H}}\times_{\N{\check{H}}}\Ntilde{\check{H}}.$$
Combining this with Proposition $\ref{jj}$ we get the desired isomorphism.
\end{proof}
\par\medskip
In the rest of this section  we will restrict ourselves to the case of $G=\GL_{n}$ and $H=\GL_{m}$ and we will describe  some additional properties of the bimodule $K(\mathcal{X})$, namely a filtration and a grading on $K(\mathcal{X})$, where the graded parts will just be some equivariant K-theory of Springer fibers. 
We will always use the same notation for $\GL_{r}$ and its Langlands dual over $\qelbar.$  In this setting we choose the morphism $\eta$ to be the canonical inclusion of $\GL_{n}$ into $\GL_{m}.$ The map $\sigma$ is obtained by the composition 
$$\GL_{n}\times \G_{\m}\to\GL_{n}\times \mathrm{SL_{2}}\overset{id\times \xi}{\longrightarrow }\GL_{n}\times \GL_{m-n}\longrightarrow \GL_{m},$$ 
where the last arrow is the inclusion of the standard Levi subgroup associated to the partition $(n,m-n)$ of $m$ and  $\xi$ corresponds to the principal unipotent orbit as in \cite{Arthur}. Then the restriction of the map  $\xi$ to $\G_{\m}$ is the cocharacter  $(0,\dots,0,m-n-1,m-n-3,\dots,1+n-m).$ 
Let $U_{0}=\K^{m}$ be the standard representation of $\GL_{m},$ and $\{u_{1},\dots,u_{m}\}$ be the standard basis of $U_{0}.$   The element  $x=d\xi(e)$ is a nilpotent element  of $\mathrm{Lie}(\GL_{m})$ such that  $x(u_{i})=0$ for $1\leq i\leq n+1$ and that $x(u_{i+1})=u_{i}$ for $n+1\leq i<m.$ Let  $G_{2}=\GL_{m-n}$ and $B_{2}$ be the unique Borel subgroup in $G_2$ such that $x$ lies in $\mathrm{Lie}(B_2).$ 
\par\medskip
Let $Z_{G_{2}}(x)$ be the stabilizer  of $x$ in $G_{2}$. It acts naturally on $\Ntilde{\check{G}}\times_{\N{\check{H}}}\Ntilde{\check{H}}:$ for any $y$ in $Z_{G_{2}}(x)$  and any  $(z,\borel_{1},\borel)$ in $\Ntilde{\check{G}}\times_{\N{\check{H}}}\Ntilde{\check{H}},$ 
$$y.(z,\borel_{1},\borel)=(z,\borel_{1},y\borel y^{-1}).$$
For any $s$ in $\G_{\m}$, the element $\xi(s)$ clearly normalizes $Z_{G_{2}}(x)$ and  the semi-direct product $Z_{G_{2}}(x)\rtimes \G_{\m}$ is a subgroup of $G_{2}.$
 The group $Z_{G_{2}}(x)\rtimes \G_{\m}$ acts on  $\Ntilde{\check{G}}\times_{\N{\check{H}}}\Ntilde{\check{H}}$ and this action commutes with the $\check{G}$-action. 
\par\medskip
\begin{theorem}
There exists a $\check{G}\times \G_{m}$-invariant filtration 
$$\emptyset=F^{0}\subset F^{1}\subset \ldots \subset F^{r}=\Ntilde{\check{G}}\times_{\N{\check{H}}}\Ntilde{\check{H}}$$
such that for $0\leq i\leq r$, each $K^{\check{G}\times \G_{m}}(F^{i})$ is a submodule over both affine extended Hecke algebras $K^{\check{G}\times G_{m}}(\Ntilde{\check{G}}\times_{\N{\check{G}}}\Ntilde{\check{G}})$ 
and $K^{\check{H}\times G_{m}}(\Ntilde{\check{H}}\times_{\N{\check{H}}}\Ntilde{\check{H}}).$
Moreover the spaces $K^{\check{G}\times \G_{m}}(F^{i})$ for $0\leq i\leq r$ define a filtration on $K(\mathcal{X})$.  
\end{theorem}
\begin{proof}
For any $\check{G}$-orbit $\Obb$ on  $\N{\check{G}}$ we denote by $Y_{\Obb}$  the preimage of $\Obb$ in $\Ntilde{\check{G}}\times_{\N{\check{H}}}\Ntilde{\check{H}}$ under the projection 
$$\Ntilde{\check{G}}\times_{\N{\check{H}}}\Ntilde{\check{H}}\to \N{\check{G}}$$
sending $(z,\borel_{1},\borel)$ to $z.$  We  refer the reader  to \cite[\S\ 3.2]{CG} for details on nilpotent orbits and stratification of the nilpotent cone $\N{\check{G}}$ into $\check{G}$-conjugacy classes and the stratification of the Steinberg variety of $\check{G}$. The orbits $Y_{\Obb}$ form a $\check{G}\times \G_{\m}$-invariant stratification of $\Ntilde{\check{G}}\times_{\N{\check{H}}}\Ntilde{\check{H}},$ which is also $Z_{G_{2}}(x)$-invariant.  The $\check{G}$-orbit  $\Obb$ is given by a partition $\theta=(n_{1}\geq n_{2}\geq \dots\geq n_{r}\geq 1)$ of $n$.  Let $M_{\theta}$ denote the standard Levi subgroup corresponding to this partition, namely
$$M_{\theta}\iso \GL_{n_{1}}\times \dots \times \GL_{n_{r}}.$$
We denote by $z_{\theta}$   the standard upper triangular  regular nilpotent element in $\mathrm{Lie}(M_{\theta})$;   $z_{\theta}$ lies in the orbit  $\Obb$.  Let $Z_{\theta}$  be the  stabilizer of $z_{\theta}$ in $\check{G}\times \G_{\m}$, $Z_{\theta}$ is connected. Denote by $\B_{\check{G},\theta}$  the preimage of $z_{\theta}$ under the Springer map $\Ntilde{\check{G}}\to\N{\check{G}}$. Let $\B_{\check{H},\theta}$ be the preimage of $z_{\theta}+x$ under the Springer map $\Ntilde{\check{H}}\to\N{\check{H}}.$ We have an isomorphism
$$(\check{G}\times\G_{\m})\times_{Z_{\theta}}(\B_{\check{G},\theta}\times \B_{\check{H},\theta})\iso Y_{\Obb}$$
sending $(g,s,\borel_{1},\borel)$ to $(s^{-2}gz_{\theta}g^{-1},g\borel_{1}g^{-1},g\xi(s)\borel\xi(s)^{-1}g^{-1}).$
Hence  we have an isomorphism of groups 
\begin{equation}
\label{obb}
K^{\check{G}\times \G_{\m}}(Y_{\Obb})\iso K^{Z_{\theta}}(\B_{\check{G},\theta}\times \B_{\check{H},\theta}).
\end{equation}

According to \cite{Spal} the scheme $\B_{\check{G},\theta}$ and $\B_{\check{H},\theta}$ respectively admit a finite paving by affine spaces stable under the action of $Z_{\theta}$.  Hence $\eqref{obb}$ is a free $\mathrm{R}(Z_{\theta})$-module of finite type. 
\par\medskip
 
We enumerate the nilpotent orbits $\Obb_{1}$,$\Obb_{2}$,$\dots$,$\Obb_{r}$ in $\N{\check{G}}$ in such an order that $$\dim(\Obb_{1})\leq\dim(\Obb_{2})\leq\dots\leq\dim(\Obb_{r}).$$
 
If $\overline{F}^{j}=\cup_{i\leq j}\Obb_{i}$, then $\overline{F}^{j}$ is closed in $\N{\check{G}}$ and we have a filtration 
 $$\emptyset=\overline{F}^{0}\subset \overline{F}^{1}\subset\dots\subset \overline{F}^{r}=\N{\check{G}}.$$  

Let $F^{j}$ be the preimage of $\overline{F}^{j}$ in $\Ntilde{\check{G}}\times_{\N{\check{H}}}\Ntilde{\check{H}}.$ We get a $\check{G}\times\G_{\m}$-invariant filtration
$$\emptyset= F^{0}\subset F^{1}\subset \dots\subset F^{r}=\Ntilde{\check{G}}\times_{\N{\check{H}}}\Ntilde{\check{H}}.$$
We can refine the filtration $F^{i}$ in such way that the refined filtration be $\check{G}\times \G_{\m}$-stable and the corresponding strata of the stack quotient of $( \Ntilde{\check{G}}\times_{\N{\check{H}}}\Ntilde{\check{H}})/(\check{G}\times\G_{\m})$ satisfy the assumptions of Lemma $\ref{cellularfibration}.$ Then by using this Lemma, we see that for each $i$ the sequence 
$$0\longrightarrow K^{\check{G}\times\G_{\m}}(F^{i-1})\longrightarrow K^{\check{G}\times\G_{\m}}(F^{i})\longrightarrow K^{\check{G}\times\G_{\m}}(Y_{\Obb_{i}})\longrightarrow 0$$
 is exact and  $K^{\check{G}\times \G_{\m}}(F^{i}),$ $0\leq i \leq r$ define a filtration on $K(\mathcal{X}).$ Moreover,  for each $i$, $K^{\check{G}\times\G_{\m}}(F^{i})$ is a submodule over both extended affine Hecke algebras  $K^{\check{G}\times\G_{\m}}(\Ntilde{\check{G}}\times_{\N{\check{G}}}\Ntilde{\check{G}})$ and $K^{\check{H}\times\G_{\m}}(\Ntilde{\check{H}}\times_{\N{\check{H}}}\Ntilde{\check{H}}).$
 \end{proof}
 The above proof relies on the following Lemma whose proof will now be given. 
\begin{lemma}[Cellular fibration]
\label{cellularfibration}
Let us consider the following general situation: $k$ is an  algebraically  closed field  of arbitrary characteristic and  $\mathcal{X}$ is a $k$-stack of finite type equipped with a filtration
$$\emptyset= F^{0}\subset F^{1}\subset \dots\subset F^{r}=\mathcal{X}$$
by closed substacks of $\mathcal{X}.$ Assume that for   $1\leq i\leq r $ there exists an affine space $E^{i}$ and a connected linear algebraic group $P^{i}$ such that 
$$F^{i}-F^{i-1}\iso E^{i}/P^{i},$$
where $E^{i}/P^{i}$ is the stack quotient. 
Then the natural  sequence 
$$0\longrightarrow K(F^{i-1})\longrightarrow K(F^{i})\longrightarrow K(E^{i}/P^{i})\longrightarrow 0$$
is exact and $K(F^{i})$ is a free $\Z$-module. 
\end{lemma}
\begin{proof}
 Let $U^{i}$ be the unipotent radical of $P^{i}$ and $G^{i}=P^{i}/U^{i}$ be the reductive quotient. 
Choose a section of the natural projection from $G^{i}$ to  $P^{i}$, it yields a map from $E^{i}/G^{i}$ to $E^{i}/P^{i}$ inducing  an isomorphism (combined with Thom's isomorphism) 
$$K(E^{i}/P^{i})\iso K(E^{i}/G^{i})\iso K^{G^{i}}(\mathrm{Spec}(k))\iso \mathrm{R}(G^{i}),$$
where $\mathrm{R}(G^{i})$ denotes the representations ring of $G^{i}$ (which is a free $\Z$-module). One has an exact sequence 
$$K_{1}(E^{i}/P^{i})\overset{\delta}{\longrightarrow}K(F^{i-1})\longrightarrow K(E^{i}/P^{i})\longrightarrow 0.$$
Let us show that the map $\delta$ vanishes. 
By  \cite[5.2.18]{CG}, we have that 
$$K_{1}^{P^{i}}(E^{i})\iso K_{1}^{G^{i}}(E^{i})$$
and by Thom's isomorphism for higher K-theory \cite[5.4.17]{CG} we obtain that $$K_{1}^{G^{i}}(E^{i})\iso K_{1}^{G^{i}}(\mathrm{Spec}(k)).$$
Now, by \cite[Corollary 6.12]{Srinivas}, $K^{G^{i}}(\mathrm{Spec}(k))$ is isomorphic to $k^{*}\otimes_{\Z} S$, where $S$ is a free abelian  group generated by the irreducible representations of $G^{i}$.  By induction on $i$ we may assume that  $K^{i}(F^{i-1})$ is a free $\Z$-module. To finish the proof note that for any free $\Z$-module $S$, one has $\mathrm{Hom}_{\Z}(k^{*},S)=0$. 

\end{proof}
\section {Howe correspondence in terms of $K(\mathcal{X})$ for dual reductive pairs of type II}
\label{9}
Let $G=\GL_{n}$ and $H=\GL_{m}$ with $n\leq m$. We have presented some motivation for the forthcoming conjecture in the introduction.
Consider the Gronthendieck group of the geometric bimodule $DP_{I_{H}\times I_{G}}(\Pi(F))$.
The  group $K(DP_{I_{H}\times I_{G}}(\Pi(F)))$ is naturally a module over $K(DP_{I_{G}}(\flagvar{I_{G}}))$. This  
K-group $K(DP_{I_{H}}(\flagvar{I_{H}}))\otimes \qelbar$ is isomorphic to the Iwahori-Hecke algebra 
$\iwahorihecke{H}$. According to \cite{Iwahori-mat}, the Iwahori-Hecke algebra $\cH_{I_H}$ identifies with 
$\HH_H\otimes_{\ZZ[s,s^{-1}]}\Qlb$ for the map $\ZZ[s,s^{-1}]\to\Qlb$ sending $s$ to $q^{\frac{1}{2}}$. This isomorphism is naturally upgraded to the isomorphism 
$$
K(DP_{I_H}(\Fl_H))\otimes \qelbar\,
\isom\, \HH_H\otimes_{\Z[s,s^{-1}]} 
\qelbar
$$ 
such that the multiplication by $s$ in 
$\HH_H$ corresponds to the cohomological 
shift by $-1$ in $K(DP_{I_H}(\Fl_H))$. Hence 
under these isomorphisms and Kazhdan-
Lusztig-Ginzburg isomorphism, $K(\mathcal{X})$ and 
$K(DP_{I_{H}\times I_{G}}(\Pi(F)))$ are 
 bimodules over the affine extended 
Hecke algebras $\affinehecke{G}$ and 
$\affinehecke{H}$.  Let us enounce the following conjecture:

 \begin{conjecture}
\label{secondeconjecture}
 The bimodules  $K(\mathcal{X})$  and $K(DP_{I_{H}\times I_{G}}(\Pi(F)))$ are isomorphic under the action of  extended affine Hecke algebras $\affinehecke{H}$ and $\affinehecke{G}$. 
 \end{conjecture} 
 The principal result of this paper is the following theorem describing geometric Howe correspondence in terms of geometric Langlands functoriality for all dual reductive pairs $(\GL_{1},\GL_{m})$. 
\begin{theorem}
\label{theoremprincipal}
Conjecture $\ref{secondeconjecture}$ is true for $(\GL_{1},\GL_{m})$ for any $m.$
\end{theorem}
\subsection{The proof of Theorem $\ref{theoremprincipal}$}$ $

The rest of the paper is devoted to the proof of Theorem $\ref{theoremprincipal}$. 
Let $n=1$ and $m\geq 1$ and let $G=\GL_{1}$ and $H=\GL_{m},$ where we consider them as Langlands dual groups.   The map $\check{G}\times \G_{\m}\longrightarrow \check{H}$ is the composition 
$$ \check{G}\times \G_{\m}\longrightarrow \check{G}\times \mathrm{SL_{2}}\longrightarrow \check{G}\times \GL_{m-1}\longrightarrow \check{H},$$
where the latter map is the inclusion of the standard Levi  subgroup $\GL_{1}\times \GL_{m-1}$ in $\check{H}$ and $\xi:\mathrm{SL_{2}}\to \GL_{m-1}$ corresponds to the principal unipotent orbit.  In particular the inclusion $\check{G}$ in $\check{H}$ is the coweight $(1,0,\dots,0)$ of the standard maximal torus of $\check{H}.$ The restriction of $\xi$ to the maximal torus $\G_{\m}$ of $\mathrm{SL}_{2}$ is the coweight  $(0,m-2,m-4,\dots,2-m)$ of $\check{H}$. The element $x=d\xi(e)$ in $\N{\check{H}}$ is the subregular nilpotent element given by $x(u_{1})=x(u_{2})=0$ and $x(u_{i+1})=u_{i}$ for all $2\leq i< m$.  

\begin{proposition}
The bimodule $K(\mathcal{X})$ identifies with the Springer fiber $\B_{\check{H},x}$ of the Springer map $\Ntilde{H}\to \N{\check{H}}$ over the point $x$.
\end{proposition}

\begin{proof}
In this case we have  $\Ntilde{\check{G},\check{H}}= \check{H}/\check{G}$ in such way that the map $f_{1}:\Ntilde{\check{G},\check{H}}=\check{H}/\check{G}\longrightarrow \N{\check{H}}$ defined in $\eqref{mapf1}$ sends $h\check{G}$ to $hxh^{-1}$. The element $s$ in $\G_{\m}$ acts on the left hand side on $\Ntilde{\check{G},\check{H}}$ by sending  the right coset $h\check{G}$ to $h\xi(s)^{-1}\check{G}.$  The variety $\Ntilde{\check{G},\check{H}}\times_{\N{\check{H}}}\Ntilde{\check{H}}$ identifies with the variety of pairs $(h\check{G}, \borel)$ such that $\borel$ is a Borel subalgebra in $\check{H}$ and $hxh^{-1}$ lies in $\borel.$ Any element $(h_{1},s)$ in $\check{H}\times \G_{\m}$ acts on $\Ntilde{\check{G},\check{H}}\times_{\N{\check{H}}}\Ntilde{\check{H}}$ by the formula:
$$(h_{1},s).(h\check{G},\borel)=(h_{1}h\xi(s)^{-1}\check{G}, h_{1}\borel h_{1}^{-1}).$$
Denote by $\B_{\check{H},x}$ the fiber of  the Springer map $\Ntilde{\check{H}}\to \N{\check{H}}$ over $x$. The map 
$$\overline{\sigma}:\check{G}\times\G_{\m}\longrightarrow \check{H}\times\G_{\m}$$
sending $(g,s)$ to $(g\xi(s),s)$ identifies $\check{G}\times \G_{\m}$ with the stabilizer in $\check{H}\times \G_{\m}$ of  the right coset of the neutral element in  $\check{H}/\check{G}.$ Any element $(g,s)$ of $\check{G}\times\G_{\m}$ acts on the Springer fiber $\B_{\check{H},x}$ by 
$$(g,s).\borel^{'}=(g\xi(s)\borel^{'}\xi(s)^{-1}g^{-1}).$$ 
This yields an isomorphism
$$K(\mathcal{X})\iso K^{\check{H}\times \G_{\m}}(\Ntilde{\check{H},\check{G}}\times_{\N{\check{H}}}\Ntilde{\check{H}})\iso K^{\check{G}\times \G_{\m}}(\B_{\check{H},x}).$$
\end{proof}
To compute $K(\mathcal{X})$, we provide an explicit description of the Springer fiber $\B_{\check{H},x}$.
\begin{lemma}
\label{springerfibre}
The Springer fiber $\B_{\check{H},x}$ is a configuration of projective lines $(V_{i})_{1\leq i\leq m-1}$. For $1\leq i< j \leq m-1$ the intersection $V_{j}\cap V_{i}$ is empty unless $j=i+1$.  
The fixed locus in $\B_{\check{H},x}$  under the action of $\check{G}\times \G_{\m}$ consists of $m$ points $p_{1}, p_{2},\dots, p_{m-1}, p_{m}$, where $p_{1}$ and $p_{m}$ are distinguished points on $V_{1}$ and $V_{m}$ and for $2\leq i\leq m-1,$ the point $p_{i}$ is the intersection of $V_{i}$ with $V_{i+1}$.  
\end{lemma}

\begin{proof}
Denote by 
$$F_{1}\subset F_{2}\subset \dots \subset F_{m}=U_{0}$$
a complete flag on the standard representation $U_{0}$ of 
$\check{H}$ preserved by $x.$ The vector space $F_{1}$ is a 
subspace of the vector space $\mathrm{Ker}(x)=\mathrm{Vect}(u_{1},u_{2}).$ We 
have $\mathrm{Vect}(u_{2})=\mathrm{Ker}(x)\cap\mathrm{Im}(x).$ 
If $F_{1}\neq \mathrm{Vect}(u_{2})$ then $F_{2}=x^{-1}
(F_{1})=\mathrm{Vect}(u_{1},u_{2}),$ $F_{3}=x^{-1}
(F_{2})=\mathrm{Vect}(u_{1},u_{2},u_{3}),\dots,$ and finally the space $F_m$ is equal to 
$x^{-1}(F_{m-1})=\mathrm{Vect}(u_{1},u_{2},\dots,u_{m})=U_{0}.$
So we may identify $V_{1}$ with the projective space of lines in $\mathrm{Vect}(u_{1},u_{2}).$ The point $p_{2}$ is $F_{1}=\mathrm{Vect}(u_{2}).$ 
If $F_{1}=\mathrm{Vect}(u_{2})\subset \mathrm{Im}(x)$ then $x^{-1}(F_{1})=\mathrm{Vect}(u_{1},u_{2},u_{3})$ and $V_{2}
$  can be identified with the space of lines in $x^{-1}(F_{1})/F_{1}.$ Inside $\mathrm{Vect}(u_{1},u_{2},u_{3})$ one has  a distinguished subspace $\mathrm{Vect}(u_{1},u_{2},u_{3})\cap \mathrm{Im}(x)=\mathrm{Vect}(u_{2},u_{3}).$ If $F_{2}$ is different from this subspace then the whole flag $F_{i}$ is uniquely defined. So the point $p_{3}$ of $V_{2}$ corresponds to $F_{2}=\mathrm{Vect}(u_{2},u_{3}).$ 
If now $F_{1}=\mathrm{Vect}(u_{2})$ and $F_{2}=\mathrm{Vect}(u_{2},u_{3})$ then $x^{-1}(F_{2})=\mathrm{Vect}(u_{1},u_{2},u_{3},u_{4})$ and $D_{3}$ is the space of lines in $x^{-1}(F_{2})/F_{2}.$ The point $p_{4}$ of $V_{3}$ corresponds  to $F_{3}=\mathrm{Vect}(u_{2},u_{3},u_{4}),$ and one can continue the construction till $F_m.$
The points $p_{1}$ is the standard complete flag on $U_{0}$ and $p_{m}$ is the flag $\Vect(u_{2})\subset\Vect(u_{2},u_{3})\subset\dots\subset \Vect(u_{2},\dots,u_{m})\subset \Vect(u_{1},\dots,u_{m}).$
\end{proof}
This result combined with the Cellular fibration Lemma in \cite[\S\ 5.5]{CG} implies the following: 
\begin{proposition}
\label{modulerangm}
The K-group $K^{\check{G}\times \G_{\m}}(\B_{\check{H},x})$ is a free  $\mathrm{R}(\check{G}\times\G_{\m})$-module of rank $m$.  Moreover, the $\mathrm{R}(\check{H})$-module structure on $ K^{\check{G}\times \G_{\m}}(\B_{\check{H},x})$ comes from $\mathrm{Res}^{\sigma}:\mathrm{R}(\check{H})\longrightarrow \mathrm{R}(\check{G}\times\G_{\m}).$ 
\end{proposition}

\par\medskip
According to \cite[Lemma 7.6.2]{CG} the assignment sending $T_{w}$ to $s^{\ell(w)}$ for $w$ in $W_{H}$ extends  by linearity  to an algebra homomorphism 
$$
\epsilon:\affinehecke{W_{H}}\longrightarrow \Z[s,s^{-1}]
$$
and it is known that the induced $\HH_H$-module
$\mathrm{Ind}_{\affinehecke{W_{H}}}^{\affinehecke{H}}\epsilon=\affinehecke{H}\otimes_{\affinehecke{W_{H}}}\epsilon $ 
is isomorphic to the polynomial representation \cite[7.6.8]{CG}.
We have the following crucial chain of isomorphisms of $\Z[s,s^{-1}]$-modules \cite[Formula (7.6.5)]{CG}: 
$$K^{\Hche\times \G_{\m}}(T^{*}\B_{\Hche})\overset{\mathrm{Thom} }{\longrightarrow}K^{\check{H}\times \G_{\m}}(\B_{\Hche})\overset{\alpha}{\longrightarrow}\mathrm{R}(\check{T}_{H})[s,s^{-1}]\overset{\beta}\longrightarrow {\mathrm{Ind}_{\affinehecke{W_{H}}}^{\affinehecke{H}}\epsilon},$$
where the first arrow is the Thom isomorphism \cite[Theorem 5.4.16]{CG}, the map $\alpha$  is the canonical isomorphism 
\begin{align}
K^{\check{H}\times \G_{\m}}(\B_{\Hche})\iso & K^{\check{H}\times \G_{\m}}(\check{H}/B_{\check{H}})\iso  K^{B_{\check{H}}\times \G_{\m}}(\mathrm{pt})\\
\iso & \mathrm{R}(\check{T}_{H}\times \G_{\m})\iso \mathrm{R}(\check{T}_{H})[s,s^{-1}],\nonumber
\end{align}
and the map $\beta$ is given for any $\lambda$  by  $\beta(e^{\lambda})=e^{-\lambda}$. 
\par\medskip
There is a natural  action of $\affinehecke{H}$ on $K^{\check{G}\times \G_{\m}}(\B_{\check{H},x})$  defined uniquely by the property that the inclusion of $\B_{\check{H},x}$ in $\B_{\check{H}}$ yields a $\mathrm{R}(\check{G}\times \G_{\m})\otimes_{\mathrm{R}(\check{H}\times\G_{\m})}\affinehecke{H}$-equivariant surjection 
$$\mathrm{R}(\check{G}\times \G_{\m})\otimes_{\mathrm{R}(\check{H}\times\G_{\m})} K^{\check{H}\times\G_{\m}}(\B_{\Hche})\iso K^{\check{G}\times \G_{\m}}(\B_{\Hche})\longrightarrow K^{\check{G}\times \G_{\m}}(\B_{\Hche,x}).$$
Consider the diagram
 \begin{equation}
 \label{digamramcommut}
 \xymatrix @C=0.5cm{
     \HH_H \ar[r]^-{\gamma_{1}} \ar[d]^{\gamma_{2}} &  K(DP_{I_{H}\times I_{G}}(\Pi(F)))  \\
      K^{\check{G}\times\G_{\m}}(\cB_{\check{H},x} \ar@{.>}[ru]), 
    }
\end{equation}
where $\gamma_1$ sends $\cT$ to $\heckefunc{H}(\cT, I_0)$, and $\gamma_2$ sends $\cT$ to the action of $\cT$ on the structure sheaf $\cO$ of $\cB_{\check{H,}x}$. Note that $\gamma_1$ and $\gamma_{2}$ are surjective.  We are now going to construct a morphism 
$$
\gJ: K^{\check{G}\times\Gm}(\cB_{\check{H},x})
\to K(DP_{I_{H}\times I_{G}}(\Pi(F))).
$$
which will be induced by $\gamma_{1}$.
One sees that $\gamma_1$ factors through the surjective morphism $\bar\gamma_1: \affinehecke{H}\otimes_{\affinehecke{W_{H}}}\epsilon\to K(DP_{I_{H}\times I_{G}}(\Pi(F)))$ of $\HH_H$-modules.  For proving Theorem $\ref{theoremprincipal}$ we are reduced to prove the following: 
\begin{proposition}
\label{theodelavie}
There is a unique isomorphism of $\affinehecke{H}$-modules $\mathfrak{J}$ making diagram $\eqref{digamramcommut}$ commutative. The map $\mathfrak{J}$ commutes with the $\affinehecke{G}$-actions. 
\end{proposition}
 
Note that if $n=m=1$  then one has $I_{H}=H(\locring)$ and this proposition can be deduced from \cite[Proposition 4]{Lysenko1}. If $m=2$, we can also provide a quick proof of the proposition; in this case both $K(DP_{I_{H}\times I_{G}}(\Pi(F)))$ and $\affinehecke{H}\otimes_{\affinehecke{W_{H}}}\epsilon$ are free $R(\check{G}\times \G_{\m})$-modules of rank 2, and $\bar\gamma_1$ is an isomorphism.

 We have seen in section $\ref{7}$ that the module $K(DP_{I_{G}\times I_{H}}(\Pi(F)))$ is free of rank $m$ over $\mathrm{R}(\check{G}\times \G_{\m})$. In the notation of this section, a basis of the group $K(DP_{I_{G}\times I_{H}}(\Pi(F)))$ is given by the elements $\IC^{k}$ for $0\leq k\leq m-1,$ and the action of $\mathrm{R}(\check{G}\times \G_{\m})$ is given on this basis in $\eqref{coeurcoeur}.$ Besides, according to Theorem $\ref{centre},$ $\mathrm{R}(\check{H})$ acts  via 
$\mathrm{Res}^{\sigma}$.
A part of these properties has been already proved for $K^{\check{G}\times \G_{\m}}(\B_{\check{H},x})$ in Proposition $\ref{modulerangm}$. In the sequel we will construct a basis of $K^{\check{G}\times \G_{m}}(\B_{\check{H},x})$  and we will identify the action of $\affinehecke{H}$ on this basis and the basis $\IC^{k}$.  The morphism sending one basis to another will be induced by $\gamma_{1}.$ Surprisingly the basis we will construct  is not the canonical basis of  Lusztig constructed in \cite{Lu6}.
\par\bigskip
We will use the polynomial representation of the affine extended Hecke algebra $\affinehecke{H}$  to describe the action of $\affinehecke{H}$ on this new basis that we will construct. So let us first describe the representation of $\affinehecke{H}$ in $\mathrm{R}(\check{T}_{H})[s,s^{-1}]$. Consider the polynomial representation of the  extended affine Hecke algebra $\affinehecke{H}$ of $H$ in $\mathrm{R}(\check{T}_{H})[s,s^{-1}].$ For $v$ in $\affinehecke{H}$ and $z$ in $\mathrm{R}(\check{T}_{H})[s,s^{-1}]$ write $v\ast
z$ for the action of $v$ on $z$. 
The element $e^{\lambda}$  denotes the element in $\mathrm{R(\check{T}_{H})}[s,s^{-1}]$ corresponding to  $\lambda$;  according to \cite[Formula (7.6.1)]{CG}, $e^{\lambda}$ as an element of $\affinehecke{H}$ acts on any element  $u$ of $\mathrm{R}(\check{T}_{H})[s,s^{-1}]$ by 
\begin{equation}
\label{poly}
e^{\lambda}\ast u =e^{-\lambda}u,
\end{equation}
 and for any simple root $\alpha,$  the action of $T_{s_{\alpha}}$ on $e^{\lambda}$ is given by the formula \cite[Theorem 7.2.16]{CG}:
 \begin{equation}
 \label{Ts}
 T_{s_{\alpha}}\ast e^{\lambda}=\frac{e^{\lambda}-e^{s_{\alpha}(\lambda)}}{e^{\alpha}-1}-s^{2}\frac{e^{\lambda}-e^{s_{\alpha}(\lambda)+\alpha}}{e^{\alpha}-1}.
\end{equation}
This formula was discovered by Lusztig  and was the starting point of the K-theoretic approach to Hecke algebras. 
The formulas $\eqref{poly}$ and $\eqref{Ts}$ together completely determine the polynomial representation of $\affinehecke{H}.$
 For $\lambda$ dominant, the element $e^{\lambda}$  corresponds in the Iwahori-Hecke algebra to the function $s^{-\ell(\lambda)}T_{t^{\lambda}}$, where $\ell(\lambda)=\langle \lambda, 2\check{\rho}_{H} \rangle$ and $T_{t^{\lambda}}$ is the characteristic function of the double coset $I_{G}t^{\lambda}I_{G}$. 
Denote by $\omega_{i}$ the coweight $(1,\dots,1,0,\dots,0),$  where $1$ appears $i$ times. For $1 \leq i< m$ , denote by $w_{i}=t^{\omega_{i}}\sigma_{i}$ the element of length zero. 
The element $w_{1}$ is the generator  of the group  $\Omega_{H}$ of length zero elements in $\weyl{H}$;  for any $i$ in $\Z,$  $w_{i}=w_{1}^{i}$. 
In the  extended affine Hecke algebra $\affinehecke{H}$ we have 
$$T_{t^{\omega_{i}}}T_{w_{i}}=T_{\sigma_{i}}.$$
Further we have $\ell(t^{\omega_{i}})=\ell(\sigma_{i})=\langle \omega_{i},2\check{\rho}_{\Hche} \rangle=i(m-i)$
and this gives  
\begin{equation}
e^{\omega_{i}}=s^{i(i-m)}T_{t^{\omega_{i}}}.
\end{equation}
In $\mathrm{R}(\check{T}_{H})[s,s^{-1}],$  $T_{\sigma_{i}}\ast 1=s^{2i(m-i)}$ and this yields 
$$(s^{i(m-i)}e^{\omega_{i}}T_{w_{i}})\ast 1=s^{2i(m-i)},$$
and 
\begin{equation}
\label{l0}
T_{w_{i}}\ast 1=s^{i(m-i)}e^{\omega_{i}}.
\end{equation}
Till now we have described the action of the Wakimoto objects and the elements of length zero. We are going to compute the action of the simple reflections $s_{i}=(i,i+1)$ and the affine simple reflection $s_{m}=t^{\lambda}w_{0}$, where $\lambda=(-1,0\dots,0,1)$ and $w_{0}=(1,m)$ is the longest element of the finite Weyl group of $H$. 
For $1\leq i\leq m$ we have $T_{w_{1}}T_{s_{i}}T_{w}^{-1}=T_{s_{i+1}}$ and $T_{w_{1}}T_{s_{m}}T_{w_{1}}^{-1}=T_{s_{1}}.$
For any integer $j$ in $\Z$ set $s_{j}=s_{j+m}$ and rewrite the above formulas all together as 
$$T_{w_{1}}T_{s_{i}}T_{w_{1}}^{-1}=T_{s_{i+1}}.$$
Thus, for all $i$ and $j$ in $\Z,$ 
$$T_{w_{j}}T_{s_{i}}T_{w_{j}}^{-1}=T_{s_{i+j}}.$$
For any cocharacter $\mu$ we have $w_{i}t^{\mu}w_{i}^{-1}=t^{\sigma_{i}(\mu)}$ and we get
$$T_{w_{i}}T_{t^{\mu}}T_{w_{i}}^{-1}=T_{t^{\sigma_{i}(\mu)}}.$$
\begin{proposition}
\label{Tsm}
In the polynomial representation the element $T_{s_{m}}$ acts on $1$ by $(s^{2}-1)+s^{2(m-1)}e^{\xi+\omega_{1}},$
where $\xi=(0,0,\dots,0,-1).$
\end{proposition}
\begin{proof}
Since  $T_{s_{m}}=T_{w_{1}}^{-1}T_{s_{1}}T_{w_{1}},$ we get using $\eqref{l0}:$
$$T_{s_{m}}\ast 1=(T_{w_{1}}^{-1}T_{s_{1}})\ast s^{m-1}e^{\omega_{1}}.$$ 
Let $\alpha_{i}=(0,\dots,0,1,-1,0\dots,0)$ and  $\mu_{i}=(0,\dots,0,1,0\dots,0),$  where $1$ appears on the $i^\mathrm{th}$ place. Then 
$$T_{s_{1}}\ast e^{\omega_{1}}=e^{\omega_{1}-\alpha_{1}}.$$
 Thus,
\begin{equation}
T_{s_{m}}\ast 1=T_{w_{1}}^{-1}\ast s^{m-1}e^{\omega_{1}-\alpha_{1}}.
\end{equation}
If $\xi=-\sigma_{1}^{-1}\omega_{1}=(0,\dots,0,-1),$ then $\xi$ is a dominant character, and we have 
$T_{t^{\xi}}T_{w_{1}}^{-1}=T_{\sigma_{1}}^{-1}$. Thus 
$$T_{w_{1}}^{-1}=s^{1-m}e^{-\xi}T_{\sigma_{1}^{-1}}.$$
Finally we have to compute 
$$T_{s_{m}}\ast 1=s^{1-m}e^{-\xi}T_{\sigma_{1}^{-1}}\ast s^{m-1}e^{\omega_{1}-\alpha_{1}}.$$
On one hand the reduced decomposition of $\sigma_{1}^{-1}$ is $s_{m-1}\dots s_{2}s_{1}$ and it follows that $T_{\sigma_{1}^{-1}}=T_{s_{m-1}}\dots T_{s_{2}}T_{s_{1}}.$
From $\eqref{Ts}$ we get that $T_{s_{1}}\ast e^{\mu_{2}}=(s^{2}-1)e^{\mu_{2}}+s^{2}e^{\omega_{1}}$. For $2\leq i\leq m-1$ we have $T_{s_{i}}\ast e^{\omega_{1}}=s^{2}e^{\omega_{1}}$. We also have $T_{s_{2}}\ast e^{\mu_{2}}=e^{\mu_{2}-\alpha_{2}}=e^{\mu_{3}}$ and  more generally,   for  $1\leq i<m,$ $T_{s_{i}}\ast e^{\mu_{i}}=e^{\mu_{i+1}}$. By induction we get
$$T_{\sigma_{1}^{-1}}\ast e^{\mu_{2}}=(s^{2}-1)e^{\mu_{m}}+s^{2(m-1)}e^{\omega_{1}}.$$
This implies that 
\begin{equation}
T_{s_{m}}\ast 1=(s^{2}-1)+s^{2(m-1)}e^{\xi+\omega_{1}}.
\end{equation}
\end{proof}

In order to prove Proposition $\ref{theodelavie}$ we have to study the $\affinehecke{H}$-module structure of $K^{\Gche\times \G_{\m}}(\B_{\Hche,x})$  and compare this action with the results obtained $\eqref{coeurcoeur}$.  Now let us construct the desired basis of $K^{\check{G}\times \G_{\m}}(\B_{\check{H},x})$.
Denote by $L_{\lambda}$ the line bundle on $\B_{\Hche}$ corresponding to  coweight $\lambda$ of $H$ as in \cite[\S 6.1.11]{CG}. The $\Hche$-module $\mathrm{H}^{0}(\B_{\Hche},L_{\lambda})$ vanishes  unless $(a_{1}\leq \dots\leq a_{m})$.  Recall that the nilpotent subregular element $x$ in $\mathrm{End}(U_{0})$ is such that $x(u_{1})=x(u_{2})=0$ and $x(u_{i})=u_{i-1}$ for all $3\leq i\leq m.$ 
The natural morphism from $\mathrm{R}(\check{T}_{H})[s,s^{-1}]$ to $K^{\check{G}\times \G_{\m}}(\B_{\check{H},x})$ sends an element $e^{\lambda}$  to $L_{-\lambda}.$ Besides, any element $\mathcal{L}$ in $K^{\Hche\times \G_{\m}}(\Ntilde{\Hche})$ acts on $K^{\Gche\times \G_{\m}}(\B_{\Hche,x})$ as the tensor product by $\mathcal{L}_{\vert_{\B_{\Hche,x}}}.$
\par\medskip
Let $\{u_{1},\dots u_{m}\}$ be the canonical basis of $U_{0}$ and $\{u_{1}^{*},\dots ,u_{m}^{*}\}$ the corresponding dual basis.    For $1\leq i\leq m$ set 
$$U_{i}=\mathrm{Vect}(u_{1},\dots,u_{i})$$
and for $1\leq i\leq m-1$ set
 $$U_{i}^{'}=\mathrm{Vect}(u_{2},\dots,u_{i+1}),$$
with $U_{0}^{'}$ being equal to $\{0\}.$ Note that  for $0\leq i\leq m-2$ the element $x$ acts on $U_{i+2}/U^{'}_{i}$ by zero. For $1\leq i <m$ let $V_{i}$ be the projective line classifying flags 
$$U_{1}^{'}\subset \dots U_{i-1}^{'}\subset W_{i}\subset U_{i+1}\subset \dots \subset U_{m},$$
where $W_{i}$ is $i$-dimensional. The line $V_{i}$ is isomorphic to $\mathbb{P}(\Vect(u_{1},u_{i+1}))$ via the map sending a line $l$ to the flag given by 
$$U_{1}^{'}\subset \dots U_{i-1}^{'}\subset l\oplus U_{i-1}^{'}\subset U_{i+1}\subset \dots \subset U_{m}.$$
Then we have $\B_{\Hche,x}=\cup_{i} V_{i},$ (see Lemma $\ref{springerfibre}$). Recall that there are $m$ fixed points  on $\B_{\Hche,x}$  under the action of $\check{G}\times \G_{m}$ corresponding to the following flags: 
\begin{enumerate}
\item $p_{1}=U_{1}\subset U_{2}\subset\dots\subset U_{m}.$
\item For $2\leq k\leq m-1,$  
$$p_{k}=U_{1}^{'}\subset \dots U^{'}_{k-1}\subset U_{k}\subset \dots \subset U_{m}.$$
\item $p_{m}=U^{'}_{1}\subset U^{'}_{2}\subset\dots\subset U_{m-1}^{'}\subset U_{m}.$
\end{enumerate}
Note that for $2\leq k\leq m-1,$ the point $p_{k}$ equals $V_{k-1}\cap V_{k}.$
\par\medskip
Each line $V_{i}$ is endowed with a tautological equivariant line bundle $\locring_{V_{i}}(-1)$ which is an equivariant subbundle of $g\locring_{V_{i}}\oplus s^{m-2i}\locring_{V_{i}}.$ Note that: for $1\leq i\leq m-1,$
$$\locring_{V_{i}}(-p_{i})=s^{2i-m}\locring_{V_{i}}(-1)\quad\mathrm{and}\quad\locring_{V_{i}}(-p_{i+1})=g^{-1}\locring_{V_{i}}(-1).$$
Thanks to Lusztig  \cite[\S 4.7]{Lu6} the elements $\locring_{p_{1^{\vphantom{COEUR}}}}$, $\locring_{V_{1}}(-1),\dots,\locring_{V_{m-1}}(-1)$ define a basis of $K^{\check{G}\times \G_{\m}}(\B_{\check{H},x})$ over $\mathrm{R}(\check{G})[s,s^{-1}].$

\par\medskip
For $1\leq i < m,$ consider the line bundle $L_{\omega_{i}}$ on $\B_{\Hche,x}$ whose fiber  at a point $F_{1}\subset\dots\subset F_{m}$ is $\mathrm{det}(F_{i})$. Recall that $\mathrm{det}(U^{'}_{i})\iso s^{i(m-i-1)}$ as a $\Gche\times\G_{\m}$-representation. We also have $L_{\omega_{m}}=g\locring$ in $K^{\Gche\times \G_{\m}}(\B_{\Hche,x}).$

\begin{proposition}
\label{basis}
 The set of line bundles $\{\locring,L_{-\omega_{1}},\dots L_{-\omega_{m-1}}\}$ forms a basis of K-group  $K^{\Gche\times \G_{\m}}(\B_{\Hche,x})$ after specialization.  
\end{proposition}
\begin{proof}
 For $1\leq k\leq m-1$ and for any $\Gche\times \G_{\m}$-equivariant line bundle $L$ on $\B_{\Hche,x},$ we have the following equality in $K^{\check{G}\times \G_{\m}}(\B_{\check{H},x})$:
$$L=\sum_{j=1}^{k-1}L_{\vert_{V_{j}}}(-p_{j+1})+L_{\vert_{V_{k}}}+\sum_{j=k+1}^{m-1}L_{\vert_{V_{j}}}(-p_{j})$$  
We apply this formula to $L_{\omega_{k}}$. Note that :
\begin{enumerate}
\item[-] If $j<k$, ${L_{\omega_{k}}}_{\vert_{V_{j}}}=gs^{(k-1)(m-k)}\locring_{V_{j}}.$
\item[-] If $j=k$,   ${L_{\omega_{k}}}_{\vert_{V_{j}}}=\locring_{V_{j}}(-1).$
\item[-] If $j>k$, ${L_{\omega_{k}}}_{\vert_{V_{j}}}=s^{k(m-k-1)}\locring_{V_{j}}.$
\end{enumerate}
Hence we get:
$$L_{\omega_{k}}=\locring_{V_{k}}(-1)+s^{(k-1)(m-k)}\Bigl[   \sum_{j=1}^{k-1}\locring_{V_{j}}(-1)+\sum_{j=k+1}^{m-1}s^{2(j-k)}\locring_{V_{j}}(-1) \Bigr].$$
Lastly 
$$\locring={\locring}_{p_{1^{\vphantom{COEUR}}}}+\sum_{j=1}^{m-1}s^{2j-m}\locring_{V_{j}}(-1).$$
Since  ${\locring}_{p_{1^{\vphantom{COEUR}}}}$, $\locring_{V_{1}}(-1),\dots,\locring_{V_{m-1}}(-1)$ is a basis of $K^{\check{G}\times \G_{\m}}(\B_{\check{H},x}),$ the previous formulas imply that $\mathcal{O},L_{\omega_{1}}\dots,L_{\omega_{m-1}}$ is a free family  which becomes a basis after specializing $s$ to $q^{1/2}.$ If we apply the duality functor, we get the same result for the family $\mathcal{O},L_{-\omega_{1}}\dots,L_{-\omega_{m-1}}.$

 \end{proof}
Consider the family $\{\mathcal{O},s^{m-1}L_{-\omega_{1}},s^{2(m-2)}L_{-\omega_{2}},\dots,s^{m-1}L_{-\omega_{m-1}}\}.$
Thanks to Proposition $\ref{basis}$ this family is also a basis of $K^{\check{G}\times \G_{m}}(\B_{\check{H},x})$ after specialization. 
The map $\gamma_{1}$ factors through morphism $\gJ$ sending this basis to $\{\IC^{0},\dots,\IC^{m-1}\}.$
\par\medskip
According to $\eqref{l0}$, we have 
$$\gamma_{2}(T_{w_{i}})=T_{w_{i}}(\mathcal{O})=s^{i(m-i)}e^{\omega_{i}}=s^{i(m-i)} L_{-\omega_{i}}.$$ Hence the action of length zero elements on the basis is compatible with their action on $\{\IC^{0},\dots,\IC^{m}\}$ in \S $\ref{7},\eqref{coeurcoeur}$.   
\par\bigskip
Now we will compute  the action of the affine simple reflection $s_{m}$. Let $\lambda$ be the cocharacter $(-1,0\dots,0,1),$ and consider the associated line bundle $L_{\lambda}$ (resp. $E$) on $\B_{\Hche,x}$ whose fiber over a  flag $F_{1}\subset \dots \subset F_{m}=U_{m}$ is $F_{1}^{*}\otimes F_{m}/F_{m-1}$(resp. $F_{m}/F_{m-1}$).
The section $u_{m}$ of the line bundle $E$ yields an exact sequence 
$$0 \longrightarrow s^{2-m}\locring \longrightarrow E\longrightarrow (L_{m-1,m})_{p_{m}}\longrightarrow 0.$$
Note that  $(E)_{p_{m}}=g\locring_{p_{m^{\vphantom{COEUR}}}},$  and 
$(L_{-\omega_{1}})_{p_{m}}=s^{2-m}\locring_{p_{m^{\vphantom{COEUR}}}}.$ 
 Tensoring by $L_{-\omega_{1}},$ we get the exact sequence on $\B_{\Hche,x}$
$$0 \longrightarrow s^{2-m}L_{-\omega_{1}}\longrightarrow L_{\lambda}\longrightarrow gs^{2-m}\locring_{p_{m^{\vphantom{COEUR}}}}\longrightarrow 0.$$
Consider $u_{1}^{*}\wedge\dots\wedge u_{m-1}^{*}$ as global section of $L_{-\omega_{m-1}}$ over $\B_{\Hche,x}.$ It vanishes only at $p_{m}$ and gives an exact sequence 
$$0 \longrightarrow g^{-1}s^{2-m}\locring \longrightarrow L_{-\omega_{m-1}}\longrightarrow \locring_{p_{m^{\vphantom{COEUR}}}}\longrightarrow 0.$$
Finally we conclude that in $K^{\Gche\times \G_{\m}}(\B_{\Hche,x})$
$$L_{\lambda}=s^{2-m}L_{-\omega_{1}}+gs^{2-m}\locring_{p_{m^{\vphantom{COEUR}}}},\quad \mathrm{and}\quad gs^{2-m}\locring_{p_{m^{\vphantom{COEUR}}}}=gs^{2-m}L_{-\omega_{m-1}}-s^{4-2m}\locring.$$ 
Thus 
$$L_{\lambda}=s^{2-m}L_{-\omega_{1}}+gs^{2-m}L_{-\omega_{m-1}}-s^{4-2m}\locring.$$
From Proposition $\ref{Tsm}$ we obtain that 
\begin{equation}
\label{formula_action_T_sm}
\gamma_{2}(T_{s_{m}})=T_{s_{m}}(\locring)=(s^{2}-1)\locring +s^{2m-2}L_{\lambda}=-\locring+s^{m}L_{-\omega_{1}}+gs^{m}L_{-\omega_{m-1}}.
\end{equation}
Finally, $s^{-1}T_{s_m}(\cO)+s^{-1}\cO$ corresponds to $\heckefunc{H}(L_{s_m}, I_0)$, and the formula $\eqref{formula_action_T_sm}$ is compatible with $\eqref{coeurcoeur}$ by using the fact that  $L_{s_m}$ is isomorphic to $\Qlb[1](\frac{1}{2})$ over $\overline{\Fl}_H^{s_m}$. 
Moreover, for $1\le i<m$ one has $T_{s_i}\ast 1=v$ in the polynomial representation, hence $T_{s_i}(\cO)=v\cO$ in $K^{\check{G}\times \G_{\m}}(\cB_{\check{H},x})$. The other relations are readily obtained by symmetry (the action of elements of length zero). This finishes the proof of Proposition $\ref{theodelavie}$ and so Conjecture $\ref{secondeconjecture}$. 
\newpage
\section{Appendix}
\label{Appendix1}
Let $\K$ be an  algebraically closed field of characteristic zero. Let $G$ be a linear algebraic group over $\qelbar$. Denote by $\mathrm{R}(G)$ the representation ring of $G$ over $\qelbar$. By equivariant K-theory on a scheme or a stack we always mean K-theory of $G$-equivariant coherent sheaves. For  more details we refer the reader  to \cite[Chapter 5]{CG}. 
\subsection{Generalities on convolution product in K-theory}
\label{A}
\subsubsection{}
\label{a}
 Let $Y$ be a smooth $G$-variety and $\pi: Y\to X$ be a proper $G$-equivariant map. According to \cite[5.2.20]{CG} $K^G(Y\times_X Y)$ is an associative $\mathrm{R}(G)$-algebra.
Moreover, $K^G(Y)$ is naturally a left module over $K^G(Y\times_X Y)$. Namely, for any  $L$ in $K^G(Y\times_X Y)$ and  any $F$ in $K^G(Y)$, consider the restriction with supports (see \cite[\S 5.2.5 (iii)]{CG}) of  an element $L\boxtimes F$ of $K^G((Y\times_X Y)\times Y)$ with respect to the smooth closed embedding
$$
\begin{array}{ccc}
Y\times Y &\toup{\id\times\diag}& Y\times Y\times Y\\
\cup && \cup\\
Y\times_X Y & \to & (Y\times_X Y)\times Y
\end{array}
$$
and denote the result by $L\otimes p_2^*F\in K^G(Y\times_X Y)$. Then we have $L\ast F=(p_1)_*(L\otimes p_2^*F)\in K^G(Y)$.

\subsubsection{}
\label{B}
Let $Z$ be a smooth variety. Consider a $G$-equivariant morphism from $Z$ to $X$. Then $K^G(Y\times_X Y)$ acts on $K^G(Z\times_X Y)$ by convolution on the right. Additionally, this action is $\mathrm{R}(G)$-linear. Namely, for any $F$  in $K^G(Z\times_X Y)$ and  any $L$  in $K^G(Y\times_X Y)$, consider the element 
$p_{12}^*F\boxtimes p_{34}^*L$ in $K^G((Z\times_X Y)\times(Y\times_X Y))$.
Let us apply the restriction with supports functor with respect to the smooth closed embedding $\id\times \diag\times\id$ in the following diagram to $p_{12}^*F\boxtimes p_{34}^*L$
$$\begin{array}{ccc}
Z\times Y\times Y & \toup{\id\times \diag\times\id} & Z\times Y\times Y\times Y\\
\cup && \cup\\
Z\times_X Y\times_X Y & \to &(Z\times_X Y)\times(Y\times_X Y)
\end{array},$$
and denote the result by $p_{12}^*F\otimes p_{23}^*L$ in $K^G(Z\times_X Y\times_X Y)$. The projection $p_{13}: Z\times_X Y\times_X Y\to Z\times_X Y$ is proper, and we obtain the convolution product of $F$ and $L$ denoted by 
$$F\ast L=(p_{13})_*(p_{12}^*F\otimes p_{23}^*L)\in K^G(Z\times_X Y).$$

\subsubsection{}
\label{D}
Let $Y$ be a smooth $G$-variety and $\pi: Y\to X$ a proper $G$-equivariant morphism. Let $X\to \bar X$ and $Z\to \bar X$ be $G$-equivariant morphisms of varieties. Assume $Z$ to be smooth. Then $K^G(Y\times_X Y)$ acts on the left by convolution on $K^G(Y\times_{\bar X} Z)$. Indeed, for any $F$ in $K^G(Y\times_{\bar X} Z)$ and $L$ in  $K^G(Y\times_X Y)$, consider $p_{12}^*L\boxtimes p_{34}^*F$ in $K^G((Y\times_X Y)\times (Y\times_{\bar X} Z)).$
Apply the restriction with supports with respect to the smooth closed embedding $\id\times\diag\times\id$ in the following diagram to $p_{12}^*L\boxtimes p_{34}^*F$
$$
\begin{array}{ccc}
Y\times Y\times Z & \toup{\id\times\diag\times\id} &
Y\times Y\times Y\times Z\\
\cup && \cup\\
Y\times_X Y\times_{\bar X} Z & \to &  
(Y\times_X Y)\times (Y\times_{\bar X} Z).
\end{array}
$$
and denote the result by $p_{12}^*L\otimes p_{34}^*F$ in $K^G(Y\times_X Y\times_{\bar X} Z)$. The projection $p_{13}: Y\times_X Y\times_{\bar X} Z\to Y\times_{\bar X} Z$ is proper, and we obtain the convolution product of $L$ and $F$ denoted by 
$$ L\ast F=(p_{13})_*(p_{12}^*L\otimes p_{34}^*F)\in K^G(Y\times_{\bar X} Z).$$

Note actually that the essential thing we need is the fact that the structure sheaf $\cO_Y$ of the diagonal $Y\subset Y\times Y$ admits a finite $G$-equivariant resolution by locally free $\cO_{Y\times Y}$-modules of finite rank. Then restrict this resolution with respect to the flat projection $p_{23}: Y\times Y\times Y\times Z\to Y\times Y$. Assume $Z\to \bar X$ to be  proper, then $K^G(Z\times_{\bar X} Z)$ acts on $K^G(Y\times_{\bar X} Z)$ by convolutions on the right, and the actions of $K^G(Z\times_{\bar X} Z)$ and of  $K^G(Y\times_X Y)$ commute. 
\par\medskip
Let $\overline{x}$ be a $G$-fixed point  in $\overline{X}$. Assume that the morphism $X\to \overline{X}$ factors  through $X\to \overline{x}\to \overline{X}$. Let $Z_{\overline{x}}$ be  the fiber of $Z\to \overline{X}$ over $\overline{x}$. Moreover, assume that $Z_{\overline{x}}$ is smooth and that it satisfies the conditions of K\"unneth of formula \cite[Theorem 5.6.1]{CG}. Then, we have 
\begin{equation}
\label{kunneth}
K^{G}(Y\times_{\overline{X}} Z)\iso K^{G}(Y\times Z_{\overline{x}})\iso K^{G}(Y)\otimes_{\mathrm{R}(G)}K^{G}(Z_{\overline{x}}).
\end{equation}
Note that $K^{G}(Z_{\overline{x}})$ is naturally a $K^{G}(Z\times_{\overline{X}} Z)$-module and this is action is $\mathrm{R}(G)$-linear. The action of $K^{G}(Z\times_{\overline{X}} Z)$ on $K^{G}(Y\times_{\overline{X}}Z)$ is $\mathrm{R}(G)$-linear as well. One checks that the action of $K^{G}(Z\times_{\overline{X}} Z)$ on the right hand side of $\eqref{kunneth}$ comes by functoriality from the corresponding action on $K^{G}(Z_{\overline{x}})$. 
\subsection{Generalities on group actions and stacks}
\label{C}
 Let $G$ and $H$ be two algebraic groups, $\phi:G\to H$ be a morphism of groups and let $X$ be a $G$-variety. The induced $H$-variety  $H\times_{G} X$  with respect to $\phi$  is the stack quotient $(H\times X)/G$, where $G$ acts on $H\times X$ by 
$$g.(h,x)=(h\phi(g)^{-1},g.x).$$  
Let us show that  that $(H\times_{G} X)/H$ and $X/G$ are isomorphic as stacks.  The space $(H\times_{G}X)$ can be represented by the groupo\"id
$$\def\dar[#1]{\ar@<2pt>[#1]\ar@<-2pt>[#1]}
  \xymatrix{H\times X & G\times H \times X \dar[l]^{t}_{s}}$$
  where $s(g,h,x)=(h,x)$ and $t(g,h,x)=(h\phi(g)^{-1}, gx)$. The $H$-action on the objects and morphisms of this groupid is given by 
  $$h^{'}.(h,x)=(h^{'}h,x)$$
  $$h^{'}.(g,h,x)=(g,h^{'}h,x)$$
  Thus $(H\times_{G}X)/H$ is represented by the groupo\"id $\mathcal{G}$ given by 
  $$\def\dar[#1]{\ar@<2pt>[#1]\ar@<-2pt>[#1]}
  \xymatrix{H\times X & H\times G\times H \times X \dar[l]^{t^{'}}_{s^{'}}}$$
  where $s^{'}(h_{1},g,h_{2},x)=(h_{2},x)$ and $t^{'}(h_{1},g,h_{2},x)=(h_{1}h_{2}\phi(g)^{-1},g.x)$. It is then easy to check that the natural morphism from $\mathcal{G}$ to the action groupo\"id 
  $$\def\dar[#1]{\ar@<2pt>[#1]\ar@<-2pt>[#1]}
  \xymatrix{X & G\times X \dar[l]}$$
  is an equivalence.
  \par\medskip
This kind of arguments will be used repeatedly. To avoid writing down the stack morphisms we will deal with induced varieties as if they were ordinary schemes.   
\par\medskip
Let $\phi:G\to H$ be morphism of groups (as before) and let $X^{'}$ and $Y'$ be two G-varieties with $Y^{'}$ being smooth.  Let $\pi': Y'\to X'$ be a proper morphism of $G$-varieties. Let $X\,=\, H\times_G X'$ and $Y\,=\, H\times_G Y'$, $X$ and $Y$ are $H$-stacks. Then we have the following isomorphism 
$$Y\times_X Y\,\iso\, H\times_G(Y'\times_{X'} Y')$$ 
as $H$-varieties. So, we have an isomorphism of stack quotients 
$$(Y\times_X Y)/H\,\iso\, (Y'\times_{X'} Y')/G,$$ 
and we get an isomorphism of algebras 
$$K^H(Y\times_X Y)\,\iso\, K^G(Y'\times_{X'} Y').$$
\par\medskip
%

\par\medskip

 Let $Y_{1}$ be a $G$-scheme, and  $Y,$ $\tilde{Y}$ be two $H$-schemes. Consider the Cartesian diagram

 \[
\xymatrix @R=1cm{
Y_{1}\times_{Y} \tilde{Y} \ar[r] \ar[d]& Y_{1}\ar[d] \\
\tilde{Y}  \ar[r] & Y, 
}
\]
where  the map $\tilde{Y}\to Y$ is $H$-equivariant and the map $f:Y_{1}\to Y $ is $G$-equivariant, the action of $G$ on $Y$ being induced by morphism $\phi$. The group $G$ acts diagonally on the fiber product $Y_{1}\times_{Y} \tilde{Y}.$ This allows us to consider the induced space $H\times_{G} (Y_{1}\times_{Y} \tilde{Y}).$
On the other hand, we have a $H$-equivariant map $f_{1}:H\times_{G}Y_{1}\to Y$ given by $f_{1}(h,y_{1})=hf(y_{1}).$
Consider the cartesian diagram 
 \[
\xymatrix @R=1cm{
(H\times_{G}Y_{1})\times_{Y}\tilde{Y} \ar[r] \ar[d]& H\times_{G}Y_{1}\ar[d]^{f_{1}} \\
\tilde{Y} \ar[r] & Y, 
}
\]
and let $H$ act diagonally on the fiber product $(H\times_{G}Y_{1})\times_{Y}\tilde{Y}.$

\begin{lemma}
\label{gg}
There is a $H$-equivariant isomorphism of stacks
\begin{equation}
\label{ghequiv}
H\times_{G} (Y_{1}\times_{Y} \tilde{Y})\iso (H\times_{G} Y_{1})\times_{Y} \tilde{Y}.
\end{equation}
\end{lemma}
\begin{proof}
The isomorphism is furnished by the $H$-equivariant map  
\begin{align}
H\times_{G} (Y_{1}\times_{Y} \tilde{Y})\to & (H\times_{G} Y_{1})\times_{Y} \tilde{Y}\nonumber\\
(h,(y_{1},u)) \to & ((h,y_{1}),hu)\nonumber
\end{align}
 For $g$ in $G$, this map is given by
 $$(h\phi(g),(g^{-1}y_{1},\phi(g)^{-1}u))\to ((hg,g^{-1}y_{1}),hu).$$
It is $H$-equivariant isomorphism and yields the desired isomorphism $\eqref{ghequiv}$. 
\end{proof}

\end{document}